\documentclass[11pt,reqno]{amsart}
\usepackage{amsmath}
\usepackage{amsthm}
\usepackage{graphicx}
\usepackage{amsfonts}
\usepackage{amssymb}
\usepackage{hyperref}
\usepackage{bm}
\usepackage{color}

\usepackage[margin=1.1in]{geometry}
\usepackage{enumerate}
\usepackage{mathrsfs}

\usepackage[numbers,square]{natbib}
\usepackage{mathtools}

\numberwithin{equation}{section}
\newtheorem{theorem}{Theorem}[section]
\newtheorem{lemma}[theorem]{Lemma}
\newtheorem{proposition}[theorem]{Proposition}
\newtheorem{corollary}[theorem]{Corollary}
\newtheorem*{assumption*}{\assumptionnumber}
\providecommand{\assumptionnumber}{}
\makeatletter
\newenvironment{assumption}[1]
 {%
  \renewcommand{\assumptionnumber}{Assumption #1}%
  \begin{assumption*}%
  \protected@edef\@currentlabel{\textbf{\textup {#1}}}%
 }
 {%
  \end{assumption*}
 }

\theoremstyle{definition}
\newtheorem{example}[theorem]{Example}
\newtheorem{definition}[theorem]{Definition}
\newtheorem{remark}[theorem]{Remark}

\def\E{{\mathbb E}}
\def\R{{\mathbb R}}
\def\N{{\mathbb N}}
\def\P{{\mathbb P}}

\def\L{{\mathcal L}}

\newcommand{\green}[1]{\textcolor{green}}
\makeindex
\def\Law{\mathcal{L}}
\def\tr{\mathrm{Tr}}

\newcommand{\argmin}{\operatornamewithlimits{arg\,min}}

\newcommand{\wgrad}{\nabla_{\W}}

\newcommand{\W}{\mathbb W}

\newcommand{\ent}{ \mathbb H}
\newcommand{\fin}{\mathbb I}

\begin{document}

\title[Sharp chaos for mean field Langevin, control, and games]{Sharp propagation of chaos for mean field Langevin dynamics, control, and games}

\begin{abstract}
We establish the sharp rate of propagation of chaos for McKean-Vlasov equations with coefficients that are non-linear in the measure argument, i.e., not necessarily given by pairwise interactions. Results are given both on bounded time horizon and uniform in time. As applications, we deduce the sharp rate of propagation of chaos for the convergence problem in mean field games and control, and for mean field Langevin dynamics, the latter being uniform in time in the strongly displacement convex regime. Our arguments combine the BBGKY hierarchy with techniques from the literature on weak propagation of chaos.
\end{abstract}

\author{Manuel Arnese and Daniel Lacker}

\maketitle

\section{Introduction}
We study a system of particles 
\begin{equation}
\label{eq: particle SDE}
    \begin{split}
        &dY_t^i= V(m_t^n,Y_t^i)\,dt+\sqrt2 \sigma \, dB_t^i, \qquad i=1,\dots,n, \\
        &Y_0^i \sim \text{ iid }  \mu_0 , \qquad m_t^n= \frac 1n\sum_{i=1}^n \delta_{Y_t^i},
    \end{split}
\end{equation}
that interact through the empirical measure $m_t^n$. Each particle $Y_t^i$ takes values in $\R^d$, and $B_t^1,\dots,B_t^n$ are independent Wiener processes. It is moreover important for our results that $\sigma$ is non-zero.  Systems of this type have numerous applications in physics, economics, and machine learning, to name but a few, and we refer to \cite{ChaintronChaos1,ChaintronChaos2} for a thorough recent survey.
The large-$n$ limit of the system is typically expected to be described by the following McKean-Vlasov equation: 
\begin{equation}    \label{eq: main SDE}
dX_t=  V(\mu_t,X_t)\,dt+\sqrt 2 \sigma \,dB_t,  \ \ \Law(X_0)= \mu_0, \ \ \mu_t=\mathrm{Law}(X_t).
\end{equation}
There is by now a large body of work studying the convergence of \eqref{eq: particle SDE} to \eqref{eq: main SDE}, for a wide variety of specializations of $ V$. There are two equivalent ways to formulate the large-$n$ limit:
\begin{enumerate}
\item Mean-field limit: As $n\to\infty$, $m^n_t$ converges to $\mu_t$ (weakly in probability).
\item Propagation of chaos: For each fixed $k$, as $n\to\infty$, the joint law of $(Y_t^1,\dots,Y_t^k)$ converges to the product measure $\mu_t^{\otimes k}$.
\end{enumerate}
Although it is classical that these two formulations are equivalent \emph{qualitatively}, they turn out to be quite different \emph{quantitatively}, and it remains unclear how to transfer quantitative global bounds into local bounds in a sharp way, or vice versa.
Moreover, these two formulations are somewhat different in spirit: The first (1) is a \emph{global} property that involves  the entire configuration $(Y^1,\dots,Y^n)$, while the second (2) is more \emph{local} as it involves low-dimensional marginals.

The majority of the literature focuses on systems with \emph{pairwise interactions}, where $ V$ takes the form
\begin{equation}
 V(\mu,x)=\int \phi(x,y)\,d\mu(y), \qquad \phi : \R^d \times \R^d \to \R^d \label{intro:pairwise}
\end{equation}
Pairwise interactions are especially prominent in applications in physics, where the relevant interaction forces in many-body systems (e.g., electrostatic or gravitational) naturally act pairwise.
Major challenges arise when the interaction function $\phi$ is singular, and for many such models it remains an active topic of research to justify (1) or (2) even just qualitatively \cite{JabinWangReview}.
 
In this paper we focus on non-pairwise interactions, where $ V$ is a fairly general nonlinear functional of its measure argument, assumed smooth in a sense we make precise later. This setting arises naturally in the contexts of mean field Langevin dynamics, mean field games, and mean field control, all of which will be discussed in more detail below.
Our main goal is to prove sharp estimates of propagation of chaos.

There are two main quantitative techniques used in prior literature for general interactions, which are both global in nature.
First is the synchronous coupling approach, popularized by Sznitman \cite{snitzmann}. If $ V(\mu,x)$ is jointly Lipschitz, with the 2-Wasserstein distance $\W_2$ used for $\mu$ (though any $p\neq 2$ works analogously), then one can prove a global bound of
\[
\W_2^2(\pi^n_t , \mu^{\otimes n}_t) = O(1),
\]
where $\pi^n_t$ is the joint law of $(Y^1_t,\ldots,Y^n_t)$. By a subadditivity inequality, this implies
\begin{equation}
\W_2^2(\pi^k_t , \mu^{\otimes k}_t) = O(k/n), \quad k \le n, \label{intro:W2-subadditive}
\end{equation}
where $\pi^k_t$ is the joint law of $(Y^1_t,\ldots,Y^k_t)$.
The second and more recent approach, known as \emph{weak propagation of chaos} \cite{Chassagneux2022-jt,UiTWeakChaos}, is based on the analysis of the Kolmogorov backward equation on the space of measures associated with the dynamics of $(\mu_t)_{t \ge 0}$. For appropriately smooth functions $f:\mathcal P(\R^d)\to \R$ this yields the sharp estimate
\begin{equation}
|\E[f(m_t^n)-f(\mu_t)]|= O(1/n). \label{intro:weakPoC}
\end{equation}
In fact, these techniques can yield precise expansions of $\E f(m_t^n)$ around $f(\mu_t)$ in powers of $1/n$, with coefficients determined by derivatives of $f$.

The second author recently introduced in \cite{HierarchiesPaper} a purely \emph{local} method for studying propagation of chaos, based on the BBGKY hierarchy.
For pairwise interactions \eqref{intro:pairwise}, with $\phi$ being Lipschitz or bounded, it was shown that
\begin{equation}
\ent(\pi_t^k\,\|\,\mu_t^{\otimes k})= O(k^2/n^2), \label{intro:localbound}
\end{equation}
where $\ent$ is relative entropy.  When $\phi$ is Lipschitz, applying a Talagrand inequality leads to a notable improvement of \eqref{intro:W2-subadditive} to $O(k^2/n^2)$. On the other hand, \eqref{intro:weakPoC} and \eqref{intro:localbound} are hard to compare directly, except that the latter easily implies the former when $f$ is a regular enough U-statistic.

In this paper we will prove \eqref{intro:localbound} for a broad class of general (non-pairwise) interactions $ V$, both on bounded time horizons and uniform in time, the latter when $V$ is displacement monotone (see Assumption \ref{Assumption: UiT}). A first step toward this appeared in \cite[Section 2D]{HierarchiesPaper}, but with a slower rate and with unnatural and difficult-to-check smoothness assumptions on $ V$.

Our proof builds on the BBGKY approach of \cite{HierarchiesPaper}. In the pairwise interaction case, the time evolution of the marginal laws $(\pi_t^k)_{k=1}^n$ follows the BBGKY hierarchy; in particular, the evolution of the $k$-th marginal $\pi_t^k$ depends only on the conditional law of $Y_t^{k+1}$ given $(Y_t^1,\dots,Y_t^k)$. The approach of \cite{HierarchiesPaper} was to analyze the dynamics of the relative entropy at each level of the hierarchy. This led to a system of differential inequalities,
\begin{equation}
    \label{eq: Hierarchy pairwise case}
    \frac{d}{dt}\ent(\pi^k_t\,\|\,\mu^{\otimes k}_t)\lesssim \frac{k^3}{n^2}+k\Big(\ent(\pi^{k+1}_t\,\|\,\mu^{\otimes (k+1)}_t)-\ent(\pi^k_t\,\|\,\mu^{\otimes k}_t)\Big), \quad k < n,
\end{equation}
which can then be used to derive the desired estimate \eqref{intro:localbound}.
Subsequent work extended this method and result to be uniform in time \cite{SharpChaosLacker2023,TimeUniformLSI}, to obtain higher-order (cluster) expansions of $\pi^k_t$ around $\mu^{\otimes k}_t$ \cite{Hess-Childs2025-ux}, to non-exchangeable interactions \cite{lacker2024quantitativepropagationchaosnonexchangeable}, to hold in Fisher information \cite{grass2025propagation}, to interactions through the diffusion coefficient \cite{sharpChaosDiffusion}, and even to certain singular interactions \cite{wang2024sharplocalpropagationchaos}, but always in the setting of pairwise interactions.

To handle non-pairwise interactions, we adapt the BBGKY approach by performing a Taylor expansion of $ V(m^n_t,\cdot)$ around $ V(\mu_t,\cdot)$ at an appropriate moment. The first-order term produced takes the form of a pairwise interaction which can be analyzed using prior methods. This ultimately leads to a system of differential inequalities
\begin{equation}
\label{eq: differential inequality}
    \frac{d}{dt}\ent(\pi^k_t\,\|\,\mu^{\otimes k}_t)\lesssim \frac{k^3}{n^2}+k\Big(\ent(\pi^{k+1}_t\,\|\,\mu^{\otimes (k+1)}_t)-\ent(\pi^k_t\,\|\,\mu^{\otimes k}_t)\Big)+\E [R_t(m^n_t)],
\end{equation}
for some non-linear ``remainder'' term $R$. The problem then reduces to understanding the rate at which this  remainder vanishes. To do this, we use ideas from the literature on weak propagation of chaos mentioned above, in particular using tools from \cite{Chassagneux2022-jt} and \cite{Buckdahn2017-op}. However, the off-the-shelf estimate in \eqref{intro:weakPoC} yields only $1/n$ instead of our desired rate of $1/n^2$, and we must instead carefully exploit the fact that the remainder $R_t(\cdot)$ vanishes at first order at $\mu_t$.

It is worth stressing that some smoothness of $ V(\mu,x)$ is necessary. Mere Wasserstein-Lipschitz continuity suffices for the non-sharp bound \eqref{intro:localbound} but  not  for the sharp bound \eqref{intro:localbound}. This fact, explained in Example \ref{example}, lies in contrast with the pairwise interaction case \eqref{intro:pairwise}, where Lipschitz $\phi$ was good enough for the sharp bound \eqref{intro:localbound}.

For simplicity, we focus on the case of iid initializations, but see Remark \ref{rk: beyond iid} for some discussion of how to go beyond this case. 
It should be possible to obtain \eqref{intro:localbound} without assuming iid initial conditions, but rather just assuming  $\ent(\pi^k_0\,\|\,\mu^{\otimes k}_0)=O(k^2/n^2)$.
Indeed, this was shown in prior work in the case of pairwise interactions, starting in \cite{SharpChaosLacker2023}. Many kinds of mean field Gibbs measures $\pi^n_0$ are known to satisfy this condition for suitable $\mu_0$, as was shown in \cite{lacker2022quantitative,ren2024sizechaosgibbsmeasures}.

\section{Results} \label{se:mainresults}

\subsection{Analysis on the space of measures}

For a random variable $X$ we write $\Law(X)$ for its law.
For a separable metric space $(E,d)$ we write $\mathcal{P}(E)$ for the set of Borel probability measures on $E$. For $p \ge 1$ we let $\mathcal{P}_p(E)$ denote the subset of $m \in \mathcal{P}(E)$ for which $\int d^p(x_0,\cdot)\,dm < \infty$ for some $x_0 \in E$. Define the Wasserstein distance $\W_p$ as usual by
\[
\W_p^p(\mu,\nu) = \inf_\pi \int_{E \times E} d^p(x,y)\,\pi(dx,dy),
\]
where the infimum is over couplings of $\mu$ and $\nu$. The relative entropy is defined by 
\[
\ent(\nu\,\|\,\mu) = \begin{cases} \int_E \frac{d\nu}{d\mu}\log \frac{d\nu}{d\mu}\, d\mu &\text{if } \nu \ll \mu \\ \infty &\text{otherwise}. \end{cases}
\]

We will make heavy use of  differentiability concepts on the space of probability measures; for an introduction, see \cite{Ambrosio2008-dt} (for an optimal transport perspective) or \cite[Chapter 5]{Carmonabook1} (for a more probabilistic approach). Consider a function $f:\mathcal P_2(\R^d)\to \R$. A \emph{flat derivative} of $f$ is a function $\delta_m f:\mathcal P_2(\R^d)\times \R^d\to \R $ that satisfies, for every $\nu,\eta \in \mathcal P_2(\R^d)$
\begin{equation*}
\int \delta_m f(\nu,y)\,d(\nu-\eta)(y)=\lim_{h \to 0}\frac{f((1-h)\nu+h\eta)-f(\nu) }{h}, \quad \delta_m f(\nu,\cdot) \in L^2(\nu) \cap L^2(\eta).
\end{equation*}
If it exists, the function $x\mapsto \delta_m f(\nu,x)$ is unique up to an additive factor that is allowed to depend on $\nu$; this is because the constant gets erased by integrating against a zero mass measure. For concreteness, for each $\nu$ we choose the unique normalization that guarantees that $\delta_m f(\nu,0)=0$.

Notice that flat derivatives acquire an extra space argument compared to the original function $f$. We can iterate the operation of taking a flat derivative by keeping the extra space argument fixed, obtaining the $k$-flat derivative $\delta_m^k f: \mathcal P_2(\R^d)\times (\R^d)^k\to \R$.
The Wasserstein (or Lions) derivative $\wgrad f : \mathcal P_2(\R^d) \times \R^d \to \R^d$ is defined as 
\[\wgrad f(\nu,y)=\nabla_y \delta_m f(\nu,y).\]
To obtain the $k$-Wasserstein derivative, we apply this operation $k$ times, resulting in a function $\wgrad^k f : \mathcal P_2(\R^d) \times (\R^d)^k \to (\R^d)^{\otimes k}$,
\[\wgrad^k f(\nu,\boldsymbol{y})=\nabla_{y_k}\delta_m\dots \nabla_{y_1} \delta_m f(\nu,y_1,\dots,y_k).\]
Under standard regularity assumptions (e.g.,  \cite[Theorem 2.6]{Chassagneux2022-jt} or \cite[Lemma 5.4]{Buckdahn2017-op}), space and flat derivatives commute, and so we might also apply $\delta_m^k$ and then $k$ gradient operations, one in each variable, to obtain the same function $\wgrad^k f$.

\begin{definition} \label{def:multiindex}
Pick two non negative integers $\mathsf{w},b\in \N_0$ and a vector of $\mathsf{w}$ non negative integers $\boldsymbol{c}\in \N_0^{\mathsf w}$. We say that $(\mathsf{w},\boldsymbol{i})=(\mathsf{w},b,\boldsymbol{c})$ is a multi-index of size $\mathsf{w}+|\boldsymbol{i}|:=\mathsf{w}+b+\sum_{j=1}^\mathsf{w} c_j$. For a function $f=f(\nu,x)$ on $\mathcal P_2(\R^d)\times \R^d$ taking values in a Euclidean space, we denote
    \begin{equation*} D^{\mathsf{w},\boldsymbol{i}}f(\nu,x,\boldsymbol{y})=\nabla_{y_1}^{c_1}\dots \nabla_{y_{\mathsf{w}}}^{c_{\mathsf{w}}}\wgrad^{\mathsf{w}}\nabla_x^b f(\nu,x,y_1,\dots y_\mathsf{w}).
    \end{equation*}
\end{definition}
Let us explain this piece of notation in words: the first component of a multi index is a number $\mathsf{w} \in \N_0$ that denotes how many Wasserstein derivatives in $\nu$ we are taking, while the second counts the number of derivatives in the spatial argument $x$. The third element is a vector $\boldsymbol{c}\in \N^\mathsf{w}_0$ that counts the number of derivatives taken in each of the new space arguments that arise due to the $\mathsf w$ Wasserstein gradients. Notice that for a fixed $\nu$, $D^{\mathsf{w},\boldsymbol i}f(\nu,\cdot):\R^{d(\mathsf{w}+1)}\to (\R^{d})^{\mathsf{w}+|\boldsymbol i|}$, and we equip $(\R^{d})^{\mathsf{w}+|\boldsymbol i|}$ with the Euclidean norm.

\begin{definition}
We denote by $\mathcal C^k_{\textup{bd}}(\mathcal P_2(\R^d)\times \R^d)$ the set of functions $f:\mathcal P_2(\R^d)\times \R^d\to \R^d$ such that for every multi-index $(\mathsf w,\boldsymbol{i})=(\mathsf w, b,\boldsymbol c)$, with $0<\mathsf w+|\boldsymbol{i}|\leq k$, the derivative $D^{\boldsymbol{\mathsf w,i}}f$ exists and is bounded and Lipschitz with respect to the $\W_2$ and Euclidean metrics. We similarly write $C^k_{\textup{bd}}(\mathcal P_2(\R^d) )$ for the subset of $C^k_{\textup{bd}}(\mathcal P_2(\R^d)\times \R^d)$ consisting of those functions which depend only on the measure argument, and the notation $D^{\mathsf w,\boldsymbol{i}}f$ for such a function will be defined only for a multi-index of the form $(\mathsf{w},\boldsymbol{i})=(\mathsf{w},0,c)$, i.e., with $b=0$.
\end{definition}
We stress that the strict inequality $0 < \mathsf w+|\boldsymbol{i}|$ means that $\mathcal C^k_{bd}(\mathcal P_2(\R^d)\times \R^d)$ contains unbounded functions (of linear growth). 
Examples include functions of the form $f(\nu,x)=g(x,\int h(x,y)\,d\nu(y))$, where $g:\R^d\times \R \to \R$ and $h:\R^d \times \R^d \to \R$ each have $k$ bounded space derivatives.

\subsection{Assumptions and main results}

We work with two sets of assumptions. The first is for bounded time horizon, and the second is for uniform-in-time results.

\begin{assumption}{($A$)} { \ }
\begin{enumerate}
\item The function $V$ belongs to $\mathcal C_{\textup{bd}}^6(\mathcal P_2(\R^d)\times \R^d)$. 
\item The initial condition $\mu_0$ satisfies a T$_1$ transport inequality: there exists a constant $C_{\textup{T}_1}(0)> 0$ such that for every $\nu \in \mathcal P(\R^d)$,
\[\W_1^2(\mu_0,\nu)\leq 2C_{\textup{T}_1}(0) \,\ent (\nu\,\|\,\mu_0). \]
\item The noise is non-degenerate: $\sigma>0$.
\end{enumerate}
\label{assumption on Phi}
\end{assumption}

Let us note that Assumption \ref{assumption on Phi}(2) is equivalent to $\mu_0$ being $C_{\textup{T}_1}(0)$-subgaussian, in the sense that
\begin{equation}
\log \int_{\R^d} e^{c f(x)}d\mu(x) \le \lambda \int_{\R^d} f\,d\mu + C_{\textup{T}_1}(0)\frac{c^2}{2}, \label{ineq:mu_0subgaussian}
\end{equation}
for all $c \in \R$ and all 1-Lipschitz functions $f$. This equivalence is due to  Bobkov-G\"otze \cite[Theorem 4.8]{VanHandelHDP}.

Under Assumption \ref{assumption on Phi}, the stochastic differential equations \eqref{eq: particle SDE} and \eqref{eq: main SDE} both admit pathwise unique strong solutions. Let us write $\pi^n \in \mathcal P(C(\R_+;(\R^{d})^{n}))$ for the law of the process $(Y^1,\ldots,Y^n)$ and $\pi^k$ for the marginal of $(Y^1,\ldots,Y^k)$, for $1 \le k \le n$. 
In what follows, for $t \ge 0$ and $\nu \in \mathcal P( C(\R_+;\R^m))$, we write $\nu[t]$ for its restriction to the time interval $[0,t]$, i.e., the pushforward by the restriction map $C(\R_+;\R^m) \ni x\mapsto x|_{[0,t]} \in C([0,t];\R^m)$.
We also write $\nu_t$ for the time-$t$ marginal, i.e., the pushforward by $C(\R_+;\R^m) \ni x\mapsto x_t \in \R^m$.

Our first main result is the following propagation of chaos bound on a finite time horizon.

\begin{theorem}
\label{th: short time}
    Under assumption \ref{assumption on Phi}, for every $k\leq n$ and every fixed $t>0$, we have
    \begin{equation}
        \ent(\pi^k[t]\,\|\,\mu^{\otimes k}[t]) = O(k^2/n^2).
    \end{equation}
 The hidden constant depends on $t$, $C_{\textup{T}_1}(0)$, $\mu_0$, and the derivative bounds on $ V$.
\end{theorem}
This rate cannot be improved under the present assumptions, as explained by the Gaussian example in \cite{HierarchiesPaper}. 
Theorem \ref{th: short time} implies bounds in total variation by Pinsker's inequality, and a coupling argument yields control in Wasserstein distance. 
\begin{corollary}
\label{cr: metric convergences short time}
    Under assumption $\ref{assumption on Phi}$, for every $k\leq n $ and every fixed $t>0$,
    \[\|\pi^k[t]-\mu^{\otimes k}[t]\|^2_{\textup{TV}}+\W_{2}^2(\pi^k[t],\mu^{\otimes k}[t]) =O(k^2/n^2)\]    
\end{corollary}
To be clear, in Corollary \ref{cr: metric convergences short time}, $\pi^k[t]$ and $\mu^{\otimes k}[t]$ are measures on the path space $C([0,t];\R^d)^k \cong C([0,t];\R^{dk})$ which is equipped with the usual sup-norm, and the Wasserstein distance is defined accordingly.

\begin{remark}
When the particles take values in dimension  $d=1$, Assumption \ref{assumption on Phi}(1) can be relaxed to require only two bounded Wasserstein derivatives, instead of six in all variables, and the proof is much simpler, particularly in the analysis of the remainder term in \eqref{eq: differential inequality}. See Remark \ref{rem:d=1} for details.
\end{remark}

\begin{remark}
This bound on $\W_{2}^2$ does not invoke a $T_2$ (Talagrand) inequality, as one might guess. The path measure $\mu^{\otimes k}[t]$ need not satisfy a (dimension-free) $T_2$ inequality because we assumed only that $\mu_0$ satisfies the significantly weaker $T_1$ inequality.
\end{remark}

\begin{remark} \label{rk: time dependent setting-intro}
We have stated our result in a time homogeneous setting, but the same applies without change to the case that $ V$ is a time dependent function; we explain how to adapt the proof in Remark \ref{rk: time dependent setting}.    
\end{remark}

In the above results, the constants hidden by the big-O notation grow exponentially in $t$. Under stronger assumption, we obtain similar sharp bounds uniformly in time, at the level of the time-marginals instead of path measures. The extra requirements are displacement monotonicity of the drift (in the sense of optimal transport) and a smallness condition of the interaction.

\begin{assumption}{(UiT)} { \ }
\begin{enumerate}
\item Assumption \ref{assumption on Phi} holds.
\item Monotonicity: There exists $\lambda>0$ such that for any square-integrable $\R^d$-valued random variables $X$ and $Y$ we have
\begin{equation*}
    \E\Big[\big( V(\Law(X),X)- V(\Law(Y),Y)\big)\cdot(X-Y)\Big]\leq -\lambda \E\Big[|X-Y|^2\Big].
\end{equation*}
\item Small interaction: $ \||\wgrad V|^2_{\text{op}}\|_\infty<\lambda^2/3$ (where for a $d\times d$ matrix $M$, $|M|_{\text{op}}$ is the operator norm induced by the Euclidean distance). 
\end{enumerate}
\label{Assumption: UiT}
\end{assumption}

\begin{theorem}
\label{th: UiT}
    Under assumption \ref{Assumption: UiT}, for every $k\leq n$,
    \begin{align*}
        \sup_{t\geq 0} \ent(\pi_t^k\,\|\,\mu_t^{\otimes k})=O(k^2/n^2).
    \end{align*}
     The hidden constant depends on $C_{\textup{T}_1}(0)$, $\mu_0$, and the derivative bounds on $ V$.
\end{theorem}
Bounds in entropy readily translate again to uniform in time estimates for the total variation and Wasserstein metrics.
\begin{corollary}
\label{cr: metric convergence UiT}
Under assumption \ref{Assumption: UiT}, for every $k \le n$,
\begin{align*}
    \sup_{t\geq 0}\|\pi_t^{k}-\mu_t^{\otimes k}\|^2_{L^1(\R^d)}+\sup_{t\geq 0} \W_2^2(\pi_t^k,\mu_t^{\otimes k})= O(k^2/n^2).
\end{align*}     
\end{corollary}

\subsection{Comments on the assumptions}

The assumption that $V$ has bounded Wasserstein and space gradient is important; in the pairwise case $ V(\mu,x)=\int \phi(x,y)\,d\mu(y)$ this means that $\phi$ is Lipschitz, the original assumption of \cite{HierarchiesPaper}. The boundedness of higher order derivatives is not essential and could most likely be relaxed with an appropriate approximation procedure. While the existence of $6$ derivatives is a demanding assumption, it is satisfied in many applications of interest; moreover, in low dimension fewer than $6$ derivatives are needed.
The other significant assumption is that $\mu_0$ concentrates well enough to satisfy a T$_1$ inequality, a useful property that leads the particles $Y_t^i$ to also concentrate well.

We do not think that the requirement for existence of $6$ derivatives is sharp, but a simple example shows that having at least one derivative is necessary: Lipschitzianity is not enough to guarantee dimension-free convergence rates.
\begin{example}
\label{example}
    Let $\gamma$ be the law of a $d$-dimensional Brownian motion, so that $\gamma_t$ is a Gaussian measure on $\R^d$ with mean 0 and variance $tI$. Define the function $ V_t(\nu)=\W_1(\nu,\gamma_t)$ and consider the particle system 
    \begin{align*}
        dY_t^i &=  V_t(m_t^n)\cdot\mathbf 1\,dt+dB_t^i\\
        Y_0^i&=0\,\,,\,\, i=1,\dots,n.
    \end{align*}
    where $\mathbf 1\in \R^d$ is a vector of all ones. The function $ V_t$ is Lipschitz (in the measure argument, and trivially in space), but does not have Wasserstein derivatives. The corresponding McKean Vlasov process $X_t$ is simply a Brownian motion: indeed Brownian motion solves the SDE, and the solution is unique. Then, using the representation of $\pi^1[T]$ given by Lemma \ref{lm: projection path space} along with a standard Girsanov calculation of relative entropy, 
    \begin{align*}
\ent(\pi^1[T]\,\|\,\mu[T]) =\ent(\pi^1[T]\,\|\,\gamma[T]) &=\frac d2\int_0^T \E\Big[\E\big[\W_1(m_t^n,\gamma_t)\mid Y^1[t]\big]^2\Big]\,dt\\
&\geq \frac d2 \inf_{q} \W_1^2(q,\gamma_t),
    \end{align*}
    where the infimum is taken over all atomic probability measures $q$ with $n$ atoms. A well-known quantization lower bound  (see, e.g., \cite[Theorem 2.1]{Quantization})  yields
    \[\lim_{n\to \infty} n^{2/d}\inf_{|Q|=n} \W_1^2(Q,\gamma_t) \geq C_d>0\]
    and thus it is impossible to obtain the good $1/n^2$ rate.
\end{example}

Assumption \ref{Assumption: UiT}(2) is natural from a couple of perspectives. In the case where $V=-\wgrad\Psi$ for some smooth functional $\Psi$ on Wasserstein space, it is equivalent to the $\lambda$-displacement convexity of $\Psi$ in the sense of optimal transport \cite[Lemma 3.6]{GangboConvexity}. In the pairwise case $V(\mu,x)= -\nabla U(x) - \int  \nabla W(x-y)\,d\mu(y)$ this holds when $U$ is strongly convex and $W$ is convex, which is a standard regime for uniform-in-time propagation of chaos ever since the work \cite{Malrieu2001-va}. Assumption \ref{Assumption: UiT}(3) is a smallness condition, constraining the strength of the interaction relative to the strength of monotonicity $\lambda$. Restrictions of this kind are common in the literature on uniform-in-time propagation of chaos \cite{elementaryUITchaos,UiTWeakChaos,SharpChaosLacker2023}, particularly for going beyond the aforementioned convex regime. It seems unlikely assumption \ref{Assumption: UiT}(3) is necessary for the sharp rate $(k/n)^2$, given that we are imposing the convexity assumption \ref{Assumption: UiT}(2), but it remains an open question if the sharp rate can be obtained without such a smallness condition even in the pairwise case treated in \cite{SharpChaosLacker2023}. This was recently achieved in the stationary (equilibrium) regime in \cite{ren2024sizechaosgibbsmeasures}.

Sharpening previous work on sharp uniform-in-time propagation of chaos \cite{SharpChaosLacker2023,TimeUniformLSI}, we do not directly assume a log-Sobolev inequality (LSI) for the initial measure $\mu_0$ or its dynamical counterpart $\mu_t$. Instead, we assume merely the T$_1$ inequality, and we exploit the regularizing effects of the dynamics to prove that $\mu_t$ satisfies a LSI for large enough time, with constant uniformly bounded for $t \ge t_0$; this is an application of a recent result of \cite{Chen2021-ch}, which shows that the LSI constant ``comes down from infinity'' for subgaussian initial conditions. We then apply the short-time result from Theorem \ref{th: short time} to control the time interval $[0,t_0]$. 

\begin{remark}
Factor of $3$ in Assumption \ref{Assumption: UiT}(3) agrees with the smallness assumption of \cite[Theorem 2.1(1)]{SharpChaosLacker2023}, after correcting a minor error therein: In their proof of Theorem 2.1(1) on page 468 it is required that $\alpha > 3$, which means that $\sigma^4 > 12\gamma\eta$ in their notation is required for their proof of the sharp uniform in time rate $(k/n)^2$. To map their notation to ours, note that their $\sigma$ is our $\sqrt{2}\sigma$, their log-Sobolev constant $\eta$ is our $\sigma^2/2\lambda$, and their $\gamma$ is our $(2\sigma^2/\lambda)\||\wgrad V|^2_{\text{op}}\|_\infty$. Hence, their condition $\sigma^4 > 12\gamma\eta$ is equivalent to our $\||\wgrad V|^2_{\text{op}}\|_\infty < \lambda^2/3$.
\end{remark}

\subsection{Application: Mean Field Langevin Dynamics}
Motivated by models of scaling limits of neural networks and connections to optimization problems over the space of probability measures, the paper \cite{MeanFieldLangevin} introduced the mean field Langevin dynamics (MFLD)
\begin{equation} 
    \label{eq: MFLD}
    dX_t =-\wgrad \Psi(\Law(X_t),X_t)\,dt+\sqrt{2}\sigma dB_t, 
\end{equation}
where we recall that $\Law(X_t)$ denotes the law of $X_t$.
Equation \eqref{eq: MFLD} (or rather, the associated nonlinear Fokker-Planck PDE) can be interpreted as the Wasserstein gradient flow of the functional 
\begin{equation}
\label{eq: MFLD function}
    \mathcal P_2(\R^d)\ni\mu \mapsto \Psi(\mu)+\sigma^2 \int \mu\log\mu,
\end{equation} 
with the integral (differential entropy) interpreted as $+\infty$ if $\mu$ does not possess a sufficiently integrable density.
A key property of MFLD is that  $\mu_t$ converges to the minimizer of \eqref{eq: MFLD function} as $t\to\infty$, and does so very quickly under the correct assumptions (see for instance \cite{chizat2022meanfield}). This makes simulating MFLD a viable approach to sampling from the measure which minimizes \eqref{eq: MFLD function}. Two popular and distinct sets of assumptions are known to lead to fast convergence: the flat convexity regime, where $\Psi$ is taken to be convex in the usual sense, and the displacement convexity regime where the functional $\Psi$ is convex along displacement interpolations, i.e., geodesics in the Wasserstein space $(\mathcal{P}_2(\R^d),\W_2)$.

The most natural way to simulate MFLD is via the particle approximation
\begin{equation}
\label{eq: MFLD particles}
    dY_t^i=-\wgrad \Psi(m_t^n,Y_t^i)\,dt+\sqrt{2}\sigma dB_t^i
\end{equation}
The quality of this particle approximation to  \eqref{eq: MFLD} is precisely a question of quantitative propagation of chaos.
This question has been tackled by \cite{ChaosMFLD} and \cite{suzuki2023uniformintime} (see also \cite{ChaosKineticMFLD} for the related kinetic version). The article \cite{ChaosMFLD} proves a strong propagation of chaos result with global methods, deducing from subadditivity that
\begin{equation*}
    \sup_{t\geq 0}\ent (\pi_t^k \,\|\,\mu_t^{\otimes k}) = O \Big(e^{-at}+\frac kn \Big).
\end{equation*}
On the other hand \cite{suzuki2023uniformintime} proves a weak propagation of chaos result for a specific function $V$ with a strong (superquadratic) regularization: for every $f \in \mathcal C_{\text{bd}}^2(\mathcal P(\R^d)),$
\begin{equation*}
    \sup_{t\geq 0}\E[|f(m_t^n)-f(\mu_t)|]^2 = O(1/n).
\end{equation*}
The more recent paper \cite{nitanda2024improved} improved the dependence (hidden here) on the log-Sobolev constants. See also the recent work \cite{chewi2024uniform,wang2024uniform} with improved uniform-in-$n$ bounds on the LSI constant of the $n$-particle system, which lead quickly to uniform-in-time propagation of chaos when combined with finite-horizon bounds.
These results were all obtained in the flat convexity regime, which arises naturally in models of two-layer neural networks.

The natural  question of \emph{sharp} uniform-in-time propagation of chaos for MFLD has not yet been addressed.
An application of Theorem \ref{th: UiT} yields such a result in the displacement convex regime:

\begin{corollary}
    \label{co: Mean Field Langevin Dynamics}
    Take $\lambda>0$. Let the functional $\Psi$ be $\lambda$-convex, in the sense that for any square-integrable $\R^d$-valued random variable $X$ and $Y$ defined on any probability space, and any $t \in (0,1)$,
    \begin{equation}
    \Psi(\Law(t X + (1-t)Y)) \le t\Psi(\Law (X)) + (1-t)\Psi(\Law (Y)) - \frac12\lambda t(1-t)\E|X-Y|^2. \label{asmp:convexityUIT}
    \end{equation}
    Assume $V=-\wgrad \Psi$ exists and satisfies Assumption \ref{assumption on Phi} as well as $\|\|\wgrad^2 \Psi\|_{\text{op}}\|^2_\infty < \lambda^2/3$. Then
    \begin{align*}
        \sup_{t\geq 0}\ent (\pi_t^k\,\|\,\mu_t^{\otimes k})=O(k^2/n^2).
    \end{align*}
\end{corollary}
\begin{proof}
Apply Theorem \ref{th: UiT} after noting that the convexity assumption \eqref{asmp:convexityUIT} implies (in fact, is equivalent to) the monotonicity condition \ref{Assumption: UiT}(2) for $V=-\wgrad \Psi$; see Lemma \ref{lm: convexity implies monotonicity}.
\end{proof}

Corollary \ref{co: Mean Field Langevin Dynamics} imposes stronger assumptions than previous work on MFLD, namely higher order smoothness and the weak interaction condition, but it improves the rate of convergence to the sharp $k^2/n^2$. 

There are two natural followup problems which we do not address here and would likely require different methods. The first, as mentioned before, is removing the likely unnecessary smallness condition $\|\wgrad^2 \Psi\|^2_\infty < \lambda^2/3$. The second is to address the flat-convex instead of displacement convex regime, in light of its relevance to neural network models. We note that \eqref{asmp:convexityUIT} is equivalent to the usual notion of $\lambda$-displacement convexity from the theory of optimal transport, as explained in \cite[Lemma 3.6]{GangboConvexity}; it is also equivalent to the $(\lambda/n)$-convexity of the function $(\R^d)^n \ni (y^1,\ldots,y^n) \mapsto \Psi(\frac{1}{n}\sum_{i=1}^n \delta_{y^i})$.

\subsection{Application: Mean Field Games}

In this section we introduce the mean field game convergence problem and explain how Theorem \ref{th: short time} leads to a new sharp propagation of chaos estimate.
Fix a time horizon $T > 0$, and consider cost functions $f_0,g: \mathcal P(\R^d)\times \R^d \to \R$ and $f_1:\R^d \times \R^d \to \R$. Precise assumptions on these functions will be given in the theorem statements below. Consider the following game played by $i=1,\dots,n$ players: each player controls a private state process
\begin{equation*}
    dY_t^i = \alpha^i(t,Y_t)\,dt+\sqrt 2 \sigma dB_t^i, \qquad Y_0^i \sim \text{ iid }  \mu_0 ,
\end{equation*}
and wishes to minimize the cost functional
\begin{align*}
    J_n(\alpha^1,\ldots,\alpha^n)= \E\bigg[\int_0^T f_0( m_t^n,Y_t^{i})+f_1(Y_t^{i},\alpha^i(t,Y_t))\,dt+g( m_T^n,Y_T^{i}).\bigg]
\end{align*}
where $m_t^n$ is the empirical measure associated to $Y_t^1,\dots,Y_t^n$. Here the choice variable $\alpha_i$ is a Markovian (closed-loop) control, chosen to be a measurable function of $Y_t=(Y^1_t,\ldots,Y^n_t)$ of linear growth, so that the SDE for $Y$ has a unique in law solution. The associated  Hamiltonian $H$ is
\begin{equation*}
    H(x,y)=-\inf_{a \in \R^d}\{a\cdot y+f_1(x,a)\}.
\end{equation*}
We will always assume that $H$ is a continuous function with a derivative in the $y$ argument so that the following expressions make sense. A Nash equilibrium is a profile of controls $(\alpha^i)_{i =1}^n$ such that for each $i$,
\[
\alpha^i \in \argmin_{\beta^i} J_n( \beta^i, (\alpha^j)_{j\neq i}).
\]
Under some regularity assumptions (see for instance \cite[Proposition 6.26]{CDMFGbook2}) a Nash equilibrium exists and takes the form
\begin{equation}
    \label{eq: controls in a NE}\alpha^i(t,x)= -\nabla_y H(x_i,\nabla_{x_i} u^{n,i}(t,x)), \quad x=(x_1,\ldots,x_n) \in (\R^d)^n,
\end{equation}where $(u^{n,i})_{i=1}^n: [0,T] \times (\R^d)^n\to \R$ is the solution to the so-called Nash system of PDEs:
\begin{align}
\begin{split}
\label{eq: Nash system}
    \partial_t u_t^{n,i}(x)&   - H(x^i,\nabla_{x^i}u^{n,i}_t(x))  +\sigma^2\sum_{j=1}^n \Delta_{x^j} u_t^{n,i}(x) + f_0(m_x^n,x^i) \\
        &- \sum_{j\neq i}^n \nabla_y H(x^j,\nabla_{x^j}u^{n,j}_t(x))\cdot \nabla_{x^j}u_t^{n,i}(x)  =0, \\
    u_T^{n,i}(x)&=g(m_x^n,x^i), \qquad m^n_x := \frac{1}{n}\sum_{j=1}^n\delta_{x^j}.
\end{split}
\end{align}
The equilibrium trajectories follow the dynamics 
\begin{align} 
    \label{eq: trajectories in Nash Equilibrium}
    dY_t^i &= -\nabla_y H(Y_t^i,\nabla_{x_i} u^{n,i}(t,Y_t))\,dt+\sqrt 2 \sigma \,dB_t^i.   
\end{align}
Viewing $Y^i$ as a random element of $C([0,T];\R^d)$, let $\pi^k=\Law(Y^1,\dots,Y^k)$ for $1 \le k \le n$.

The large-$n$ behavior of the above $n$-player game is expected to be captured by the following mean field game, in which the large finite population is replaced by a single representative player interacting with a mean field.
The representative player controls the state process $dX_t=\alpha_t\,dt+\sqrt 2 \sigma dB_t$. Assuming the distribution of the other players' states at time $t$ is given by $\mu_t$, the representative player seeks to choose $\alpha$ to minimize the cost functional
\begin{align*}
    J(\alpha,\mu)= \E\bigg[\int_0^T f_0(\mu_t,X_t)+f_1(X_t,\alpha_t)\,dt+g(\mu_T, X_T)\bigg].
\end{align*}
A mean field equilibrium (MFE) is a couple $(\mu,\alpha)$ such that $\alpha \in \argmin J(\cdot,\mu)$ and $\mu=\Law(X)$. Associated to this mean field game is the so-called \emph{master equation}, a PDE for a function $U : [0,T] \times \mathcal{P}_2(\R^d) \times \R^d \to \R$ which plays the role of the value function:
\begin{align}
\label{eq: Master Equation}
\begin{split}
\partial_t U_t(\nu,x)&-H(x,\nabla_x U_t(\nu,x))+\sigma^2 \Delta_x U_t(\nu,x)+f_0(\nu,x) \\
    &+\int_{\R^d} \Big(\sigma^2\text{Tr}(\nabla_y \wgrad U_t(\nu,x,y)) - \nabla_yH(y,\nabla_xU_t(\nu,y)) \wgrad U_t(\nu,x,y)\Big)\,d\nu(y)=0 \\
    U_T(\nu,x)&=g(\nu,x).
\end{split}
\end{align}
When the master equation admits a suitably smooth solution, the unique MFE control is given by 
\begin{equation}
    \label{eq: controls in MFE}
    \alpha(t,\nu,x)=-\nabla_y H(x,\nabla_x U_t(\nu,x))
\end{equation} so that the equilibrium state process is 
\begin{align}
    \label{eq: trajectories in MFE}
    dX_t&=-\nabla_y H(X_t,\nabla_x U_t(\Law(X_t),X_t))\,dt+\sqrt 2 \sigma\,dB_t.
\end{align}
Let $\mu=\L(X)$ denote the law (on path space) of this process. We refer  to \cite{MasterEquation} and \cite{Carmonabook1} for additional background on mean field games.

The \emph{convergence problem} in mean field game theory is the problem of rigorously showing that the $n$-player Nash equilibria converge in some sense to the MFE. This is known to hold qualitatively in great generality \cite{lacker2020convergence,djete2023large}. The earliest quantitative results go through the master equation \cite{MasterEquation}, but more recent progress has developed non-asymptotic viewpoints based on an analysis of the PDEs \cite{cirant2025priori,JacksonGAMES} or forward-backward SDEs \cite{jackson2024quantitative,possamai2025non} which characterize the $n$-player equilibria. 
Comparing the $n$-player equilibrium states to the state process of the MFE, these papers obtain estimates like
\begin{equation}
\sup_{t \in [0,T]}\W^2_1(\mu^{\otimes k}_t,\pi^k_t)=O(k/n). \label{MFGpriorbounds}
\end{equation}
The following theorem improves this to $(k/n)^2$, when the master equation admits a smooth enough solution.

\begin{theorem}
\label{th: chaos for MFG}
Assume that:
    \begin{enumerate} 
    \item The functions $f_0,f_1,g$ are measurable and subpolynomial in the following sense: There exists $C\geq 0$ such that for each $x,a \in \R^d, \nu \in \mathcal P_2(\R^d)$,
    \[|f_0(\nu,x)|+|f_1(x,a)|+|g(\nu,x)|\leq C\Big(1+|x|^2+|a|^2+\int |y|^2\,d\nu(y)\Big).\] 
         \item There exists a classical solution $(u^{i,n})_{i =1}^n$ to the Nash system \eqref{eq: Nash system} that has space derivatives with linear growth, and a corresponding Nash equilibrium given by \eqref{eq: controls in a NE} with a well posed optimal state process $\eqref{eq: trajectories in Nash Equilibrium}$.
         \item  There exists a classical solution $U$ to the master equation \eqref{eq: Master Equation} with $\nabla_x U,\nabla_x \wgrad U,\wgrad U$ bounded and Lipschitz uniformly in time and a corresponding Mean Field Equilibrium $(\mu,\alpha)$ where $\alpha$ is given by \eqref{eq: controls in MFE} and $\mu$ is the law on path space of the equilibrium state process given by \eqref{eq: trajectories in MFE}.
     \end{enumerate}
    Assume further that
    \begin{enumerate}[(a)]
        \item The function $\nabla_x U_t$ belongs to $\mathcal C_{\textup{bd}}^{6}(\mathcal P_2(\R^d)\times \R^d)$ uniformly in $t$:
        \begin{align*}
            \sup_{t\leq T}\sup_{0<\mathsf w +|\boldsymbol{i}|\leq 6}|D^{\mathsf w,\boldsymbol{i}}\nabla_x U_t|\leq C<\infty.
        \end{align*}
        \item The initial condition $\mu_0$ satisfies the $T_1$ inequality of assumption \ref{assumption on Phi}(2).
        \item The function $(x,y)\mapsto \nabla_y H(x,y)$ has 6 bounded derivatives.
    \end{enumerate}
    Then 
\begin{equation*}
\W^2_1(\mu^{\otimes k},\pi^k) = O(k^2/n^2).
\end{equation*}
\end{theorem}
\begin{proof}
The proof combines Theorem \ref{th: short time} with a known estimate from the mean field game literature.
Note that Theorem \ref{th: short time} does not apply directly, because the drift of $Y^i$ in \eqref{eq: trajectories in Nash Equilibrium} varies with $n$ due to the dependence on $n$ of $u^{n,i}$.
We follow the idea of \cite{MasterEquation} of introducing an intermediate process $(Z^i)_{i=1}^n$ defined by
\begin{align} 
    \label{eq: intermediate problem MFGs}
    dZ_t^i&=-\nabla_y H(Z_t^i,\nabla_x U(t,m_{t}^n,Z_t^i))\,dt+\sqrt 2 \sigma dB_t^i, \qquad      m_t^n= \frac 1n \sum_{j=1}^n \delta_{Z_t^j}. 
\end{align}
This system, unlike \eqref{eq: trajectories in Nash Equilibrium}, does fit into the setting of our Theorem \ref{th: short time}, by setting
\[
V(t,\nu,x) = -\nabla_y H(x,\nabla_x U(t,\nu,x)).
\]
(See Remark \ref{rk: time dependent setting-intro} regarding time dependence.) Thanks to the assumed regularity of $U$ and $H$, we can apply our Corollary \ref{cr: metric convergences short time} to obtain a sharp estimate on the propagation of chaos,
\begin{equation}
\W_1\big(\L(Z^1,\ldots,Z^k), \, \mu^{\otimes k}\big) = O(k/n). \label{MFGapplication1}
\end{equation}
This was the piece missing in prior work on the mean field game convergence problem.
It is known from prior work that
\[
\E\bigg[\sup_{0 \le t \le T}|Y^1_t-Z^1_t|\bigg] \le C/n,
\]
for some constant $C$ independent of $n$.
This was shown in \cite[Theorem 6.6]{MasterEquation} under smoothness assumptions for the master equation implied by ours, when the state space is the torus, and we refer to \cite[Theorem 6.32]{CDMFGbook2} or \cite[Theorem 4.2]{DelarueLackerRamananCLT} for analogous estimates on the state space $\R^d$. 
Using exchangeability, this implies
\begin{equation*}
\W_1\big(\L(Z^1,\ldots,Z^k), \, \pi^k\big) \le  \E\bigg[\sup_{0 \le t \le T}\Big(\sum_{i=1}^k \big|Y_t^i-Z_t^i\big|^2\Big)^{\frac 12}\bigg]\leq k \E\Big[\sup_{0 \le t \le T}\big|Y_t^1-Z_t^1\big|\Big] \leq  \frac{Ck}{n}. 
\end{equation*}
Combine this with \eqref{MFGapplication1} and the triangle inequality to complete the proof.
\end{proof}

Note that Theorem \ref{th: chaos for MFG} requires the solution of the master equation $U$ to be very regular. Regularity of the (first order) master equation is well studied, for instance in \cite{DelarueClassical}, \cite{PorettaRegularityME}, \cite{MasterEquation}, and \cite[Section 5]{CDMFGbook2}. In particular the latter three references provide conditions under which the master equation has two bounded Wasserstein derivatives. The following lemma gives some sufficient conditions for $U$ to satisfy the assumption of Theorem \ref{th: chaos for MFG} on short time horizons.  We expect the result to be valid for arbitrary time horizon under a monotonicity condition, such as that of Lasry-Lions \cite{lasry2007mean} or displacement monotonicity \cite{gangbo2022mean}, but we do not pursue this here.

\begin{lemma}
\label{lm: smoothness of MasterEq}
Assume that
    \begin{enumerate}
        \item $\nabla_x f_0,\nabla_x g$ are of class $\mathcal C_{\textup{bd}}^{6}(\mathcal P_2(\R^d)\times \R^d)$;
        \item $\nabla_a^2 f_1(x,a)\geq I\lambda$ for $\lambda>0$;
        \item $\nabla_x f_1,\nabla_y H$ are six times continuously differentiable, with derivatives of order 1--6 being bounded.
    \end{enumerate}
    Then, for $T$ sufficiently small, $U$ satisfies assumption (1) of Theorem \ref{th: chaos for MFG} on $[0,T]$.
\end{lemma}

We omit the proof, as it is essentially a long continuation of the proof of \cite[Theorem 5.49]{CDMFGbook2}. The idea is to consider the BSDE representation of $\nabla_x U$ given by the Master field (\cite[Section 4]{CDMFGbook2}, which allows to prove regularity for $\nabla_x U$ by proving regularity for the driving BSDE system. This is done by differentiating the BSDE informally, checking that the candidate solution is uniquely solvable, and taking limits. This program is carried out in detail in \cite{DelarueClassical} and \cite[Section 5]{CDMFGbook2} for the first two derivatives; higher order derivatives work in a similar way, as shown in the case of forward SDEs by \cite{Chassagneux2022-jt}.

\subsection{Application: Mean Field Control}
In this section we consider the convergence problem for mean field control, which can be seen as a variant of  mean field games in which players cooperate instead of competing. The mathematical analysis is considerably simpler, as optimization is simpler than Nash equilibrium, and as a result much more is known.

Consider cost functions $F,G:\mathcal P_2(\R^d)\to \R$ and $L:\R^d\times \R^d\to \R$. Consider the $n$-particle cooperative control problem
\begin{align}
\label{eq: n player control}
\begin{split}
    \min_{\alpha}\,\,&\E\bigg[\frac 1n\sum_{i=1}^n\int_0^T L(Y_t^i,\alpha_t^i)+ F(m_{Y_t}^n)\,dt+ G(m_{Y_T}^n)\bigg],\text{ where for } x \in (\R^d)^n, m^n_x := \frac{1}{n}\sum_{j=1}^n \delta_{x^j}, \\
    \text{s.t.\,\,} dY_t^i&= \alpha_t^i\,dt+\sqrt 2 \sigma\, dB_t^i \qquad Y_0^i \sim \text{ iid }  \mu_0 , \ \   i=1,\dots,n.
    \end{split}
\end{align}
Classical results yield that under regularity assumptions on $F,G$ and $L$, there exist optimal feedback controls which can be expressed in terms of the value function $v(t,x)$ of the control problem and the Hamiltonian
\begin{equation*}
    H(x,y) = -\inf_{a \in \R^d}\{a\cdot y+L(x,a)\}.
\end{equation*}
where we recall that the value function $v : [0,T] \times (\R^{d})^n \to \R$ solves an $n$-dimensional HJB equation
\begin{align*}
    \partial_t v_t(x)&+\sigma^2\Delta v_t(x)+F(m_{x}^n)-\frac 1n\sum_{i=1}^n H(x^i,n\nabla_{x_i}v_t(x))=0\\
    v_T(x)&=G(m_x^n)
\end{align*}
The optimal trajectories are given by the SDE
\begin{align}
\label{eq: n-control optimal trajectories}
    dY_t^i=-\nabla_y H(Y_t^i,n\nabla_i v_t(Y_t))\,dt+\sqrt 2 \sigma\,dB_t^i
\end{align}
Viewing each $Y^i$ as a random element of $C([0,T];\R^d)$, let $\pi^k=\Law(Y^1,\dots,Y^k)$ for $1 \le k \le n$.

The large $n$ limit of the cooperative problem is described by the mean field control problem,
\begin{align}
\label{eq: mean field control}
   \begin{split} 
    \min_{\alpha} \ \ &\int_0^T\E[L(X_t,\alpha_t)]+ F(\L(X_t))\,dt+ G(\L(X_T)) \\
    \text{s.t.\,\,} dX_t&= \alpha_t\,dt+\sqrt 2 \sigma \, dB_t, \quad X_0 \sim \mu_0.
\end{split}
\end{align}
Under appropriate smoothness assumption on $F$, $G$, and $L$, the optimal trajectory is given by
\begin{align}
\label{eq: MFC optimal trajectory}
\begin{split}
dX_t&=-\nabla_y H(X_t,\wgrad U_t(\mu_t,X_t))\,dt+\sqrt 2 \, \sigma dB_t, \qquad \mu=\Law(X),
\end{split}
\end{align}
where $U : [0,T] \times \mathcal{P}_2(\R^d) \to \R$ is a suitably defined value function of the problem, which solves the infinite dimensional HJB equation
\begin{align*}
    \partial_t U_t(\nu)&+F(\nu)+\int_{\R^d}\Big(\sigma^2\tr\big( \nabla_y \wgrad U_t(\nu,y)\big) - H(y,\wgrad U_t(\nu,y))\Big)d\nu(y)=0\\
    U_T(\nu)&=G(\nu).
\end{align*}
The convergence problem is again about relating the finite-$n$ control problem to the mean field problem. Qualitative results are known in general situations similar to mean field games \cite{lacker2017limit,djete2022mckean}, but the quantitative picture of mean field control is now understood in much finer detail due to the sequence of papers \cite{cardaliaguet2023algebraic,cardaliaguet2023regularity,DAUDIN2024110660,cardaliaguet2023sharpconvergenceratesmean}. These papers, especially  \cite{DAUDIN2024110660}, clarify how convergence rate of the value functions depends on the smoothness of the cost functions. As with mean field games, these sharp convergence rates for the value functions do not translate directly to sharp convergence rates (propagation of chaos) for the state processes.

We address this gap in the literature in the nicest regime where the master equation is smooth. 
\begin{theorem}
\label{th: chaos for control}
    Assume that
    \begin{enumerate}
    \item The functions $F,L,G$ are measurable and subpolynomial in the following sense: there exists $C\geq 0$ such that for each $x,a \in \R^d, \nu \in \mathcal P_2(\R^d)$,
    \begin{equation*}
        |F(\nu)|+|L(x,a)|+|G(\nu)|\leq C\Big(1+|x|^2+|a^2|+\int |y|^2\,d\nu(y)\Big).
    \end{equation*}
         \item There exists a classical solution $v$ to the $n$-HJB such that the optimal state process given by $\eqref{eq: n-control optimal trajectories}$ is well posed.
         \item  There exists a classical solution $U$ to the master equation such that the optimal state given by \eqref{eq: MFC optimal trajectory} is well posed. Moreover $\wgrad U,\nabla_x \wgrad U$ are Lipschitz uniformly in $t$.
     \end{enumerate}
    Assume further that
    \begin{enumerate}[(a)]
        \item The function $\wgrad U$ belongs to $\mathcal C_{bd}^{6}(\mathcal P_2(\R^d)\times \R^d)$ uniformly in time:
        \begin{align*}
            \sup_{t\leq T}\sup_{0<\mathsf w +|\boldsymbol{i}|\leq 6}|D^{\mathsf w,\boldsymbol{i}}\wgrad U_t|\leq C<\infty.
        \end{align*}
        \item The initial condition $\mu_0$ satisfies the $T_1$ inequality of assumption \ref{assumption on Phi}(2).
        \item The function $(x,y)\mapsto \nabla_y H(x,y)$ has 6 bounded derivatives.
    \end{enumerate}
    Then 
\begin{align*}
    \sup_{t \in [0,T]}\W^2_1(\mu_t^{\otimes k},\pi_t^k) = O(k^2/n^2).
\end{align*}
\end{theorem}

We give the proof in Appendix \ref{appendix: control}, and it closely parallels the argument for Theorem \ref{th: chaos for MFG}. The assumed smoothness of the value function $U$ can be checked along similar lines to Lemma \ref{lm: smoothness of MasterEq}, but more easily so we omit the details.

\subsection{Related literature}
The closest result to  our Theorem \ref{th: short time} in prior literature, for non-pairwise $V$ (and not uniform in time), is in previous work of the second author \cite[Theorem 2.17, Corollary 2.18]{HierarchiesPaper}. There it was shown that  $\ent(\pi^k[t]\,\|\,\mu^{\otimes k}[t])) = O(k^{3-\epsilon}/n^{2-\epsilon})$ for arbitrary $\epsilon > 0$. This was not so satisfactory, both due to the worse exponent on $n$ and especially $k$, and more importantly due to restrictive assumptions on $V$. Indeed, there it was required that $V$ admits a power series expansion around $\mu_t$,
\[
V(\nu,x)= V(\mu_t,x)+\sum_{p=1}^\infty \frac{1}{p!}\int_{(\R^d)^p} \delta_m^p  V(\mu_t,x,\boldsymbol{y})d(\nu-\mu_t)^{\otimes p}(\boldsymbol{y})
\]
with a strong decay condition on the flat derivatives:  $\sum_{p=1}^\infty \|\delta_m^p  V\|_\infty e^{cp}/p! < \infty$ for all $c > 0$. This already excludes functions like $ V(\mu,x)=f(\int y\,d\mu(y))$ for smooth $f$ of compact support, which have unbounded flat derivatives but bounded Wasserstein derivatives. And it is prohibitively difficult to check in the setting of mean field games and control described above.

Let us highlight also the result of  \cite[Theorem 4.2]{Szpruch-Tse-sampling}, which also studied non-pairwise interactions $V$, giving a (global) bound on the rate of convergence of the empirical measure. For $\overline{Y}^i\sim\mu$ denoting iid solutions of the McKean-Vlasov equation, coupled with $(Y^1,\ldots,Y^n)$ synchronously (i.e., using the same Brownian motions and initial conditions, they show that
\begin{equation}
\E\bigg[ \frac{1}{n}\sum_{i=1}^n\sup_{0 \le t \le T}|Y^i_t-\overline{Y}^i_t|^2 \bigg] = O(1/n). \label{intro:szpruchtse}
\end{equation}
In the case of pairwise interaction, where $V$ is given as in \ref{intro:pairwise} with $\phi$ Lipschitz, this bound was already implicit in \cite{snitzmann}. And it was already well known that a suboptimal (dimension-dependent) bound can be shown for non-pairwise $V$ under a $\W_2$-Lipschitz assumption by using the same synchronous coupling approach \cite[Theorem 1.20]{carmona2016lectures}.
The contribution of \cite[Theorem 4.2]{Szpruch-Tse-sampling} was to show how smoothness of $V$ leads to the sharper bound in \eqref{intro:szpruchtse} even in the non-pairwise case.
By exchangeability, the global bound \eqref{intro:szpruchtse} implies the local bound
\begin{equation}
    \W_2^2(\mu^{\otimes k}_t,\pi_t^k)=O(k/n). \label{intro:szpruchtse2}
\end{equation}
The main assumption in \cite[Theorem 4.2]{Szpruch-Tse-sampling} is that $ V$ has 3 bounded Wasserstein derivatives. Our Theorem \ref{th: short time} improves \eqref{intro:szpruchtse2} to $O(k^2/n^2)$, at the price of requiring more smoothness of $V$. We employ similar methods of \emph{weak propagation of chaos} as in \cite{Szpruch-Tse-sampling}, but implemented in a new way within the BBGKY approach of \cite{HierarchiesPaper}.

Lastly, it is worth mentioning that relative entropy and related modulated energy methods have been popular recently not only due to their ability to obtain sharp rates of propagation of chaos, but also (arguably more so) because they are able to handle singular (pairwise) interactions which arise in many physically relevant models. The work of Jabin and Z.\ Wang \cite{JabinWangBounded,JWEntropicChaos} was particularly influential in this direction, and S.\ Wang recently showed in \cite{wang2024sharplocalpropagationchaos} how to combine the ideas of \cite{JWEntropicChaos} with the BBGKY approach of \cite{HierarchiesPaper} to achieve sharp rates for certain singular interactions. For additional references and discussion of the active line of research on mean field limits for singular interactions, especially for Coulomb and Riesz gases, we refer to the recent survey \cite{serfaty2024lectures}.
Quantitative propagation of chaos for models without noise is not as well understood; sharp rates are known  in weaker metrics \cite{Paul2019SizeOfChaos,Duerinckx2021SizeOfChaos}.
For further references and methods in the vast literature on propagation of chaos, we refer to the thorough review  articles \cite{ChaintronChaos1,ChaintronChaos2}.

\subsection{Outline} In Section \ref{sect: entropic inequality} we prove via the BBGKY hierarchy a version of the marginal entropy estimates of \cite{HierarchiesPaper,SharpChaosLacker2023}, both in the short and long time regimes, generalized to the setting of non-pairwise interactions. Compared to prior work, a new remainder term appears here, and its analysis is the subject of Sections \ref{sect: weak chaos} and \ref{sect: remainder}. Section \ref{sect: weak chaos}  presents some general tools of weak propagation of chaos borrowed from \cite{Chassagneux2022-jt}, and in Section \ref{sect: remainder} we use them to prove the main estimate of the remainder term. Three appendix sections are devoted to justifying claims from the applications to mean field Langevin dynamics, games, and control.

\section{Entropy inequality via BBGKY hierarchy}
\label{sect: entropic inequality}
In this section we establish the differential inequality \eqref{eq: differential inequality} and prove Theorems \ref{th: short time} and \ref{th: UiT}, pending estimates from Section \ref{sect: remainder}. Assumption \ref{assumption on Phi} is in force at all times. We work with a filtered probability space $(\Omega,\mathcal A,\mathcal F,\P)$ that supports an $\mathcal F$-Brownian motion $B_t$, and assume that the process $Y_t=(Y_t^1,\dots,Y_t^n)$ is the unique strong solution of \eqref{eq: particle SDE}. We will use the letter $C$ to denote a constant that may change from line to line and that is allowed to depend on $\sigma,C_{\text{T}_1}(0),\mu_0$ and the derivative bounds on $V$, but crucially not on $k$ or $n$.

\subsection{Hierarchies of linear differential inequalities}

As a first preparation, we present a Gronwall-type lemma for a type of linear systems of differential inequality that arises in BBGKY-approaches to entropic propagation of chaos. Initially, in \cite{HierarchiesPaper}, it was analyzed by a direct iteration and somewhat tedious manipulations of iterated integrals. A simpler inductive argument is given in \cite{Hess-Childs2025-ux}. Here we give a short proof following the strategy of \cite[Section 1.4]{lacker2024quantitativepropagationchaosnonexchangeable}.

Consider functions $R=R_t:[0,T] \to \R_+$ and $f=f_t(k) : [0,T] \times [n] \to \R_+$, and scalars $a,b,c\geq 0, p\geq 1$. We will deal with the following system of differential inequalities: For each $1 \le k \le n$, 
\begin{equation}
\label{eq: linear differential inequality}
    \frac{d}{dt}f_t(k)\leq ak(f_t(k+1)-f_t(k))1_{k < n}  + \frac{b}{n^2}k^p + kR_t - cf_t(k).
\end{equation}
Hierarchies of this type are a cornerstone of sharp approaches to propagation of chaos. 

\begin{lemma}
  \label{lm: linear differential inequality}
  Assume that $f:[0,T]\times [n]\to \R_+$ is a non-decreasing function in $k$ that satisfies \eqref{eq: linear differential inequality} for $p\in\{2,3\}$. Assume that there exists a constant $C_0\geq 0$ such that 
  \begin{align}
  \label{eq: yule process initial condition}
      f_s(k)\leq C_0\frac{k^2}{n^2}, \quad 1 \le k \le n.
  \end{align}
  Then, for $0 \le s \le T$ and $1 \le k \le n$
    \begin{equation*}
        f_T(k)\leq 2C_0\frac{k^2}{n^2} e^{(2a-c)(T-s)}+\frac{8bk^p}{n^2}\int_s^T e^{(ap-c)(t-s)} \,dt + k\int_s^T e^{(a-c)(t-s)} R_t \,dt.
    \end{equation*}
\end{lemma}
\begin{proof}
We will use a semigroup interpretation of \eqref{eq: linear differential inequality} to organize the analysis.
Consider the $n \times n$ matrix $\mathcal G$ which acts on a function $\varphi : [n] \to \R$ by
\[\mathcal G \varphi(k)=ak(\varphi(k+1)-\varphi(k))\mathbf{1}_{k<n}.\]
Each row of this matrix sums to zero (i.e., $\mathcal G$ maps constant functions to zero), and every off-diagonal entry is nonnegative. Hence, $\mathcal G$ is the infinitesimal generator of a continuous-time Markov process on state space $[n]$; in fact, this process goes by the name of the \emph{Yule process} stopped at level $n$, and we denote it by $(\mathcal Y_t)_{t\geq 0}$. The matrix exponential is the semigroup of this process and admits the representation $e^{t\mathcal G}\varphi(k)=\E[\varphi(\mathcal Y_t)\,|\,\mathcal Y_0=k]$. We will not use this stochastic interpretation except to deduce a monotonicity property, namely that $e^{t\mathcal G}\varphi \ge e^{t\mathcal G}\psi$ pointwise whenever $\varphi \ge \psi$ pointwise.
With this notation we can rewrite \eqref{eq: linear differential inequality} as a coordinatewise inequality between vectors,
\begin{equation*}
    \frac{d}{dt}f_t \leq (\mathcal G - cI)f_t + \frac{b}{n^2}m_p + R_tm_1,
\end{equation*}
where $I$ denotes the $n \times n$ identity matrix, and $m_q(k) := k^q$ for $q > 0$.
Using the aforementioned monotonicity, we deduce
\begin{align*}
    \frac{d}{dt}\Big[e^{-ct} e^{t\mathcal G}f_{T-t} \Big] &= e^{-ct} e^{t\mathcal G}\Big[(\mathcal G - cI) f_{T-t}-\frac{d}{ds}f_{s}\big\vert_{s=T-t}\Big] \\
    &\ge - e^{-ct}e^{t\mathcal G}\big(\frac{b}{n^2}m_p + R_tm_1\big).
\end{align*}
Integrate to get the coordinatewise inequality
\begin{align}
\label{eq: semigroup after gronwall}
    f_{T} &\le e^{-c(T-s)}e^{(T-s)\mathcal G} f_s + \int_s^Te^{-c(t-s)}e^{(t-s)\mathcal G}\Big(\frac{b}{n^2} m_p + R_tm_1\Big)\,dt.
\end{align}
Now we estimate $e^{t\mathcal G}m_q$ for $q=1,2,3$. This could be done by using the known fact that the Yule process follows a negative binomial distribution, but we instead present a short self-contained argument. 

\begin{itemize}
\item $q=1$: Use $\mathcal G m_1(k) = ak(m_1(k+1)-m_1(k))=ak$ to get
\begin{align*}
\frac{d}{dt}e^{t\mathcal G}m_1=e^{t\mathcal G}\mathcal Gm_1 = ae^{t\mathcal G}m_1,
\end{align*}
and thus $e^{t\mathcal G}m_1=e^{at}m_1$.
\item $q=2$: We have $\mathcal G m_2(k) = ak\big((k+1)^2-k^2\big) = 2a m_2(k) + am_1(k)$, and thus
\begin{align*}
\frac{d}{dt}e^{t\mathcal G}m_2 &= 2ae^{t\mathcal G}m_2 + a e^{t\mathcal G}m_1.
\end{align*}
Multiply by $e^{-2at}$ and integrate to get
\begin{align*}
e^{t\mathcal G}m_2 &= e^{ 2at}m_2 + ae^{2at}\int_0^t e^{-2as} e^{s\mathcal G}m_1 \,ds.
\end{align*}
Using $e^{t\mathcal G}m_1=e^{at}m_1$, we get the coordinatewise inequality
\[
e^{t\mathcal G}m_2 = e^{ 2at}m_2 + e^{2at}(1-e^{-at}) m_1 \le 2 e^{2at}m_2.
\]
\item $q=3$: We have $\mathcal G m_3(k) = ak\big((k+1)^3-k^3\big) = 3a m_3(k) + 3am_2(k) + am_1(k)$, and thus
\begin{align*}
\frac{d}{dt}e^{t\mathcal G}m_3 &= 3ae^{t\mathcal G}m_3 + 3a e^{t\mathcal G}m_2 + a e^{t\mathcal G}m_1.
\end{align*}
Multiply by $e^{-3at}$ and integrate to get
\begin{align*}
e^{t\mathcal G}m_3 &= e^{ 3at}m_3 + ae^{3at}\int_0^t e^{-3as} \big(3  e^{s\mathcal G}m_2 +   e^{s\mathcal G}m_1\big) \,ds.
\end{align*}
Use $e^{s\mathcal G}m_1\le e^{2as} m_3$ and $e^{s\mathcal G}m_2 \le 2 e^{2as}m_3$ to get
\begin{align*}
e^{t\mathcal G}m_3 &\le \bigg[e^{ 3at} + 7ae^{3at}\int_0^t e^{-as} \,ds\bigg] m_3 \le 8 e^{3at} m_3.
\end{align*}
\end{itemize}
Using the $q=2$ case and the time-$s$ assumption, we have $e^{(T-s)\mathcal G} f_s(k) \le 2C_0 e^{2a(T-s)}k^2/n^2$.
We can summarize all cases $q=1,2,3$ as $e^{t\mathcal G}m_q \le 8e^{qat}m_q$.
Combining these bounds with \eqref{eq: semigroup after gronwall} yields
\begin{align*}
f_{T}(k) \le 2C_0 e^{(2a-c)(T-s)} \frac{k^2}{n^2}  + \int_s^Te^{-c(t-s)}\Big(\frac{8bk^p}{n^2}e^{pa(t-s)}   + R_tke^{a(t-s)}\Big)\,dt.
\end{align*}
\end{proof}

\subsection{Short time Analysis}
In this section we prove the differential inequality \ref{eq: differential inequality}, and perform the iteration procedure to tackle the pairwise term; we leave the analysis of the remainder for the next sections. First note that Assumption \ref{assumption on Phi}(1) implies that $V$ is $L$-Lipschitz for some $L>0$, in the sense that
\[
|V(\nu,x)-V(\nu',x')| \le L(|x-x'| + \W_2(\nu,\nu')).
\]
We begin with a fairly standard fact about T$_1$ inequalities propagating along Lipschitz SDE solutions, but we include the short proof in Appendix \ref{ap:T1shorttime} as we could not locate a precise reference with explicit constants.

\begin{lemma} \label{lm: T1 short time}
Suppose Assumption \ref{assumption on Phi} holds. Define
\[
C_{\textup{T}_1}(t) := e^{2Lt}\max(2C_{\textup{T}_1}(0),4\sigma^2 t).
\]
Let $0 \le t \le T$. Then $\mu[t]$ satisfies the T$_1$ inequality
\[
\W_1^2(\mu[t],\nu[t])\leq 2C_{\textup{T}_1}(t) \,\ent (\nu[t]\,\|\,\mu[t]), \quad \text{ for all } \nu[t] \in \mathcal P(C([0,t];\R^d)).
\]
Moreover, for every $p\geq 1$ we have 
\[
\sup_{t \le T}\E\big[|X_t|^p\big]+\sup_{t \le T}\E\big[|Y^1_t|^p\big] \leq C_{p,T} < \infty.
\]
where the constant $C_{p,T}$ does not depend on $n$.
\end{lemma}

The first main ingredient of the argument is a probabilistic representation of the law $\pi^k[T]$, provided by a well known projection lemma (see \cite[Lemma 4.1]{HierarchiesPaper} and references therein).
Let $Y^i[t]=(Y^i_s)_{s \in [0,t]}$ denote the trajectory up to time $t$ of $Y^i$, for each $i=1,\ldots,n$, and similarly write $Y^{[k]}[t]=(Y^1[t],\dots,Y^k[t])$.

\begin{lemma}
\label{lm: projection path space}
Suppose Assumption \ref{assumption on Phi} holds.
There exists a standard Brownian motion $W=(W^i)_{i=1}^k$, with respect to the filtration generated by $Y^{[k]}$, such that
\begin{equation}
    dY_t^i =  \E\big[ V(m_t^n,Y_t^i)\mid Y^{[k]}[t]\big]\,dt+\sqrt2 \sigma dW^i_t , \ \    i=1,\dots,k. \label{eq:projectionSDE}
\end{equation}
\end{lemma}

This can be seen as a form of the BBGKY hierarchy, and we use it to develop entropy estimates. Because of the non-linearity of $V$, there will be a remainder term $\mathcal R(t)$ that is new here compared to prior work on sharp propagation of chaos.

\begin{proposition}
\label{pr: differential inequality short time}
Suppose Assumption \ref{assumption on Phi} holds.
\begin{enumerate}
    \item There exists a constant $C>0$ independent of $n$ and $k$ such that
    \begin{equation}
    \label{eq: differential inequality short suboptimal}
    \ent(\pi^k[T]\,\|\,\mu^{\otimes k}[T])\leq C\frac{k^3}{n^2}+C\int_0^T k  \mathcal R(t)\,dt,
\end{equation}
where we define
\begin{align}
    \mathcal R(t) = \int_0^1\E\bigg[\Big|\int_{\R^d\times \R^d} \delta_m^2  V(sm_t^n+(1-s)\mu_t,Y_t^1,y,z)\,d(m_t^n-\mu_t)^{\otimes2}(y,z)\Big|^2\bigg]\,\,ds. \label{def:remainder1}
\end{align}
\item If we further assume that $\ent(\pi^3[T]\,\|\,\mu^{\otimes 3}[T])\leq M/n^2$ for some constant $M$ independent of $n$ and $k$, then 
\begin{align}
\label{eq: differential inequality short optimal}
    \ent(\pi^k[T]\,\|\,\mu^{\otimes k}[T])\leq C\frac{k^2}{n^2}+C\int_0^Tk \mathcal R(t)\,dt,
\end{align}
where again $C$ is a constant that does not depend on $n$ or $k$.
\end{enumerate}
\end{proposition}
\begin{proof}
We begin with \textit{(1)}. Let $1 \le k < n$. By Lemma \ref{lm: projection path space}, $\pi^k[T]$ is the law of a weak solution of the SDE \eqref{eq:projectionSDE}. Using a well known formula for the path entropy (e.g., \cite[Lemma 4.4]{HierarchiesPaper}) and exchangeability of $Y^1,\dots,Y^k$ we have the following formula for the path entropy between $\pi^k$ and $\mu^{\otimes k}$:
\begin{align}
\begin{split}
    \label{eq: entropy in terms of L2}
    \ent(\pi^k[t]\,\|\,\mu^{\otimes k}[t])&=\sum_{i=1}^k\frac{1}{4\sigma^2}\int_0^t \E\bigg[\Big | V(\mu_s,Y_s^i)-\E\big[ V(m_s^n,Y_s^i)\mid Y^{[k]}[s]\big]\Big |^2\bigg]\,ds\\
    &=\frac{k}{4\sigma^2}\int_0^t \E\bigg[\Big | V(\mu_s,Y_s^1)-\E\big[ V(m_s^n,Y_s^1)\mid Y^{[k]}[s]\big]\Big |^2\bigg]\,ds.
    \end{split}
    \end{align}
    In particular, $\ent(\pi^k[t]\,\|\,\mu^{\otimes k}[t])$ is absolutely continuous in $t$, and it will be convenient to take derivatives. First, we expand $V$ around $\mu_t$ in flat derivatives up to second order:
\begin{align*}
V&(\mu_t,Y_t^1) - \E\Big[ V(m_t^n,Y_t^1)\mid Y^{[k]}[t]\Big] \\
    &= \E \Big[\int_{\R^d} \delta_m  V(\mu_t,Y_t^1,y)d(m_t^n-\mu_t)(y)\,\Big \vert\, Y^{[k]}[t]\Big]\\
    &\quad + \int_0^1(1-s)\E\Big[\int_{\R^d}\int_{\R^d}\delta_m^2  V(s\mu+(1-s)m_t^n,Y_t^1,y,z)\,d(m_t^n-\mu_t)^{\otimes 2}(y,z)\,\Big\vert\,Y^{[k]}[t]\Big]\,ds.
\end{align*}
This is nothing but a Taylor expansion with remainder of $f(1)$ around $f(0)$ for the scalar function $f(s)=V(s\mu_t+(1-s)m_t^n,Y_t^1)$; for details see, e.g., \cite[Lemma 2.2]{Chassagneux2022-jt}. Break up the square and use Jensen's inequality to eliminate the conditional expectation from the second term,
\begin{align}
\label{eq: entropy first bound} 
\frac{d}{dt}\ent(\pi^k [t]\,\|\,\mu^{\otimes k}[t])
    &\leq  \text{Term}_1+\frac{k}{2\sigma^2}\mathcal R(t),
\end{align}
where $\mathcal R(t)$ was defined in the statement of the proposition, and
\begin{align*}
 \text{Term}_1 &:= \frac{k}{2\sigma^2} \E\bigg[\Big|\E\Big[\,\frac{1}{n}\sum_{i=1}^n \delta_m  V(\mu_t,Y_t^1,Y_t^i)-\int \delta_m  V(\mu_t,Y_t^1,z)\,d\mu_t(z)\,\Big\vert\,Y^{[k]}[t]\Big]\Big|^2\bigg].
\end{align*}
Term$_1$ appears also in the previous arguments from \cite{HierarchiesPaper}: observe that it features a pairwise interaction structure. The second term $\mathcal R(t)$ is a remainder that is due to the higher-order interactions embodied in $ V$ (and indeed $\mathcal R(t)=0$ for pairwise interactions $V$). 

We first analyze Term$_1$, along similar lines to \cite{HierarchiesPaper}.
Break up the sum to find
\begin{align}
\label{eq: k^3 equation}
\begin{split}
    \text{Term}_1&\leq \frac{k}{\sigma^2 n^2}\E\bigg[\Big |\sum_{i=1}^k  \delta_m V(\mu_t,Y_t^1,Y_t^{i})-\int_{\R^d}\delta_m V(\mu_t,Y_t^1,z)\,d\mu_t(z) \Big |^2\bigg]\\
    &+  \frac{k}{\sigma^2 n^2}\E\bigg[\bigg | \sum_{i=k+1}^n\bigg(\E\Big[ \delta_m V(\mu_t, Y_t^1,Y_t^{i})\,\Big\vert\, Y^{[k]}[t]\Big]-\int_{\R^d}\delta_m V(\mu_t,Y_t^1,z)\,d\mu_t(z)\bigg)\bigg |^2\bigg].
\end{split}
\end{align}
Fix a continuous function $y:[0,T]\to (\R^{d})^{k}$ and $i>k$. By Assumption \ref{assumption on Phi}(3), the map $z\mapsto \delta_m V(\nu, x,z)$ is Lipschitz uniformly in $\nu,x$ since $\nabla_z \delta_m  V(m,x,z)=\wgrad  V(m,x,z)$ is bounded. Define $\pi^{k+1\mid k}_{[t]\mid [t]}(\cdot \mid Y^{[k]}[t])$ to be the conditional law of $Y^{k+1}[t]$ given $Y^{[k]}[t]$. By Kantorovich duality,
\begin{align}
\label{eq: change of measure}
\begin{split}
\bigg|\E&\bigg[ \delta_m V(\mu_t, Y_t^1,Y_t^{i})\,\Big\vert\, Y^{[k]}[t]\bigg]-\int_{\R^d}\delta_m V(\mu_t,Y_t^1,z)\,d\mu_t(z)\bigg|\\
&=\bigg|\int_{\R^d} \delta_m  V(\mu_t,Y_t^1,z_t) \,d\Big(\pi_{[t]\mid [t]}^{i \mid k}\,(\cdot \mid Y^{[k]}[t])-\mu[t]\Big)(z)\bigg|\\
&\leq \|\wgrad V\|_\infty \W_1\big(\pi_{[t]\mid [t]}^{i \mid k}(\cdot \mid Y^{[k]}[t]),\mu[t]\big).
\end{split}
\end{align}
Plug \ref{eq: change of measure} into \ref{eq: k^3 equation} and apply Cauchy-Schwarz to both sums to see 
\begin{align*}
   \text{Term}_1&\leq \frac{k^2}{\sigma^2 n^2}\sum_{i=1}^k\E\bigg[\Big |\delta_m  V(\mu_t,Y_t^1,Y_t^1)-\int_{\R^d}\delta_m  V(\mu_t,Y_t^1,z)\,d\mu_t(z)\Big |^2\bigg]\\ 
    &+\frac{k (n-k)}{\sigma^2 n^2}\|\wgrad V\|_\infty^2\sum_{i=k+1}^n\E\Big[   \W_1^2\big(\pi_{[t]\mid [t]}^{i \mid k}(\cdot \mid Y^{[k]}[t]),\mu_t\big) \Big].
    \end{align*}
    Use again the fact that $\delta_m V$ is Lipschitz and the the square-integrability of $Y^i$ from Lemma \ref{lm: T1 short time} to bound the first term by $Ck^3/n^2$. For the second term, the T$_1$ inequality from Lemma \ref{lm: T1 short time} yields
    \begin{equation} \label{T1inequality-proof}
    \W_1^2\big(\pi_{[t]\mid [t]}^{i \mid k}(\cdot \mid Y^{[k]}[t]),\mu[t]\big) \le 2C_{\text{T}_1}(t) \ent\Big(\pi_{[t]\mid [t]}^{k+1\mid k}(\cdot\mid Y^{[k]}[t])\,\big\|\,\mu[t]\Big).
    \end{equation}
    By exchangeability, the conditional distributions $\pi^{i\mid k}_{t\mid [t]}$ are the same for every $i>k$. Thus 
    \begin{align}
    \text{Term}_1 &\le C\frac{k^3}{n^2} + \frac{2k}{\sigma^2}\|\wgrad V\|_\infty^2 C_{\text{T}_1}(t) \E\bigg[\ent\Big(\pi_{[t]\mid [t]}^{k+1\mid k}(\cdot\mid Y^{[k]}[t])\,\big\|\,\mu[t]\Big)\bigg].
    \end{align}
    Use the chain rule of relative entropy to get
    \begin{align}
    \begin{split}    \label{eq: term 1 end}
        \text{Term}_1&\leq C\frac{k^3}{n^2}+ \frac{2k}{\sigma^2}\|\wgrad V\|_\infty^2 C_{\text{T}_1}(t)\bigg(\ent(\pi^{k+1}[t]\,\|\,\mu^{\otimes k+1}[t])-\ent(\pi^k[t]\,\|\,\mu^{\otimes k}[t])\bigg).
    \end{split}
    \end{align}
    This derivation was valid for $1 \le k < n$, and for $k=n$ the same bound holds except that the final term vanishes.
    Plug \eqref{eq: term 1 end} into \eqref{eq: entropy first bound} to finally obtain the differential inequality \eqref{eq: differential inequality} from the introduction:
    \begin{align}
    \begin{split}
    \label{eq: differential inequality end}
    \frac{d}{dt}\ent(\pi^k[t]\,\|\,\mu^{\otimes k}[t])&\leq C\frac{k^3}{n^2} +Ck\bigg(\ent(\pi^{k+1}[t]\,\|\,\mu^{\otimes k+1}[t])-\ent(\pi^k[t]\,\|\,\mu^{\otimes k}[t])\bigg) +\frac{k}{2\sigma^2}\mathcal R(t).
    \end{split}
\end{align}
We then apply Lemma \ref{lm: linear differential inequality} to the functions $f_t(k) = \ent(\pi^k[t]\,\|\,\mu^{\otimes k}[t])$, with $p=3$, and absorb the exponentials into the constant $C$, to obtain
\begin{align*}
    \ent(\pi^k[T]\,\|\,\mu^{\otimes k}[T])&\le C\frac{k^3}{n^2}+C k\int_0^T \mathcal R(t)\,dt.
\end{align*}
This completes the proof of \eqref{eq: differential inequality short suboptimal}.

We next prove \textit{(2)}, under the extra assumption that $\ent(\pi^{3}[T]\,\|\,\mu^{\otimes 3}[T])\leq M/n^2$ for some $M\geq 0$. The point is that this additional assumption lets us sharpen the estimate on the first expression on the right-hand side of \eqref{eq: k^3 equation} from $Ck^3/n^2$ to $Ck^2/n^2$. To see this, define
\[
Z_t^i= \delta_m  V(\mu_t,Y_t^1,Y_t^i)-\int_{\R^d} \delta_m  V(\mu_t,Y_t^1,z)\,d\mu_t(z).
\] 
Use the same reasoning as in \eqref{eq: change of measure}, followed by the transport inequality: For $i \neq 1$ and $j \notin \{1,i\}$, taking $k=2$, we have
\begin{align*}
    \E\Big[|\E[Z_t^j \mid Y^1[t],Y^i[t]]|^2\Big]&= \E\bigg[\Big|\E\Big[\delta_m  V(\mu_t,Y_t^1,Y_t^j)-\int_{\R^d} \delta_m  V(\mu_t,Y_t^1,z)\,d\mu_t(z)\,\Big\vert\, Y^1[t],Y^i[t]\Big]\Big|^2\bigg]\\
    &\leq 2\|\wgrad V\|_\infty^2 C_{\text{T}_1}(t) \E\Big[\ent\Big(\pi_{t\mid [t]}^{k+1\mid k}(\cdot\mid Y^{1}[t],Y^{i}[t])\,\big\|\,\mu_t\Big)\Big]\\
    &\leq 2\|\wgrad V\|_\infty^2 C_{\text{T}_1}(t) \ent(\pi^3[t]\,\|\,\mu^{\otimes 3}[t]).
\end{align*}
The last step used  the chain rule for relative entropy and discarded the subtracted term. This is bounded by $C/n^2$ thanks to the assumption that $\ent(\pi^{3}[T]\,\|\,\mu^{\otimes 3}[T])\leq M/n^2$.
Clearly $Z_t^i$ is a function of $Y_t^1,Y_t^i$, and thus by the tower property, Cauchy--Schwartz, the above bound and the moment bound of Lemma \ref{lm: T1 short time}, we have
\[
\E[|Z_t^i \cdot Z_t^j|]=\E\big[|Z_t^i \cdot \E[Z_t^j \mid Y^1[t],Y^i[t]]|\big] \leq C/n, \quad 1 \le i < j \le k.i \neq j.
\]
We return to the first expression on the right-hand side of \eqref{eq: k^3 equation}, using this new bound to control the cross terms when  expanding the  square of the sum. That is, 
\begin{align*}
\E\bigg[ &\Big|\sum_{i=1}^k \delta_m  V(\mu_t,Y_t^1,Y_t^i)-\int_{\R^d} \delta_m  V(\mu_t,Y_t^1,z)\,d\mu_t(z)\Big|^2\bigg] = \E\bigg[\Big|\sum_{i=1}^k Z_t^i \Big|^2\bigg] \\
    &= \sum_{i=1}^k \E[|Z_t^i|^2] +  \sum_{i=1}^k \sum_{j \neq i}\E[Z_t^i \cdot Z_t^j] \\
    &\leq Ck +C\frac{k^2}{n} \le Ck.
    \end{align*}
    Hence, the first expression on the right-hand side of \eqref{eq: k^3 equation} is bounded by $Ck^2/n^2$.
The rest of the proof continues as before, yielding the differential inequality \eqref{eq: differential inequality end} except with the $k^3$ replaced by $k^2$. We again apply Lemma \ref{lm: linear differential inequality}, now with $p=2$ instead of $p=3$, to obtain the  claimed inequality \eqref{eq: differential inequality short optimal}.
\end{proof}

The main estimate on the remainder is the following:
\begin{proposition}
\label{pr: estimating remainder}
Suppose Assumption \ref{assumption on Phi} holds.
There exists a constant $C=C(T)$ such that $\mathcal R(t) \le C/n^2$ for every $t\leq T$.
\end{proposition}

The proof of Proposition \ref{pr: estimating remainder} is involved and will occupy Sections \ref{sect: weak chaos} and \ref{sect: remainder}. Taking it for granted for now, we show how to prove Theorem \ref{th: short time}.

\begin{proof}[Proof of Theorem \ref{th: short time}]
    Use Proposition \ref{pr: estimating remainder} along with the inequality from Proposition \ref{pr: differential inequality short time}(1) to obtain  $\ent (\pi^k[T]\,\|\,\mu^{\otimes k}[T]) \leq Ck^3/n^2$ for all $1 \le k \le n$, 
    where $C$ does not depend on $k$ or $n$. As a result, $\ent (\pi^3[T]\,\|\,\mu^{\otimes 3}[T])\leq C/n^2$, and thus Proposition \ref{pr: differential inequality short time}(2) applies. Combine Proposition \ref{pr: differential inequality short time}(2) with Proposition \ref{pr: estimating remainder} to get  $\ent (\pi^k[T]\,\|\,\mu^{\otimes k}[T]) \leq Ck^2/n^2$ for all $1 \le k \le n$, which completes the proof.
\end{proof}

We briefly explain how to adapt the proof to the time-dependent case.
\begin{remark}
    \label{rk: time dependent setting} Theorem \ref{th: short time} is stated and proved for a time independent drift $ V=V(m,x)$. If $ V$ depends on time, we can apply the usual trick of treating the time variable as an additional coordinate of the spatial variable. For instance, rewrite \eqref{eq: main SDE} as
    \begin{align*}
        dZ_t&=\,dt\\
        dX_t&=  V(\mu_t,X_t,Z_t)\,dt+\sqrt 2 \sigma\,dB_t.
    \end{align*}
    with $Z_0=0$, and transform the particle system \eqref{eq: particle SDE} similarly, by appending a time variable for each particle. The only issue with this transformation is that the $d+1$-dimensional process $(X_t,Z_t)$ is now driven by a degenerate noise. However, the only place in our paper where non-degenerate noise is used is to justify the entropy identity  \eqref{eq: entropy in terms of L2}. But, as noted in \cite[Remark 4.5]{HierarchiesPaper}, the entropy identity still holds in this case, because the noise is degenerate only in a coordinate for which the drift does not depend on the measure argument. In fact, for a similar reason, our analysis applies also to kinetic (underdamped) mean field Langevin dynamics. Note that since the drift $V$ does not depend on the law of $Z$, we do not need assumptions on time regularity.
\end{remark}

\begin{proof}[Proof of Corollary \ref{cr: metric convergences short time}]
The bound on total variation follows immediately by Pinsker's inequality. By Lemma \ref{lm: projection path space}, we can realize $Y^{[k]}=(Y^1,\ldots,Y^k)$ as the solution of an SDE driven by a Brownian motion $W=( W^i)_{i=1}^k$ in the filtration generated by $Y^{[k]}$.
Now, consider $k$ iid copies of the McKean-Vlasov equation, coupled synchronously to $Y^{[k]}$ using this Brownian motion $W$. That is, consider the SDE system
\begin{align*}
dX_t^i =  V(\mu_t,X^i_t)dt + \sqrt 2 \sigma\,d W_t^i, \quad X^i_0=Y^i_0.
\end{align*}
Since the coefficients are Lipschitz, there exists a unique strong solution to the above SDE, with $\Law(X_t^i)=\mu_t$ for each $i$ thanks to uniqueness of the McKean-Vlasov equation. Recalling the form of the dynamics for $Y^{[k]}$ from \eqref{eq:projectionSDE}, we use Ito's formula, Young's inequality, and the $L$-Lipschitz property of $ V$ to obtain
\begin{align*}
    \sup_{t\leq T}&|X_t^i-Y_t^i|^2=2\sup_{t\leq T}\int_0^t (X_s^i-Y_s^i) \cdot ( V(\mu_s,X^i_s)-\E\big[ V(m_s^n,Y_s^i)\mid Y^{[k]}[s]\big])\,ds\\
    &\leq \int_0^T |X_s^i-Y_s^i|^2+2| V(\mu_s,X_s^i)- V(\mu_s,Y_s^i)|^2+2\Big| V(\mu_s,Y_s^i)-\E\big[ V(m_s^n,Y_s^i)\mid Y^{[k]}[s]\big]\Big|^2\,ds\\
    &\leq \int_0^T (1+2L)\sup_{t\leq s}|X_t^i-Y_t^i|^2+2\Big| V(\mu_s,Y_s^i)-\E\big[ V(m_s^n,Y_s^i)\mid Y^{[k]}[s]\big]\Big|^2\,ds
\end{align*}
Using Gronwall's inequality, taking expectations and summing over the indices yields
\begin{align*}
    \E\bigg[\sum_{i=1}^k\sup_{t\leq T}|X_t^i-Y_t^i|^2\bigg]&\leq 2e^{(1+2L)T}\sum_{i=1}^k\int_0^T \E\bigg[\Big| V(\mu_s,Y_s^i)-\E\big[ V(m_s^n,Y_s^i)\mid Y^{[k]}[s]\big]\Big|^2\bigg]\,ds\\
    &=2e^{(1+2L)T}4\sigma^2\ent (\pi^k[T]\,\|\,\mu^{\otimes k}[T])=O(k^2/n^2).
\end{align*}
Where we are using \cite[Lemma 4.4]{HierarchiesPaper} to identify the time integral with relative entropy. Because $\Law\big(Y^{[k]},(X^1,\dots,X^k)\big)$ is a coupling of $\pi^k$ and $\mu^{\otimes k}$, we deduce that
\begin{equation*}
\W_2^2(\pi^k[T],\mu^{\otimes k}[t])\leq\E\bigg[\sup_{t\leq T}\sum_{i=1}^k |X_t^i-Y_t^i|^2\bigg]\leq \E\bigg[\sum_{i=1}^k\sup_{t\leq T}|X_t^i-Y_t^i|^2\bigg]=O(k^2/n^2). \qedhere
\end{equation*}
\end{proof}

\subsection{Uniform in Time Analysis}
    In this section we prove the uniform in time estimates from \ref{th: UiT} under the stronger set of assumptions \ref{Assumption: UiT}.  
    We begin by showing that our assumptions imply some fairly standard moment bounds, as well as the somewhat less standard fact that $\mu_t$ satisfies a log-Sobolev inequality uniformly in $t \ge \overline{T}$, after some threshold time $\overline{T}$. The latter fact is less standard because we have not assumed an LSI at time zero, just that $\mu_0$ is subgaussian. The holds after a certain time thanks to the regularizing effects of the diffusion, which we will show (quantitatively) using a result of  \cite{Chen2021-ch}.
    
    \begin{remark}
    \label{rk: UiT LSI}
    Assumption \ref{Assumption: UiT}(3) implies that $V$ has negative definite Jacobian $\nabla_x  V \le -\lambda I$ by \cite[Proposition B.6]{GangboConvexity}. Equivalently,
    \begin{equation}
    (V(\mu_t,x)-V(\mu_t,y))\cdot (x-y) \le -\lambda |x-y|^2, \quad \forall x,y\in\R^d, \ t \ge 0. \label{remark:dissipation}
    \end{equation}
    Technically \cite[Proposition B.6]{GangboConvexity} is stated for the Wasserstein gradient of a geodesically convex function, but the proof adapts without change to the monotone non-gradient case.
\end{remark}
We will make some use of the 2-divergence and relative Fisher information, defined respectively as follows for $\alpha,\beta \in \mathcal P(\R^d)$:
\begin{align*}
    D_2(\alpha\,\|\,\beta) &=\begin{cases}
    \int \big|\frac{d\alpha}{d\beta}\big|^2\,d\beta \text{ if }\alpha \ll\beta \\
    +\infty \text{ otherwise,} \end{cases} \\
    \fin (\alpha\,\|\,\beta) &=\begin{cases}
        \int |\nabla \log \frac{d\alpha}{d\beta}|^2 \, d\alpha \text{  if } \alpha \ll \beta, \text{ and } \log \frac{d\alpha}{d\beta} \text{ admits a weak gradient} \\
        +\infty \text{ otherwise}.
    \end{cases}
\end{align*}

\begin{proposition}
\label{pr: LSI}
Suppose Assumption \ref{Assumption: UiT} holds.
For every $\delta > 0$ there exist $\overline{T} \in [0,\infty)$ such that, for every $t\geq \overline{T}$, $\mu_t$ satisfies a log-Sobolev inequality with constant $C_{\textup{LS}} \le 2(1+\delta)\sigma^2/\|\wgrad V\|_\infty$. That is, 
    \[
    \ent(\nu\,\|\,\mu_t)\leq \frac{C_{\textup{LS}}}{4}\fin(\nu\,\|\,\mu_t), \quad \text{for all } \nu\in \mathcal P(\R^d).
    \]
    Moreover, for every $p \ge 1$, we have
    \[
\sup_{t \ge 0}\E\big[|X_t|^p\big]+\sup_{t \ge 0}\E\big[|Y^1_t|^p\big] \leq C_p < \infty.
\]
where the constant $C_p$ does not depend on $n$.
\end{proposition}
\begin{proof}
We begin with the moment bounds.
The drift of the system $(Y^1,\ldots,Y^n)$ is given by $b=(b^1,\ldots,b^n) : (\R^d)^n \to (\R^d)^n$, where
\[
b^i(y^1,\ldots,y^n) := V\bigg(\frac{1}{n}\sum_{j=1}^n\delta_{y^j},y_i\bigg).
\]
For $y=(y^1,\ldots,y^n)$ it satisfies
\begin{align*}
y \cdot b(y) &= \sum_{i=1}^n y^i \cdot \bigg(V\bigg(\frac{1}{n}\sum_{j=1}^n\delta_{y^j},y^i\bigg) - V(\delta_0,0)\bigg) + \sum_{i=1}^n y^i \cdot V(\delta_0,0) \\
    &\le -\lambda\sum_{i=1}^n |y^i|^2 + \frac{\lambda}{2}\sum_{i=1}^n|y^i|^2 + \frac{nd}{2\lambda }|V(\delta_0,0)|^2 = -\frac{\lambda}{2}|y|^2 + \frac{nd}{2\lambda }|V(\delta_0,0)|^2,
\end{align*}
where the inequality used the montonicity assumption \ref{Assumption: UiT}(2). 
It is a standard consequence of this that $\sup_{t \ge 0}\E[|Y^1_t|^p] < \infty$ for each even integer $p > 0$. Indeed, recalling that $dY_t=b(Y_t)dt +\sqrt{2}\sigma dB_t$, applying Ito's formula yields
\begin{align*}
d|Y_t|^p &= \Big(p|Y_t|^{p-2}Y_t \cdot b(Y_t) + \sigma^2 p(p+nd-2) |Y_t|^{p-2}\Big)dt + \sqrt{2}\sigma p|Y_t|^{p-2}Y_t \cdot dB_t.
\end{align*}
By integrating (technically, up to the first exit time of $Y_t$ from the ball of radius $r$, and then sending $r\to\infty$), taking expectations, and applying the previous bound,
\[
\frac{d}{dt}\E\big[|Y_t|^p\big] \le -\frac{\lambda p}{2}\E\big[|Y_t|^p\big] + \Big(\frac{ndp}{2\lambda} +  \sigma^2 p(p+nd-2)\Big) \E\big[|Y_t|^{p-2}\big].
\]
By Young's inequality, for any $\epsilon > 0$ we have $\E|Y_t|^{p-2} \le \epsilon\E|Y_t|^{p} + c$, for a constant $c$ depending only on $\epsilon$ and $p$. By choosing $\epsilon$ appropriately, we get
\[
\frac{d}{dt}\E\big[|Y_t|^p\big] \le -\frac{\lambda p}{4}\E\big[|Y_t|^p\big] + C(1+dn),
\]
for a constant $C$ depending only on $(\lambda,p,\sigma)$. Apply Gronwall to get 
\[
\E\big[|Y_t|^p\big] \le e^{-\lambda p t/4}\E\big[|Y_0|^p\big] + C(1+dn)(1-e^{-\lambda p t/4}).
\]
Finally, by exchangeability,
\[
\E\big[|Y_t|^p\big]= \E\bigg[ \Big(\sum_{i=1}^n |Y^i_t|^2\Big)^{p/2}\bigg] \ge \E\sum_{i=1}^n \big[|Y^i_t|^p\big] = n\E\big[|Y^1_t|^p\big].
\]
Since $\E|Y_0^i| < \infty$, we finally deduce that $\sup_{t \ge 0}\E|Y^1_t|^p < \infty$ as claimed. The proof that $\sup_{t \ge 0}\E|X_t|^p < \infty$ is similar and simpler, so we omit the details; it could also be deduced from propagation of chaos and the preceding estimate on $\E\Big[|Y^1_t|^p\Big]$.

We next prove the claimed LSI, by showing that the LSI constant of $\mu_t$ becomes finite for $t$ large enough and then that its limsup as $t\to\infty$ is at most $2\sigma^2/\lambda$.
    Let $P_t^x=\Law(X_t\mid X_0=x)$ where $X_t$ is a solution to \eqref{eq: main SDE}. The dissipativity inequality \eqref{remark:dissipation} implies (e.g., by \cite[Theorem 4.1]{collet2008logarithmic}) 
    \[\frac{2\sigma^2(1-e^{-2\lambda t})}{\lambda}\]
     for every $x \in \R^d$.
    By the Markov property, $\mu_t=\int P_t^x\,d\mu_0(x)$, and we can apply  recent results of \cite{Chen2021-ch} on the LSI for mixture distributions. To do this, we need bounds for $D_2(P_t^x\,\|\,P_t^y)$.
    To obtain these in a sharp form, we apply the reverse transport inequality (essentially equivalent to a Harnack inequality) of \cite[Theorem 6.1(5)]{HarnackInequalities}, which states that\footnote{Strictly speaking, \cite[Theorem 6.1(5)]{HarnackInequalities} is stated for gradient drifts which are not time dependent, but their short proof via Girsanov's theorem given in Sections 4.1 and A.4.1 clearly generalize to the dissipative setting of \eqref{remark:dissipation}.}
    \begin{equation}
        D_2(P_t^x\,\|\,P_t^y)\leq \exp\Big\{\frac{\lambda}{e^{2t\lambda}-1} \cdot |x-y|^2\Big\}.
    \end{equation}
    We can now apply \cite[Theorem 1(2)]{Chen2021-ch} to deduce that $\mu_t$ satisfies an LSI with a constant $C_{\text{LS}}(t)$ bounded from above by  
    \begin{align*}
           12 &\frac{\sigma^2}{\lambda}(1-e^{-2{\lambda t}})\Big(2+\int_{\R^{2d}}D_2^2(P_t^x\,\|\,P_t^y)\,d\mu_0^{\otimes 2}(x,y)\Big)\Big(1+\log \int_{\R^{2d}}D_2^2(P_t^x\,\|\,P_t^y)\,d\mu_0^{\otimes 2}(x,y))\Big) .
\end{align*}
Recall that Assumption \ref{assumption on Phi}(2) implies that $\mu_0$ is subgaussian in the sense of \eqref{ineq:mu_0subgaussian}. In particular, letting $Z$ denote a standard Gaussian in dimension $d$,
for any  $c < 1/2C_{\textup{T}_1}(0)$ we have
\begin{align}
\int_{\R^{2d}} e^{\frac{c|x-y|^2}{2}}d\mu_0^{\otimes 2}(x,y) &= \E\Big[\int_{\R^{2d}} e^{ \sqrt{c} (x-y) \cdot Z}d\mu_0^{\otimes 2}(x,y)\Big] \le \E\Big[ e^{ C_{\textup{T}_1}(0) c |Z|^2}\Big] < \infty
\end{align}
Taking $T^*> (2\lambda)^{-1}\log(1+4\lambda  C_{\textup{T}_1}(0)  )$ yields  $C_{\text{LS}}(T^*)<\infty$.
Now that we have shown the LSI constant is finite at time $T^*$, we can appeal to known results on how the LSI constant of dissipative Langevin dynamics can be controlled by an interpolation between that of the initial distribution and that of the invariant measure.
Specifically, using \cite[Proposition 1.1]{TimeUniformLSI}, we obtain for every $t\geq T^*$ that
\[C_{\text{LS}}(t)\leq e^{-2\lambda (t-T^*)}C_{\text{LS}}(T^*)+ \frac{2\sigma^2}{\lambda}(1-e^{-2\lambda (t-T^*)}).\]
The right-hand side converges to $2\sigma^2/\lambda$ as $t\to\infty$, and the claim follows.
\end{proof}

We need a version of Lemma \ref{lm: projection path space} for the time marginals of the process. The following is a straightforward adaptation of the well-known derivation of the BBGKY hierarchy (e.g., \cite[Lemma 3.4]{SharpChaosLacker2023}), and we omit its proof.
\begin{lemma}
\label{lm: mimicking theorem marginal}
Suppose Assumption \ref{assumption on Phi} holds.
    The measure flow $(\pi^k_t)_{t \ge 0} \in C([0,\infty);\mathcal P((\R^d)^k))$ is a distributional solution of the Fokker-Planck PDE
    \begin{align*}
        \partial_t \pi_t^k(x)=-\sum_{i=1}^k \nabla_i \cdot \Big(\E[V(m_t^n,Y_t^i)\mid Y_t^{[k]}=x]\pi_t^k(x)\Big)+\sigma^2 \Delta \pi_t^k(x), \quad t > 0, \ x \in (\R^d)^k.
    \end{align*}
\end{lemma}
We now prove a version of Proposition \ref{pr: differential inequality short time} for time-marginal laws, which is better-suited  to uniform in time estimates.

\begin{proposition}
\label{pr: differential inequality UiT}
Suppose Assumption \ref{Assumption: UiT} holds.  
\begin{enumerate}[(1)]
\item There exist constants $C,a > 0$ independent of $n$ and $k$ such that \begin{equation}
    \label{eq: differential inequality uit suboptimal}
    \sup_{t \ge 0}\ent(\pi^k_t\,\|\,\mu^{\otimes k}_t)\leq C\frac{k^3}{n^2} + k \sup_{t \ge 0}\mathcal R(t) 
\end{equation}
for all $T \ge 0$, where we define
\begin{align*}
    \mathcal R(t) = \int_0^1\E\bigg[\Big|\int_{\R^d\times \R^d} \delta_m^2 V(sm_t^n+(1-s)\mu_t,Y_t^1,y,z)\,d(m_t^n-\mu_t)^{\otimes2}(y,z)\Big|^2\bigg]\,\,ds.
\end{align*}
\item If we further assume that $\sup_{t\geq 0}\ent(\pi^3_t\,\|\,\mu^{\otimes 3}_t)\leq M/n^2$ for some constant $M$ independent of $n$ and $k$, then
\begin{align}
\label{eq: differential inequality uit optimal}
    \sup_{t \ge 0}\ent(\pi^k_t\,\|\,\mu^{\otimes k}_t) \leq C\frac{k^2}{n^2} + k \sup_{t \ge 0}\mathcal R(t) ,
\end{align}
for all $T \ge 0$,
where again $C$ is a constant that does not depend on $n$ or $k$.
\end{enumerate} 
\end{proposition}
\begin{proof}
    We follow the steps of Proposition \ref{pr: differential inequality short time}, but work at the level of time-marginal entropies instead of path entropies. Let $1 \le k < n$. Using Lemma \ref{lm: mimicking theorem marginal} combined with a standard entropy estimate \cite[Lemma 3.1]{SharpChaosLacker2023} we obtain 
    \begin{align*}
        \frac{d}{dt}\ent&(\pi^k_t\,\|\,\mu^{\otimes k}_t)\leq-\frac{\sigma^2}{2}\fin (\pi_t^k\,\|\,\mu_t^{\otimes k})+  \frac{k}{2\sigma^2}\E\bigg[\Big | V(\mu_t,Y_t^1)-\E\Big[ V(m_t^n,Y_t^1)\mid Y_t^{[k]}\Big]\Big |^2\bigg].
    \end{align*}
     Let $\delta > 0$, to be chosen later.
     By Proposition \ref{pr: LSI}, we can take $\overline{T}$ large enough so that $\mu_t$ satisfies the LSI with constant $C_{\textup{LS}}\le 2(1+\delta)\sigma^2/\lambda$. 
     We apply the LSI to the Fisher information term, and bound the other term by following almost exactly the same steps of the proof of Proposition \ref{pr: differential inequality short time}, substituting the path entropies with their time marginal equivalent, and with one minor modification that helps us optimize the smallness condition from Assumption \ref{Assumption: UiT}. Specifically, we first follow the argument passing from  \eqref{eq: entropy in terms of L2} to \eqref{eq: entropy first bound}, combined with the aforementioned use of the LSI, to get
     \begin{align*}
         \frac{d}{dt}\ent(\pi^k_t\,\|\,\mu^{\otimes k}_t) &\le -\frac{2\sigma^2}{C_{\text{LS}}}\ent (\pi_t^k\,\|\,\mu_t^{\otimes k}) + \frac{k}{2(1-\epsilon)\sigma^2}\mathcal R(t) \\
         &\quad + \frac{k}{2\epsilon\sigma^2}\E\bigg[\Big|\E\Big[\,\frac{1}{n}\sum_{i=1}^n \delta_m  V(\mu_t,Y_t^1,Y_t^i)-\int \delta_m  V(\mu_t,Y_t^1,z)\,d\mu_t(z)\,\Big\vert\,Y^{[k]}[t]\Big]\Big|^2\bigg].
     \end{align*}
     Here $\mathcal R$ is defined as in  \eqref{def:remainder1}, and we have used the real variable inequality $(a+b)^2 \le \epsilon^{-1}a^2+(1-\epsilon)^{-1}b^2$ with $\epsilon \in (0,1)$ is to be chosen later to be close to 1; this is the aforementioned minor modification of the arguments of Proposition \ref{pr: differential inequality short time}. The term on the last line is bounded by
     \begin{align}
\label{eq: k^3 equation uit}
\begin{split}
&\frac{k}{2\epsilon(1-\epsilon)\sigma^2 n^2}\E\bigg[\Big |\sum_{i=1}^k  \delta_m V(\mu_t,Y_t^1,Y_t^{i})-\int_{\R^d}\delta_m V(\mu_t,Y_t^1,z)\,d\mu_t(z) \Big |^2\bigg]\\
    &+  \frac{k}{2\sigma^2\epsilon^2 n^2}\E\bigg[\bigg | \sum_{i=k+1}^n\bigg(\E\Big[ \delta_m V(\mu_t, Y_t^1,Y_t^{i})\,\Big\vert\, Y^{[k]}[t]\Big]-\int_{\R^d}\delta_m V(\mu_t,Y_t^1,z)\,d\mu_t(z)\bigg)\bigg |^2\bigg].
\end{split}
\end{align}
The first term is bounded by $Ck^3/n^2$ thanks to the moment bounds from Lemma \ref{pr: LSI}, and the second is bounded as in \eqref{eq: change of measure}--\eqref{eq: term 1 end} by
\[
C_{\text{T}_1}(t)\frac{k\|\wgrad V\|_\infty^2}{\sigma^2\epsilon^2 }\bigg(\ent(\pi^{k+1}_t\,\|\,\mu^{\otimes k+1}_t)-\ent(\pi^k_t\,\|\,\mu^{\otimes k}_t)\bigg).
\]
Putting it together, and using the fact that $C_{\text{T}_1}(t)\leq  C_{\text{LS}}/2$ by the Otto-Villani theorem \cite{OttoVillaniLSITalagrand}, we get
    \begin{align*}
        \frac{d}{dt}\ent(\pi^k_t\,\|\,\mu^{\otimes k}_t)&\leq-\frac{2\sigma^2}{C_{\text{LS}}} \ent(\pi^k_t\,\|\,\mu^{\otimes k}_t) + C\frac{k^3}{n^2} + Ck\mathcal R(t) \\
        &\quad+\frac{kC_{\text{LS}}\|\wgrad V\|_\infty^2}{2\epsilon^2\sigma^2}\bigg(\ent(\pi^{k+1}_t\,\|\,\mu^{\otimes k+1}_t)-\ent(\pi^k_t\,\|\,\mu^{\otimes k}_t)\bigg)
    \end{align*}
    This was valid for $1\le k < n$, and for $k=n$ we have the same bound except that the last term vanishes.
By Theorem \ref{th: short time}, for each $1 \le k \le n$, we have $\ent_{\overline T}(k)\leq C\frac{k^2}{n^2}$ for a constant that depends on $\overline T$ only. 
Use the estimate from Lemma \ref{lm: linear differential inequality} with $s=\overline{T}$ to find, for $T \ge \overline{T}$,
\begin{align*}
    \ent_T(k)&\leq  C\frac{k^2}{n^2} \exp\Big\{(T-\overline T)\cdot \big(\frac{2C_{\text{LS}}\|\wgrad V\|_\infty^2}{2\epsilon^2\sigma^2}-\frac{2\sigma^2}{ C_{\text{LS}}}\big)\Big\} \\
    &\quad+C\frac{k^3}{n^2}\int_{\overline{T}}^T\exp\Big\{(t-\overline T)\cdot \big(\frac{3C_{\text{LS}}\|\wgrad V\|_\infty^2}{2\epsilon^2\sigma^2}-\frac{2\sigma^2}{ C_{\text{LS}}}\big)\Big\} \,dt\\
    &\quad +C\int_{\overline T}^T \exp\Big\{(t-\overline T)\cdot \big(\frac{C_{\text{LS}}\|\wgrad V\|_\infty^2}{2\epsilon^2\sigma^2}-\frac{2\sigma^2}{ C_{\text{LS}}}\big)\Big\}\Big(\frac{k^3}{n^2}+k \mathcal R(t)\Big)\,dt
\end{align*}
Now, we claim that  $\delta>0, \epsilon \in (0,1)$ may be chosen so that
\[
a := \frac{2\sigma^2}{ C_{\text{LS}}} - \frac{3C_{\text{LS}}\|\wgrad V\|_\infty^2}{2\epsilon^2\sigma^2} > 0.
\]
Recalling that $C_{\textup{LS}}\le 2(1+\delta)\epsilon^2\sigma^2/\lambda$, it suffices to have
\[
\frac{(1+\delta)^2}{\epsilon^2} \frac{\|\wgrad V\|_\infty^2}{\lambda^2} < \frac13.
\]
By Assumption \ref{Assumption: UiT}, we have $\|\wgrad V\|_\infty^2 < \lambda^2/3$, and thus we may indeed take $\delta$ small and $\epsilon$ close to 1 make $a > 0$.
We have thus obtained
\begin{align*}
    \ent(\pi_T^k\,\|\,\mu_T^{\otimes k})\leq C\frac{k^2}{n^2} + C\int_{\overline T}^{T}e^{-a(t-\overline T)}\Big(\frac{k^3}{n^2}+k \mathcal R(t)\Big)\,dt.
\end{align*}
To prove (2) under the additional assumption on $\ent(\pi_t^3\,\|\,\mu_t^{\otimes 3})$, we adapt the proof of Proposition \ref{pr: differential inequality short time}(2) similarly; we omit the details.
\end{proof}

Similarly to the short time analysis, we need to estimate the remainder in order to conclude the proof. The difference is that we require a uniform in time bound.
\begin{proposition}
\label{pr: remainder UIT}
    If we assume \ref{Assumption: UiT}, then there exists a constant $C$ such that
    \[\sup_{t\geq 0} \mathcal R(t)\leq \frac{C}{n^2}.\]
\end{proposition}

The proof of Proposition \ref{pr: remainder UIT} is given in Section \ref{sect: remainder}. Taking it for granted, we can now finish the proof of Theorem \ref{th: UiT}.

\begin{proof}[Proof of Theorem \ref{th: UiT}]
Use Proposition \ref{pr: remainder UIT} along with the inequality from Proposition \ref{pr: differential inequality UiT}(1) to obtain  $\sup_{t \ge 0}\ent (\pi^k_t\,\|\,\mu^{\otimes k}_t) \leq Ck^3/n^2$ for all $1 \le k \le n$, 
    where $C$ does not depend on $k$ or $n$. As a result, $\sup_{t \ge 0}\ent (\pi^3_t\,\|\,\mu^{\otimes 3}_t)\leq C/n^2$, and thus Proposition \ref{pr: differential inequality UiT}(2) applies. Combine Proposition \ref{pr: differential inequality UiT}(2) with Proposition \ref{pr: remainder UIT} to get  $\sup_{t \ge 0}\ent (\pi^k_t\,\|\,\mu^{\otimes k}_t) \leq Ck^2/n^2$ for all $1 \le k \le n$, which completes the proof.
\end{proof}

We check that our uniform in time estimates imply the same rates in Wasserstein distance.

\begin{proof}[Proof of Corollary \ref{cr: metric convergence UiT}]
The total variation bound follows immediately by applying Pinsker's inequality. Let $\overline T$ be the threshold time from Proposition \ref{pr: LSI}. Up to $\overline T$ we can invoke Corollary \ref{cr: metric convergences short time} since $\sup_{t\leq \overline T}\W_2^2(\pi^k_t,\mu_t^{\otimes k})\leq \W_{2,\overline T}^2(\pi^k\left[\overline T\right],\mu^{\otimes k}\left[\overline T\right])=O(k^2/n^2)$. For every $t>\overline T$, $\mu_t$ satisfies the LSI by Proposition \ref{pr: LSI}, and in particular $\mu_t^{\otimes k}$ satisfies the LSI with the same constant. By the Otto-Villani theorem \cite{OttoVillaniLSITalagrand}, $\mu_t^{\otimes k}$ satisfies the transport cost inequality
\begin{align*}
    \W_2^2(\pi_t^k,\mu_t^{\otimes k})\leq C_{\text{LS}} \ent(\pi_t^{k}\,\|\,\mu_t^{\otimes k}). 
\end{align*}
Now the uniform in time bound follows by Theorem \ref{th: UiT}.
\end{proof}

\section{Flows and semigroups over $\mathcal P_2(\R^d)$}
\label{sect: weak chaos}

In this section we begin our analysis of the remainder $\mathcal R(t)$ defined in \eqref{def:remainder1}. Our approach will use techniques developed in \cite{Buckdahn2017-op} and \cite{Chassagneux2022-jt}. 
The strategy can be roughly summarized as follows. We show that the desired estimate $\mathcal{R}(T) = O(1/n^2)$ holds when $m^n_T$ is replaced by $m^n_0$, which is an empirical measure of iid random variables. This exploits the smoothness of $V$ and the dependence of $\mathcal R(t)$ on the quadratic form $(m^n_t-\mu_t)^{\otimes 2}$. Then, we use a semigroup-type analysis (rather than directly differentiating $\mathcal R(t)$ in $t$) to show that this $1/n^2$ estimate propagates over time. Special care is needed to exploit dissipativity for the uniform in time case.

It is worth noting that in dimension $d=1$, the machinery of this section is not needed, and a much simpler argument is possible; see Remark \ref{rem:d=1}.
 
We begin by recalling some results of \cite{Buckdahn2017-op} on the differentiability of the solution of the McKean-Vlasov equation \eqref{eq: main SDE} with respect to its space and measure arguments. We work with a filtered probability space $(\Omega,\mathcal A, \mathcal F,\P)$ that we assume is rich enough to support a random variable with law $\nu$ for every $\nu \in \mathcal P(\R^d)$; we denote such random variable by $[\nu]$. Consider the decoupled system of SDEs from \cite{Buckdahn2017-op}:
\begin{align}
\begin{split}
\label{eq: linearized process}
    X^{s,\nu}_T&=[\nu]+\int_s^T  V(\Law(X^{s,\nu}_t),X^{s,\nu}_t)\,dt+\sqrt 2\sigma (W_T-W_s)\,;\\
X^{x,s,\nu}_T&=x+\int_s^T V(\Law(X^{s,\nu}_t),X^{x,s,\nu}_t)\,dt+\sqrt 2 \sigma (B_T-B_s),
\end{split}
\end{align}
where $B,W$ are independent Brownian motions. 
As $V$ is Lipschitz, these SDEs have unique strong solutions. Denote $\mu_t^{s,\nu}=\Law(X^{s,\nu}_t)$. We are interested in studying the SDE $X^{0,\mu_0}$ and thus we use the shorthand $\mu_t=\mu_t^{0,\mu_0}=\Law(X_t^{0,\mu_0})$ where $\mu_0$ is the original law of $X_0$ from \eqref{eq: main SDE}. Notice that while the value of $X_t^{s,\nu}$ depends on the random variable $[\nu]$, its law $\mu_t^{s,\nu}$ does not. Denote also
\[\mu_t^{x,s,\nu}= \Law(X_t^{x,s,\nu}).\]
We will make some use of the time-homogeneity property,
\begin{equation} \label{eq:MVtimehomo}
\mu^{s,\nu}_t = \mu^{0,\nu}_{t-s}, \qquad \mu^{x,s,\nu}_t = \mu^{x,0,\nu}_{t-s}, \quad t \ge s \ge 0,
\end{equation}
as well as the flow property, implied by uniqueness of the dynamics:
\begin{equation} \label{eq:MVflow}
\mu^{s,\mu^{r,\nu}_s}_t = \mu^{r,\nu}_{t}, \quad t \ge s \ge r \ge 0.
\end{equation}
We record some classical moment bounds, whose proof is omitted.
\begin{lemma} \label{lem:flow moment bounds}
    For each $p\geq 1$ and $T\ge t \ge s\ge 0$, there exists constant $C$ depending on $p,T,s$ and $L$ such that 
    \begin{align*}
        \int_{\R^d}|x|^p\,d\mu_t^{y,s,\nu}(x)&\leq C (1+|y|^p);\\
        \int_{\R^d}|x|^p\,d\mu^{s,\nu}_t(x)&\leq C+C\int_{\R^d} |x|^p\,d\nu(x).
    \end{align*}
\end{lemma}

In what follows we will need multiple independent copies of the processes $X_t^{x,s,\nu}$ and $X_t^{[\nu],s,\nu}$; we will denote them by $X_{t,(i)}^{x,s,\nu}$ and $X_{t,(i)}^{[\nu],s,\nu}$ for $i=1,\dots,6$. To be clear, we assume that these copies are started from independent initial conditions and driven by independent Brownian motions. We will denote by $\E_{(k)}$ the operation of taking expectation with respect to the first $1,\dots,k$ copies only (and not with respect to the original process). Since $ V$ is Lipschitz and differentiable in space, it is classical that the map $x\mapsto X_t^{x,s,\nu}$ admits a derivative with dynamics
\begin{align}
\label{eq: grad of SDE}d\nabla_{x}X^{x,s,\nu}_t&=\nabla_x V(\mu^{s,\nu}_t,X^{x,s,\nu}_t) \nabla_x X_t^{x,s,\nu}\,dt, \quad \nabla_x X_s^{x,s,\nu}=I.
\end{align}
One of the contributions of \cite{Buckdahn2017-op}, specifically Theorem 4.1 and Remark 4.5 therein, is that the map $ L^2(\P;\mathcal F_s)\ni \xi \mapsto X_t^{x,s,\Law(\xi)}\in L^2(\P)$ (where $L^2(\P;\mathcal F_s)$ is the space of square-integrable $\mathcal F_s$-measurable random variables) also admits a Fréchet derivative of the form
\begin{align*}
    D_{\xi}X_t^{x,s,\Law(\xi)}(\eta)=\E_{(1)}\big[\wgrad X_t^{x,s,\Law(\xi)}(\xi)\cdot \eta_{(1)}\big]
\end{align*}
where for each $y\in \R^d$, $\wgrad X_t^{x,s,\nu}(y)$ is the unique solution to the SDE 
\begin{align}
\frac{d}{dt}\wgrad X^{x,s,\nu}_t(y)&=\nabla_x V(\mu_t^{s,\nu},X^{s,x,\nu}_t)\wgrad X_t^{x,s,\nu}(y) +\E_{(1)}\Big[\wgrad  V(\mu_t^{s,\nu},X_t^{x,s,\nu}, X_{t,(1)}^{y,s,\nu})\nabla_x X_{t,(1)}^{y,s,\nu}\Big] \nonumber \\
&+ \E_{(1)}\Big[\wgrad  V(\mu_t^{s,\nu},X_t^{s,\nu}, X_{t,(1)}^{s,\nu})\wgrad X_{t,(1)}^{s,\nu}(y)\Big] , \quad \wgrad X_s^{x,s,\nu}=0. \label{eq: wgrad of SDE point}
\end{align}
and $\wgrad X_t^{s,\nu}$ is given by the dynamics 
\begin{align}
\frac{d}{dt}\wgrad X^{s,\nu}_t(y)&=\nabla_x V(\mu_t^{s,\nu},X^{s,\nu}_t)\wgrad X^{s,\nu}_t(y) +\E_{(1)}\Big[\wgrad  V(\mu_t^{s,\nu},X_t^{s,\nu}, X_{t,(1)}^{y,s,\nu})\nabla_x X_{t,(1)}^{y,s,\nu}\bigg]  \nonumber \\
&+ \E_{(1)}\bigg[\wgrad  V(\mu_t^{s,\nu},X_t^{s,\nu}, X_{t,(1)}^{s,\nu})\wgrad X_{t,(1)}^{s,\nu}(y)\bigg] , \ \ \ \wgrad X_s^{s,\nu}=0. \label{eq: wgrad of SDE}
\end{align}
In particular, like \cite{Buckdahn2017-op} we consider $\wgrad X_t^{x,s,\nu}$ the derivative over $\mathcal P_2(\R^d)$ of the map $\nu\mapsto X_t^{x,s,\nu}$.
It is important to note that $\wgrad X_t^{s,\nu}(y)=\wgrad X_t^{x,s,\nu}(y)\mid _{x=[\nu]}$ almost surely by strong uniqueness. 
We will differentiate \eqref{eq: wgrad of SDE point} again with respect to both the measure and space variables given our assumptions on the smoothness of $ V$.
Recall  from Definition \ref{def:multiindex} the notation for multi-indices $(\mathsf{w},\boldsymbol{i})$, which we now begin to use.
To shorten notation, we define 
\begin{equation} \label{def:C4notation}
\|X^s_t\|_{\mathcal{C}^4}=\sup_{1\leq \mathsf w +|\boldsymbol{j}|\leq| 4|}\sup_{\boldsymbol{y},x,\nu}\|D^{\mathsf w,\boldsymbol{j}}X_t^{x,s,\nu}(\boldsymbol{y})\|_{L^\infty(\P)}, \qquad t > s \ge 0.
\end{equation}
The supremum can be taken to be plus infinity if the process does not admit a derivative. 
We will make implicit use of the time-homogeneity property, that $\|X^s_t\|_{\mathcal{C}^4}$ depends on $t$ and $s$ only through $t-s$:
\[
\|X^s_t\|_{\mathcal{C}^4} = \|X^0_{t-s}\|_{\mathcal{C}^4}, \qquad  t > s \ge 0.
\]
We record some differentiability results in the following lemma, taken from \cite[Theorem 3.4]{Chassagneux2022-jt} and \cite[Theorem 3.2]{Crisan2017Smoothing}.
\begin{lemma}
\label{lm: estimates on flows}
    The map $(x,\nu) \mapsto X_t^{x,s,\nu}$ is $6$ times differentiable in its space and measure arguments. Moreover, there exist constants $C,b>0$ such that 
    \[\|X_t^{0}\|_{\mathcal C^4}\leq Ce^{bt}.\]
\end{lemma}
A key object of the analysis will be the operator $P_t$ defined as follows for an arbitrary function $f:\mathcal P_2(\R^d)\to \R$:
\[
    P_tf (\nu)= f(\mu_t^{0,\nu}).
\]
This is the semigroup associated with the McKean-Vlasov equation.
We collect a simple but necessary estimate.

\begin{lemma} 
\label{lm: bounds onf wgrad semigroup}
 Let $f:\mathcal P_2(\R^d)\to \R$ be in $C_{\textup{bd}}^4(\mathcal P_2(\R^d))$ and let $(\mathsf w,\boldsymbol{i})$ be a multi-index with size $\mathsf w+|\boldsymbol{i}|=k\leq 4$. Assume that for some constants $A \ge 0$ and $p\geq 1$,
 \begin{equation*}
     |D^{\mathsf w,\boldsymbol{i}} f(\nu,\boldsymbol{y})|\leq A\Big(1+ |\boldsymbol{y}|^p+\int_{\R^d}|z|^p\,d\nu(z)\Big).
 \end{equation*}
    Under Assumption \ref{assumption on Phi} there exists a constant $C>0$ depending only on $A,p$ such that 
    \begin{align*}
        |D^{\mathsf w,\boldsymbol{i}} P_t f(\nu,\boldsymbol{x})|&\leq \|X^0_t\|_{\mathcal{C}^4}C\big(1 \vee \|X^0_t\|_{\mathcal{C}^4}\big)^{k-1} \Big(1+\int_{\R^d} |z|^p\,d\mu_t^{0,\nu}(z)+\sum_{j=1}^{\boldsymbol{\mathsf w }}\int_{\R^d} |z|^p\,d\mu_t^{x^j,0,\nu}(z)\Big)
    \end{align*}
\end{lemma}
\begin{proof} 
    For $k=1$, by the derivation rules established in \cite[Lemma 6.1]{Buckdahn2017-op}, Jensen inequality and the growth assumption,
    \begin{align*}
        |\wgrad P_t f(\nu,y)|&= \Big|\E[\wgrad f(\mu_t^{0,\nu}, X_t^{y,0,\nu}) \nabla_x X_t^{y,0,\nu}+\wgrad f(\mu_t^{0,\nu},X_t^{0,\nu}) \wgrad X_t^{0,\nu}(y) ]\Big|\\
        &\leq \|X^0_t\|_{\mathcal{C}^4}\E\big[|\wgrad f(\mu_t^{0,\nu},X_t^{y,0,\nu})|+|\wgrad f(\mu_t^{0,\nu},X_t^{0,\nu})|\big]\\
        &\leq 2\|X^0_t\|_{\mathcal{C}^4}A\E\Big[1+|X_t^{y,0,\nu}|^p+|X_t^{0,\nu}|^p+\int_{\R^d}|z|^p\,d\mu_t^{0,\nu}(z)\Big]\\
        &\leq 4\|X^0_t\|_{\mathcal{C}^4}A\Big(1+\int_{\R^d}|z|^p\,d(\mu_t^{0,\nu}+\mu_t^{y,0,\nu})(z)\Big)
    \end{align*}
    For higher order terms, we simply repeat this argument after observing that differentiating again leads to second order derivatives of $X_t$, namely \[ \nabla_x X_t^{\cdot,s,\nu},\wgrad X_t^{\cdot,s,\nu}(\cdot), \wgrad \nabla_x X_t^{\cdot,s,y}(\cdot), \wgrad^2 X_t^{\cdot,s,\nu}(\cdot,\cdot)\]
    multiplied by the first and second derivatives of $f$.
\end{proof}
The point of Lemma \ref{lm: bounds onf wgrad semigroup} is that the Wasserstein semigroup (and its derivatives) preserve polynomial growth estimates, quantitatively. The lemma is stated for functions with bounded derivatives, so the polynomial growth bound assumed for $D^{\mathsf w,\boldsymbol{i}} f$ is redundant in some sense; the point, however, is that the constant $C$ depends only on $A$ and not on any uniform bounds of the derivatives of $f$. This will be used within an approximation argument later, when an unbounded function is approximated by bounded ones, and it will be crucial that this $C$ does not blow up as we pass to the limit. 

The remainder of this section is devoted to uniform-in-time estimates of the derivatives of the process $X_t$, which are much more involved. This will be used to obtain uniform in time estimates for the semigroup. As opposed to the PDE approach of \cite{UiTWeakChaos}, we adopt a more probabilistic perspective, leveraging the notion of displacement monotonicity. 
\begin{lemma} \label{lem:monotonicity}
    Under Assumption \ref{Assumption: UiT}, for square-integrable random variables $X$ and $Y$, 
    \begin{equation}
    \label{eq: wasserstein gradient is PD}
        \E[Y^\top \wgrad V(\Law(X),X,X_{(1)})Y_{(1)}+Y^\top \nabla_x V (\Law(X),X)Y ]\leq -\lambda \E[|Y|^2],
    \end{equation}
    where $(X_{(1)},Y_{(1)})$ are an independent copy of $(X,Y)$.
\end{lemma}
\begin{proof}
    Starting from the definition of monotonicity from Assumption \ref{Assumption: UiT}
    \[\E[(V(\Law(X),X)-V(\Law(Z),Z))\cdot (X-Z)]\leq -\lambda \E[|X-Z|^2]\]
    take $Z=X+hY$ where $Y\in L^2(\Omega)$ to obtain
     \[\frac 1h\E[(V(\Law(X+hY),X+hY)-V(\Law(X),X))\cdot Y]\leq-\lambda \E[|Y|^2]\]
Now \cite[Proposition 5.85]{Carmonabook1} asserts the existence of the derivative
\begin{align*}
    \frac{d}{dh}V(\Law(X+hY),X+hY)\Big \vert_{h=0}=\E_{(1)}[\wgrad V(\Law(X),X,X_{(1)})Y_{(1)}]+ \nabla_x V(\Law(X),X)Y,
\end{align*}
and we use  dominated convergence theorem to conclude the proof.
\end{proof}

Using the dissipativity from Remark \ref{remark:dissipation}, we can deduce the following standard uniform-in-time moment bounds. We omit the proof.
\begin{lemma} \label{lem:flow moment bounds UIT}
    Under assumption \ref{Assumption: UiT}, for each $p\geq 1$ there exists a constant $C$ such that, for every $s \ge 0$, 
    \begin{align*}
        \sup_{t\geq s}\int_{\R^d}|x|^p\,d\mu_t^{y,s,\nu}(x)&\leq C+C |y|^p;\\
        \sup_{t\geq s}\int_{\R^d}|x|^p\,d\mu_t^{s,\nu}&\leq C+C \int_{\R^d}|x|^p\,d\nu(x). 
    \end{align*}
\end{lemma}

The following proposition now gives the key uniform-in-time analog of Lemma \ref{lm: estimates on flows}. 

\begin{proposition}
\label{pr: decay of wgrad X}
    There exist constants $C,b>0$ such that
    $\|X_t^s\|_{\mathcal C^4}\leq Ce^{-b(t-s)}$ for all $t \ge s \ge 0$, where $\|X_t^s\|_{\mathcal C^4}$ was defined in \eqref{def:C4notation} .
\end{proposition}
\begin{proof}
We write the proof for $d=1$ to simplify notation, so that all of the $\nabla_x$ symbols below are really just univariate derivatives. Differentiability of the map $\nu\mapsto X_t^{x,s,\nu}$ and $x\mapsto \wgrad^{k} X_t^{x,s,\nu}$ was discussed above.
In what follows, we take $b$ and $C$ to be positive constants that can change from line to line but do not depend on time or the space and measure variables.

The proof goes by induction on the number of derivatives. We begin by estimating $\nabla_x^k X_t^{x,s,\nu}$ for $k=1$, which we recall satisfies the ODE \eqref{eq: grad of SDE}. Use the chain rule and the fact that $\nabla_x V$ is negative definite by Remark \ref{rk: UiT LSI}:
    \[\frac{d}{dt}|\nabla_x X_t^{x,s,\nu}|^2= 2  \nabla_x V (\mu_t^{s,\nu},X_t^{x,s,\nu}) (\nabla_x X_t^{x,s,\nu})^2 \leq -2\lambda |\nabla_x X_t^{x,s,\nu}|^2 ,\]
    which by Gronwall yields $|\nabla_x X_t^{x,s,\nu}|^2\leq e^{-2\lambda(t-s) }$. For the second derivative, by differentiating again we find 
    \begin{align*}
    \frac{d}{dt}\nabla_x^2 X_t^{x,s,\nu}&= \nabla_x V(\mu_t^{s,\nu},X_t^{x,s,\nu})\nabla_x^{2}X_t^{x,s,\nu} +\nabla^2_xV(\mu_t^{s,\nu},X_t^{x,s,\nu})(\nabla_x X_t^{x,s,\nu})^2 ;\\
    \nabla_x^2 X_s^{x,s,\nu}&=0.
    \end{align*}
    The second term decays exponentially, by the preceding estimate on $\nabla_x^k X_t^{x,s,\nu}$ and the boundedness of  $\nabla_x^2V$. The first term provides dissipativity as in the $k=1$ case; applying the chain rule to compute $d|\nabla_x^2 X_t^{x,s,\nu}|^2$ and then using Gronwall yields again $|\nabla_x^2 X_t^{x,s,\nu}|\leq Ce^{-b(t-s)}$.
    The higher order derivatives $\nabla_x^k X_t^{x,s,\nu}$ are estimated inductively following a similar pattern: the highest order term will hit the function $\nabla_x V$ which provides dissipativity, and the higher derivatives of $V$ appear only multiplied by lower order derivatives of $(\nabla_x^j X_t^{x,s,\nu})_{j < k}$.
    
   The reasoning for the Wasserstein derivatives is similar, but more elaborate: First consider $\wgrad X_t^{s,\nu}(y)$, recalling the dynamics from \eqref{eq: wgrad of SDE}. Check that 
     \begin{align*}
     \frac 12 \frac{d}{dt}\E[|\wgrad X_t^{s,\nu}(y)|^2]&=\E[ \wgrad X_t^{s,\nu}(y)\wgrad V(\mu^{s,\nu}_t,X_t^{s,\nu}, X_{t,(1)}^{s,\nu}(y))\wgrad  X_{t,(1)}^{s,\nu}(y)]\ \\
     &\quad +\E[\wgrad X_t^{s,\nu}(y) \nabla_xV(\mu^{s,\nu}_t,X_t^{s,\nu})\wgrad X_t^{s,\nu}(y)] \\
     &\quad +\E[\wgrad X_t^{s,\nu}(y) \wgrad V(\mu_t^{s,\nu},X_t^{[\nu]}, X_{t,(1)}^{y,s,\nu})\nabla_x X_{t,(1)}^{y,s,\nu}] \\
        &\leq -\frac\lambda 2 \E[|\wgrad X_t^{s,\nu}(y)|^2] +\frac{1}{2\lambda}\E[|\wgrad V(\mu_t^{s,\nu},X_t^{s,\nu}, X_{t,(1)}^{y,s,\nu})\nabla_x X_{t,(1)}^{y,s,\nu}|^2] .
    \end{align*}
    where we use Lemma \ref{lem:monotonicity} to control the first two terms and  Young's inequality to control the third. Our previous argument shows that $|\nabla_x X_t^{x,s,\nu}| \le e^{-2\lambda (t-s)}$, and we deduce from Gronwall and boundedness of $\wgrad V$ that $\E[|\wgrad X_t^{s,\nu}(y)|^2]\leq C e^{-b(t-s)}$. 
    Next, recalling the dynamics \eqref{eq: wgrad of SDE point} for the linearized process $\wgrad X_t^{x,s,\nu}(y)$, we obtain a similar bound for every $x$ after computing
    \begin{align*}
        \frac12 \frac{d}{dt}|\wgrad X_t^{x,s,\nu}(y)|^2&= \wgrad X_t^{x,s,\nu}(y) \E_{(1)}[\wgrad V(\mu_t^{s,\nu},X_t^{x,s,\nu}, X_{t,(1)}^{s,\nu}) \wgrad X_{t,(1)}^{s,\nu}(y)] \\
        &\quad + \wgrad X_t^{x,s,\nu}(y)\nabla_x V(\mu_t^{s,\nu},X_t^{x,s,\nu}) \wgrad X_t^{x,s,\nu}(y) \\
        &\quad + \wgrad X_t^{x,s,\nu}(y) \E_{(1)}[\wgrad V(\mu_t^{s,\nu},X_t^{x,s,\nu}, X_{t,(1)}^{x,s,\nu})\nabla_x X_{t,(1)}^{y,s,\nu}] \\
        &\leq C\E[|\wgrad X_t^{s,\nu}(y)|^2] -\frac{\lambda}{2}|\wgrad X_t^{x,s,\nu}(y)|^2 +C\E[|\nabla_x X_t^{y,s,\nu}|^2] \\ 
        &\leq Ce^{-\lambda(t-s)/2}-\frac{\lambda}{2}|\wgrad X_t^{x,s,\nu}(y)|^2+Ce^{-(t-s) \lambda} .
    \end{align*}
    Here we used $\nabla_x V \le -\lambda$ from Remark \ref{rk: UiT LSI} along with  Young's inequality, boundedness of $\wgrad V$, and the bound on $\E[|\wgrad X_t^{s,\nu}(y)|^2]$ just established. We deduce that   $|\wgrad X_t^{x,s,\nu}(y)|^2\leq Ce^{-b(t-s)}$ for some $b>0$ for every $x$; notice that this pathwise estimate holds almost surely.
    
    We continue to higher order derivatives. Recall that $\wgrad \nabla_x X_t^{x,s,\nu}(y)=\nabla_x \wgrad X_t^{x,s,\nu}(y)$. The dynamics of $\nabla_x \wgrad X_t^{x,s,\nu}(y)$ are given by (see \cite[Lemma 5.3]{Buckdahn2017-op}):
    \begin{align*}
        \frac{d}{dt}\nabla_x\wgrad X_t^{x,s,\nu}(y) =  \ &\E_{(1)}[\nabla_x\wgrad V(\mu_t^{s,\nu},X_t^{x,s,\nu}, X_{t,{(1)}}^{s,\nu}) \nabla_x X_t^{x,s,\nu}\wgrad  X_{t,(1)}^{s,\nu}(y)] \\
        &+\nabla_x V(\mu_t^{s,\nu},X_t^{x,s,\nu}) \nabla_x\wgrad  X_t^{x,s,\nu}(y) \\
        &+\nabla_x^2 V(\mu_t^{s,\nu},X_t^{x,s,\nu}) \nabla_x X_t^{x,s,\nu}\wgrad X_t^{x,s,\nu}(y) \\
        &+\E_{(1)}[\nabla_x\wgrad V(\mu_t^{s,\nu},X_t^{x,s,\nu}, X_{t,(1)}^{y,s,\nu})\nabla_x X_{t,(1)}^{y,s,\nu}\nabla_x X_t^{x,s,\nu}] .
    \end{align*}
    Notice that $\nabla_x\wgrad X_t^{x,s,\nu}(y)$ appears on the right-hand side only in the $\nabla_xV$ term. The remaining terms all contain higher derivatives of $V$, which are uniformly bounded by assumption, multiplied by lower-order derivatives of $X_{t,(1)}^{s,\nu}(y)$ and $X_t^{x,s,\nu}(y)$ which we have already shown to decay exponentially in time. Thus, using Young's inequality and these lower-order bounds, we get
    \[
    \frac{d}{dt}|\nabla_x\wgrad X_t^{x,s,\nu}(y)|^2 \le -\frac{\lambda}{2}|\nabla_x\wgrad X_t^{x,s,\nu}(y)|^2 + Ce^{-b(t-s)},
    \]
    which yields $|\nabla_x\wgrad X_t^{x,s,\nu}(y)|^2 \le Ce^{-b(t-s)}$.
    The higher orders can be checked inductively following the same scheme that we now sketch: Differentiating again leads us to 
 \begin{align*}
     \frac{d}{dt}\wgrad^2 X_t^{x,s,\nu}(y,z) = \ &\E_{(1)}[\wgrad V(\mu_t^{s,\nu},X_t^{x,s,\nu},X_{t,(1)}^{x,s,\nu}) \wgrad^2  X_t^{x,s,\nu}(y,z)] \\
     &+\nabla_x V(\mu_t^{s,\nu},X_t^{x,s,\nu}) \wgrad^2 X_t^{x,s,\nu}(y,z) \\
     &+\E_{(1)}\bigg[\wgrad\bigg(\wgrad V(\mu_t^{s,\nu},X_t^{x,s,\nu},X_{t,(1)}^{x,s,\nu})\bigg)(z) \wgrad  X_t^{x,s,\nu}(y)\bigg] \\
     &+\wgrad\bigg(\nabla_x V(\mu_t^{s,\nu},X_t^{x,s,\nu})\bigg)(z) \wgrad X_t^{x,s,\nu}(y) \\
     &+\wgrad \E_{(1)}[\wgrad V(\mu_t^{s,\nu},X_t^{x,s,\nu},X_{t,(1)}^{y,s,\nu}) \nabla_x X_{t,(1)}^{y,s,\nu}] .
 \end{align*} 
Notice that we are not fully developing the derivatives in the last three lines.
The first two terms provide dissipativity by using Lemma \ref{lem:monotonicity} after computing $(d/dt)|\wgrad^2 X_t^{x,s,\nu}(y,z)|^2$, similarly to before, while the remaining terms involve lower order derivatives of $X_t^{x,s,\nu}$, which we have already shown to be decaying exponentially,
multiplied by derivatives of $V$ which are bounded by assumption.  This yields $|\wgrad^2 X_t^{x,s,\nu}(y,z)| \leq C e^{-b (t-s)}$. Let us give one more example of the pattern, with the cross derivatives like $\nabla_y \wgrad X_t$:
 \begin{align*}
     \frac{d}{dt}\nabla_y \wgrad X_t^{x,s,\nu}(y) = \ &\E_{(1)}[\wgrad V(\mu_t^{s,\nu},X_t^{x,s,\nu},X_{t,(1)}^{s,\nu})\nabla_y\wgrad X_{t,(1)}^{s,\nu}(y)] \\
     &+\nabla_x V(\mu_t^{s,\nu},X_t^{x,s,\nu})\nabla_y\wgrad X_t^{x,s,\nu}(y) \\
     &+ \E[\nabla_2\wgrad V(\mu_t^{s,\nu},X_t^{x,s,\nu},X_{t,(1)}^{y,s,\nu})|\nabla_x X_{(1)}^{y,s,\nu}|^2] \\
     &+ \E[\wgrad V(\mu_t^{s,\nu},X_t^{x,s,\nu},X_{t,(1)}^{y,s,\nu})\nabla^2_x X_{(1)}^{y,s,\nu}]
 \end{align*}
since the last two terms decay exponentially, we can use the same calculations as above, using first $x=[\nu]$ and Lemma \ref{lem:monotonicity} displacement convexity of $V$ to find $\E[|\nabla_y \wgrad X_t^{s,\nu}(y)|^2]\leq ce^{-b(t-s)}$. In summary, derivatives of $X_{t}^{s,\nu}$ with respect to the measure or additional variables are controlled by applying Lemma \ref{lem:monotonicity} to the leading order term, whereas derivatives of $X_{t}^{x,s,\nu}$ are then controlled using $\nabla_x V \le -\lambda$ from Remark \ref{rk: UiT LSI}. 
\end{proof}

\section{Analysis of the remainder}
\label{sect: remainder}

Fix a time $T>0$ and define $R_T : \mathcal P_2(\R^d) \to [0,\infty)$ by
\begin{equation}
    R_T(\nu)= \int_0^1\int_{\R^d}\Big|\int_{\R^d\times \R^d} \delta_m^2  V(s\nu+(1-s)\mu_T,x,y,z)\,d(\nu-\mu_T)^{\otimes2}(y,z)\Big|^2\,d\nu(x)\,ds. 
\end{equation}
Recall that given the vector $Y_t=(Y_t^1,\dots,Y_t^n)$ solution to \ref{eq: particle SDE}, we define the empirical measure of the particle system 
\[m_t^n=\frac 1n \sum_{i=1}^n\delta_{Y_t^i}.\]
By exchangeability, the remainder $\mathcal R(t)$ defined in \eqref{def:remainder1} can be written as
\[
\mathcal R(t) = \E[R_t(m^n_t)].
\]
For this reason, the goal of this section is to show that $\E[R_T(m_T^n)] = O(1/n^2)$, in order to  prove Propositions \ref{pr: estimating remainder} and \ref{pr: remainder UIT}. It will be convenient to define the constant
\begin{equation}
\label{eq: definition of Theta}
    \Theta(t):=\|X^0_t\|_{\mathcal{C}^4}(1\vee \|X^0_t\|_{\mathcal{C}^4}^3)
\end{equation}
We will often use the flow property to identify $\|X^t_T\|_{\mathcal{C}^4}(1\vee \|X^t_T\|_{\mathcal{C}^4}^3)=\Theta(T-t)$.
\begin{remark} \label{rem:d=1}
    In dimension $d=1$ the analysis is particularly simple. Using Kantorovich duality twice and the definition of Wasserstein gradient,
    \begin{align*}
        \Big|&\int_{\R^d\times \R^d} \delta_m^2  V(s\nu+(1-s)\mu_T,x,y,z)\,d(\nu-\mu_T)(y)\,d(\nu-\mu_T)(z)\Big|\\
        &\leq \W_1(\nu,\mu_T)\sup_{z\in \R^d} \Big|\int_{\R^d}\nabla_z\delta_m^2 V(s\nu+(1-s)\mu_T,x,y,z)\,d(\nu-\mu_T)(y)\Big|\\
        &\leq \W_1^2(\nu,\mu_T)\sup_{y,z\in \R^d} \Big|\nabla_y\nabla_z\delta_m^2 V(s\nu+(1-s)\mu_T,x,y,z)\Big|\\
        &\leq \W_1^2(\nu,\mu_T)\|\wgrad^2 V\|_{\infty}.
    \end{align*}
    As a result, $\E[R_T(m_T^n)]\leq C\E\big[\W_1^4(m_T^n,\mu_T)\big]$. This quantity is well studied both in the short and long time regimes; focusing on the case of a bounded time horizon, a classical coupling argument shows that $\E\big[\W_1^4(m_T^n,\mu_T)\big]\leq C\E\big[\W_1^4(L_T^n,\mu_T)\big]$, where $L_T^n$ is the empirical measure associated to $n$ independent copies of the McKean--Vlasov law $\mu_T$. This term in turn  can be estimated by \cite[Theorem 2]{Fournier2015-yh} to obtain
    \[\E\big[\W_1^4(m_T^n,\mu_T)\big]\leq C/n^2,\]
    which allows us to close the estimate; notice that we are making use only of the first two Wasserstein derivatives. The same argument does not work in higher dimensions since the speed of convergence of empirical measures in Wasserstein distance deteriorates (see \cite[Theorem 1]{Fournier2015-yh}). This approach can be salvaged in higher dimension, by performing a higher-order Taylor expansion in the measure argument  to introduce three-way interactions (or four-way, etc.) which can be estimated hierarchically as we did with the pairwise interactions from the first-order expansion. But for $d>6$ this Taylor expansion requires more regularity than our argument based on weak propagation of chaos. 
\end{remark}

We will wish  to appeal to results from \cite{Buckdahn2017-op} and \cite{Chassagneux2022-jt} that apply to functions in $\mathcal C_{\text{bd}}^4(\mathcal P_2(\R^d))$. However, $R_T$ itself does not belong to $\mathcal C_{\text{bd}}^4(\mathcal P_2(\R^d))$, because $\delta_m^2V(m,x,y,z)$ may be unbounded in $(y,z)$. We thus perform an easy approximation procedure and pass to the limit at the end of the section. Let $\varphi \in C^\infty(\R^{3d})$ be an approximation of the indicator function of the unit ball in $\R^{3d}$; specifically, $0 \le \varphi \le 1$ with $\varphi\equiv 1$ on the unit ball and $\varphi \equiv 0$ outside the ball of radius $2$. Let $\varphi^N(x)=\varphi(x/N)$ for $N \in \N$. Define
\begin{equation*}
    r^N_T(s,\nu,x)= \int_{\R^d\times \R^d} \varphi^N (x,y,z)\delta_m^2  V(s\nu+(1-s)\mu_T,x,y,z)\,d(\nu-\mu_T)^{\otimes2}(y,z)
\end{equation*}
and set $R_T^N(\nu)=\int_0^1\int |r^N_T(s,\nu,x)|^2\,d\nu(x)\,ds$.
We begin with a useful lemma, which is essentially \cite[Lemma 2.7]{Chassagneux2022-jt}:

\begin{lemma}
    \label{lm: first derivatives are subpolynomial}
    For each $k \in \N$ and $f \in \mathcal C_{\textup{bd}}^k(\mathcal P_2(\R^d)\times \R^d)$ it holds 
    \begin{align*}
        \delta_m^k f(\mu,x,\boldsymbol{y})&= \prod_{i=1}^k y_i \int_{[0,1]^k} \wgrad^k f(\mu,x,q_1 y_1,\dots,q_k y_k)\,dq_1,\dots,d q_k  \\
        &\leq  \prod_{i=1}^k |y_i|\int_{[0,1]^k}|\wgrad^k f(\mu,x,y_1t_1,\dots,y_kt_k)|\,dt,
    \end{align*}
    where we treat the $k$-tensor $\wgrad^k f$ as a $d^k$-dimensional vector, and the norm is the Euclidean one.\end{lemma}
\begin{proof}
    By choice of normalization, $\delta_m^k  V (\mu,x,\boldsymbol{y})=0$ whenever at least one coordinate of $\boldsymbol{y} \in (\R^d)^k$ vanishes. The first identity follows immediately.
\end{proof}
We summarize the quartic properties of $R_T^N$.
\begin{lemma}
\label{lm: estimates on remainder}
    The following holds:
    \begin{enumerate}
    \item For every $N\in \N$, the function $R_T^N$ belongs to $\mathcal C^4_{bd}(\mathcal P_2(\R^d))$.
    \item $R_T^N(\mu_T)=\wgrad R_T^N(\mu_T,y)=\nabla_y \wgrad R_T^N(\mu_T,y)=\wgrad^2 R_T^N(\mu_T,y,z)=0$ for all $y,z$,
    \item There exist constants $C,p>0$ independent of $N$ such that for each multi-index $(\mathsf w,\boldsymbol{i})$ with size $0<k\leq 4$ we have
    \[|D^{\mathsf w,\boldsymbol{i}} R_T^N(\nu,\boldsymbol{y})|\leq C\bigg(1+\int_{\R^d}|x|^p \,d(\nu+\mu_T)(x)+|\boldsymbol{y}|^p\bigg).\]
    \item There exist constants $C,p>0$ independent of $N$ such that for each multi-index $(\mathsf w, \boldsymbol i)$ of size $k\leq 4$,
    \begin{align*}
         D^{\mathsf w,\boldsymbol{i}}P_t(R_T^N,\nu,\boldsymbol x)\leq C\Theta(t)\bigg(1+\int|z|^p \,d(\mu_t^{0,\nu}+\mu_T)(z)+\sum_{i=1}^k \int|z|^p\,d\mu_t^{x^i,0,\nu}(z)\bigg)
    \end{align*}
    \end{enumerate}
\end{lemma}
 \begin{proof} { \ } 
\begin{enumerate}
    \item By Lemma \ref{lm: first derivatives are subpolynomial} we have $|\delta_m^2 V(\nu,x,y,z)|\le \|\wgrad^2 V\|_\infty^2|y||z|$. Since $\varphi^N$ is supported on the ball of radius $2N$, we deduce $|\varphi^N\delta_m^2 V| \le 4N^2 \|\wgrad^2 V\|_\infty^2$. Similarly if we consider $\wgrad^k \delta_m^2 V(\nu,x,y,z,(u^1,\ldots,u^k),y,z)$ for some fixed $k\in\{1,2,3,4\}$ and $(x,u^1,\ldots,u^k)$, we have $|\wgrad^k \delta_m^2 V(\nu,y,z,(u^1,\dots,u^k))|\le \|\wgrad^{k+2} V\|_\infty^2|y||z|$, and we deduce that $\varphi^N\wgrad^k \delta_m^2 V$ is bounded. The $x$ derivatives of $\varphi^N\delta_m^2V$ are bounded similarly, and the claim follows readily.
    \item It is obvious that $R_T(\mu_T)=0$. Compute the Wasserstein gradient,
    \begin{align}
    \label{eq: wgrad rN}
    \begin{split}
        \wgrad r_T^N(s,\nu,x,u)&=s\int_{\R^{2d}}\varphi^N (x,y,z)\wgrad \delta_m^2  V(s\nu+(1-s)\mu_T,x,y,z,u)\,d(\nu-\mu_T)^{\otimes 2}(y,z)\\
        &+\int_{\R^d}\nabla_{y}\big[\varphi^N (x,u,z) \delta_m^2  V(s\nu+(1-s)\mu_T,x,u,z)\big]\,d(\nu-\mu_T)(z)\\
        &+\int_{\R^d}\nabla_{z}\big[\varphi^N (x,y,u)\delta_m^2  V(s\nu+(1-s)\mu_T,x,y,u)\big]\,d(\nu-\mu_T)(y).
    \end{split}
    \end{align}
    At $\nu=\mu_T$, this derivative vanishes.  Then we can check that
\begin{align*}
     \wgrad R_T^N(\nu,y)&=2\int_0^1 \nabla_x r_T^N(s,\nu,y) r^N_T(s,\nu,y)\,ds+2\int_0^1\int_{\R^d}\wgrad r_T^N(s,\nu,x,y)r_T^N(s,\nu,x)\,d\nu(x)\,ds.
 \end{align*}
 Note that $r_T^N(s,\nu,x)$, $\nabla_x r_T^N(s,\nu,x)$, and $\nabla_x^2 r_T^N(s,\nu,x)$ vanish at $\nu=\mu_T$. We thus deduce that $\wgrad R_T^N(\mu_T,\cdot)=0$, and differentiating the above shows also that $\nabla_y \wgrad R_T^N(\mu_T,y)=0$. Similar considerations show that $\wgrad^2 R_T^N(\mu_T,\cdot,\cdot) = 0$.
 \item We check that $r_T^N$ satisfies a similar property. Indeed by using Lemma \ref{lm: first derivatives are subpolynomial}, boundedness of $\varphi^N$ and $\wgrad^2  V$ and Young's inequality,
 \begin{align*}
 r_T^N(s,\nu,x)&\leq \int_{\R^d\times \R^d} \|\wgrad^2  V\|_\infty |y||z|\,d(\nu+\mu_T)^{\otimes2}(y,z)\\
 &\leq C\Big(\int |y|\,d\mu_T(y)\Big)^2+C \bigg(\int_{\R^d} |y|\,d\nu(y)\bigg)^2
 \end{align*}
 Similarly if we consider the first term of \eqref{eq: wgrad rN}, by Lemma \ref{lm: first derivatives are subpolynomial} and boundedness of $\wgrad^3 V,$
 \begin{align}
 \label{eq: wgrad r first term}
 \begin{split}
     \bigg|s\int_{\R^{2d}} &\varphi^N (x,y,z)\wgrad \delta_m^2  V(s\nu+(1-s)\mu_T,x,y,z,u)\,d(\nu-\mu_T)^{\otimes 2}(y,z)\bigg| \\
     &\leq \int_{\R^{2d}}\|\wgrad^3  V\|_\infty |y| |z|\,d(\nu+\mu_T)^{\otimes 2}(y,z)\\
     &\le C\Big(\int |y|\,d\mu_T(y)\Big)^2+C \bigg(\int_{\R^d} |y|\,d\nu(y)\bigg)^2
\end{split}
 \end{align}
 Considering instead the second term of \eqref{eq: wgrad rN} and using again Lemma \ref{lm: first derivatives are subpolynomial} and Young's inequality, we get
 \begin{align}
 \label{eq: wgrad r second term}
 \begin{split}
     \bigg|\int_{\R^{d}}&\nabla_y\varphi^N (x,u,z) \delta_m^2  V(s\nu+(1-s)\mu_T,x,u,z)\,d(\nu-\mu_T)(z)\bigg|\\
     &+\bigg|\int_{\R^{d}}\varphi^N (x,u,z) \nabla_y\delta_m^2  V(s\nu+(1-s)\mu_T,x,u,z)\,d(\nu-\mu_T)(z)\bigg| \\
     &\leq C|u|\int_{\R^{d}}\|\wgrad^2  V\|_\infty|z| \,d(\nu+\mu_T)(z)+C\int_{\R^{d}}\|\wgrad^2  V\|_\infty  |z|\,d(\nu+\mu_T)(z)\\
     &\leq C+|u|^2+ C\Big(\int |y|\,d\mu_T(y)\Big)^2+C \bigg(\int_{\R^d} |y|\,d\nu(y)\bigg)^2.
\end{split}
 \end{align}
Repeating the same reasoning yields that for every multi index $|\boldsymbol{i}|\leq 4$, for some $p\geq 1$ and universal constant $C$,
 \begin{align*}
     |D^{\mathsf w,\boldsymbol{i}} r_T^N(s,\nu,x,\boldsymbol{y})|\leq C\bigg[1+  \Big(\int_{\R^d} |z|\,d\nu(z)\Big)^p+\Big(\int_{\R^d} |z|\,d\mu_T(z)\Big)^p+|\boldsymbol{y}|^p\bigg]
 \end{align*}
 now the result follows by computing the derivatives of $R^N_T$ and using the bounds on $r^N_T$ and Jensen's inequality.
 \item Apply Lemma \ref{lm: bounds onf wgrad semigroup}.
 \end{enumerate}
 \end{proof}

\subsection{A semigroup interpolation of the remainder}

We next make use of the weak error expansion that is the main contribution of \cite{Chassagneux2022-jt}. 

\begin{proposition}
\label{pr: PDE on Wasserstein Space}
The function $P_t R^N_T$ is differentiable in time and belongs to $\mathcal C^4_{\textup{bd}}(\mathcal P_2(\R^d))$. Moreover, we have the following formula:
\begin{align}
\begin{split}
\label{eq: P_{T-t} f equation}
    \E[R^N_T(m_T^n)]&=\E[P_T R^N(m_0^n)]\\
    &+\frac{\sigma^2}{n}\int_0^T\bigg(\E[P_{t}g_{t}(m_0^n)]+\frac{\sigma^2}{n^2}\sum_{i=1}^n \int_0^t\E[\tr \wgrad^2  P_{t-s}g_t (m_s^n,Y_s^i,Y_s^i)]\,ds\bigg)\,dt
\end{split}
\end{align}
where we define $g_t : \mathcal P_2(\R^d) \to \R$ by
\begin{equation*}
    g_t(\nu)= \int_{\R^d} \tr\wgrad^2 P_{T-t} R_T^N(\nu,x,x)\,d\nu(x).
\end{equation*}
    \end{proposition}
\begin{proof}
Regularity of $P_t R_T^N$ comes from \cite[Theorem 2.18]{Chassagneux2022-jt} since $R_T^N$ and $ V$ both have $4$ derivatives in space and measure. Fix $T>0$; by \cite[Theorem 7.2]{Buckdahn2017-op}, $P_{T-t}R^N_T $ satisfies (in the classical sense) 
\begin{align} 
        \partial_t P_{T-t}R_T^N(\nu)+ \int_{\R^d}\Big(\wgrad P_{T-t}R_T^N(\nu,x)\cdot  V(\nu,x)+\nabla_x \cdot \wgrad P_{T-t} R_T^N (\nu,x)\Big) \,d\nu(x)&=0, \label{eq: PDE on Wasserstein space}
\end{align}
for $\nu \in \mathcal P_2(\R^d)$ and $t \in (0,T)$.
By the same reasoning, $s\mapsto P_{T-s} P_{T-t} f$ solves the same PDE on $[0,t]$ (with the appropriate terminal condition).
    We follow the steps of \cite[Lemma 2.11(ii)]{Chassagneux2022-jt} to obtain a second order expansion of $f$. Observe that $P_{T-t} f(m_t^n)$ is a function of the vector $Y_t$. Use Ito's rule on $P_{T-t} f (m_t^n)$ to see
    \begin{align*}
        dP_{T-t} &R_T^N(m_t^n)=\partial_t P_{T-t} R_T^N(m_t^n)\,dt+\int_{\R^d} \wgrad P_{T-t} R_T^N(m_t^n,x)\cdot  V(m_t^n,x)\,dm_t^n(x)\,dt\\
        &+\sigma^2\int_{\R^d}\nabla_x\cdot \wgrad P_{T-t} R_T^N(m_t^n,x)\,dm_t^n(x)\,dt\\
        &+\frac{\sigma\sqrt2}{n}\sum_{i=1}^n \wgrad P_{T-t} R_T^N(m_t^n,Y_t^i)\cdot dW^i_t + \frac{\sigma^2}{n^2}\sum_{i=1}^n \tr\wgrad^2 P_{T-t} R_T^N(m_t^n,Y_t^i,Y_t^i)\,dt.
    \end{align*}
    The first three $dt$ terms combine to zero, using the PDE \eqref{eq: PDE on Wasserstein space}. Notice that the stochastic integral is a true martingale since $\wgrad P_{t} R_T^N$ is bounded uniformly in $t$. Integrate and take expectations to get  
\begin{align*}
    \E[R_T^N(m_T^n)]=\E[P_T R_T^N(m_0^n)]+\frac{\sigma^2}{n} \int_0^T \frac{1}{n}\sum_{i=1}^n \E[ \tr\wgrad^2P_{T-t} R_T^N(m_t^n,Y_t^i,Y_t^i)]\,dt
\end{align*}
Notice that the integrand in the last term is precisely $\E[g_t(m_t^n)]$.
We can thus use Ito's rule again on the flow $s\mapsto P_{t-s} g_t(m_s^n)$ to find
\begin{align*}
\E[g_t(m_t^n)] = \E[P_t g_t(m_0^n)]+\frac{\sigma^2}{n^2}\sum_{i=1}^n \int_0^t\E[ \tr\wgrad^2 P_{t-s} g_{t}(m_s^n,Y_s^i,Y_s^i)]\,ds.
\end{align*}
Combine the above two equations to complete the proof.
\end{proof}

We will need some regularity estimates on the functions $g_t$ and $P_sg_t$, which easily follow from Lemma \ref{lm: bounds onf wgrad semigroup}:

\begin{corollary}
\label{cr: estimates on g} 
There exist $p,q,r\geq 1, C>0$ such that, for every $T \ge t \ge s \ge 0$ and $y_1,y_2 \in \R^d$,
    \begin{align*}
    |\wgrad g_t(\nu,y_1)| &+ |\wgrad^2 g_t(\nu,y_1,y_2)| + |\nabla_y\wgrad g_t(\nu,y_1)| \\ &\leq C \Theta(T-t) \int_{\R^d}(1+|x|^p) \,d(\mu_T+\mu_T^{t,\nu}+\mu_T^{y,t,\nu})(x)  \\
    |\wgrad P_s g_t(\nu,y_1)|&+ |\wgrad^2 P_s g_t(\nu,y_1,y_2)| + |\nabla_y\wgrad P_s g_t(\nu,y_1)| \\ & \leq  C\Theta(s)\Theta(T-t)\int_{\R^d}(1+|x|^p) \,d(\mu_T+\mu_{T-(t-s)}^{0,\nu}+\mu_{T-(t-s)}^{y_1,0,\nu}+\mu_{T-(t-s)}^{y_2,0,\nu})(x) \end{align*}
\end{corollary}
\begin{proof}
    By the chain rule and symmetry of Wasserstein derivatives,
    \begin{align*}
        (\wgrad g_t(\nu,y))_i = \partial_{y_i}\big[ \tr \wgrad^2  P_{T-t} R^N_T(\nu, y,y) \big] +  \int_{\R^d}\tr \wgrad^2 (\wgrad P_{T-t}R^N_T(\nu,x,x,y))_i\,d\nu(x).
    \end{align*}
    Use Lemma \ref{lm: bounds onf wgrad semigroup}, Lemma \ref{lm: estimates on remainder}(4) and H\"older's inequality we obtain,
    \begin{align*}
        |\wgrad g_t(\nu,y)|\leq  C\Theta(T-t) \Big(1+\int_{\R^d}|x|^p\,d(\mu_{T-t}^{y,0,\nu}+\mu_{T-t}^{0,\nu}+\mu_T)(x)\Big).
    \end{align*}
    The reasoning for the second derivatives is the same. The claimed bounds on the derivatives of $g$ now follow from the time-homogeneity \eqref{eq:MVtimehomo}. To study derivatives applied to the semigroup, use Lemma \ref{lm: bounds onf wgrad semigroup} and the first claim to see
    \begin{align*}
        |\wgrad P_s g_t(\nu,y)| &\le C\Theta(s)\Theta(T-t) \Big(1+\int_{\R^d}|x|^p\,d(\mu_T^{t,\mu_s^{0,\nu}}+\mu_T^{t,\mu_s^{y,0,\nu}}+\mu_T)(x)\Big) \\
        &\le C\Theta(s)\Theta(T-t) \Big(1+\int_{\R^d}|x|^p\,d(\mu_{T-(t-s)}^{0,\nu}+\mu_{T-(t-s)}^{y,0,\nu}+\mu_T)(x)\Big),
    \end{align*}
    where the last step used the  time-homogeneity \eqref{eq:MVtimehomo} and flow properties \eqref{eq:MVflow}. The reasoning for higher order derivatives is the same.
\end{proof}

\subsection{Proof of Propositions \ref{pr: estimating remainder} and \ref{pr: remainder UIT}}
We fix $T > 0$ and prove both propositions together by estimating $\E[R_T^N(m^n_T)]  \le C/n^2$, for a constant $C$ that does not depend on $n$ or $N$, and in the case of Proposition \ref{pr: remainder UIT} does not depend on $T$ either. The same bound will apply to $\mathcal R(t)=\E[R_T(m^n_T)]$ by sending $N\to\infty$, noting that dominated convergence applies because $\delta_m^2V$ has quadratic growth in the spatial variables by Lemma \ref{lm: first derivatives are subpolynomial}, and because of the moment bounds of Lemma \ref{lm: T1 short time} and Proposition \ref{pr: LSI}. We begin by showing that, for $y,z  \in \R^d$ and $t\in[0,T]$, 
     \begin{equation}
\label{eq: derivative of semigroup of R}
     \wgrad P_{T-t} R^N_T(\mu_t,y) =  0, \qquad \wgrad^2 P_{T-t}R^N_T(\mu_t,y,z) =0.
\end{equation} 
We do this for $d=1$ to simplify notation.
To prove the first, use \cite[Equation (6.3)]{Buckdahn2017-op} to compute
\begin{align*}
    \wgrad P_{T-t} R^N_T(\nu,y)= \E\big[\wgrad R^N_T(\mu_T^{t,\nu}, X_T^{y,t,\nu})\nabla_x  X_T^{y,t,\nu}+\wgrad R^N_T(\mu_T^{t,\nu},  X_T^{t,\nu}) \wgrad  X_T^{ t,\nu}(y)\big]
\end{align*}
Recalling $\mu_T^{t,\mu_t}=\mu_T^{0,\mu_0}=\mu_T$ by the flow property \eqref{eq:MVflow}, and also that $\wgrad R(\mu_T,y)=0$ for every $y$ by Lemma \ref{lm: estimates on remainder}, the first equation of \eqref{eq: derivative of semigroup of R} follows. 
To differentiate again to compute $\wgrad^2 P_{T-t}R^N_T(\nu,y,z)$ is more cumbersome, but the method is explained in detail in \cite[Lemma 6.2]{Buckdahn2017-op}. It is a sum of many terms, and the only important point for our purposes is that each term contains a factor of $\wgrad R_T^N(\mu_T^{t,\nu},y)$, $\nabla_y \wgrad R_T^N(\mu_T^{t,\nu},y)$, or  $\wgrad^2 R_T^N(\mu_T^{t,\nu},y,z)$. When applied with $\nu=\mu_t$, the flow property $\mu_T^{t,\mu_t}=\mu_T$ combined with Lemma \ref{lm: estimates on remainder} proves the second identity of  \eqref{eq: derivative of semigroup of R}.
One consequence of \eqref{eq: derivative of semigroup of R}, which we will only need for $t=0$, is that 
    \begin{align*}
        0=\wgrad^2 P_T R^N_T(\mu_0,x,y)&=  \nabla_x \nabla_y \delta_m^2 P_T R^N_T(\mu_0,x,y).
    \end{align*}
    Hence, $\delta_m^2 P_{T} R_T^N(\mu_0,x,y)=f(x)+g(y)$ for some functions $f,g$.

    Now, to estimate $\E[R^N_T(m^n_T)]$, we will use the expansion from Proposition \ref{pr: PDE on Wasserstein Space}, and separately estimate each of the three terms on the right-hand side of \eqref{eq: P_{T-t} f equation}. We begin with $\E[P_TR^N_T(m^n_0)]$. We perform a  second-order Taylor expansion of $R^N_T(m^n_0)$ around $R^N_T(\mu_0)$, which was developed in \cite{Chassagneux2022-jt}. Specifically, we apply the final formula from the proof of \cite[Theorem 2.14(ii)]{Chassagneux2022-jt}:
    \begin{align*}
        \E[P_T R^N_T(m_0^n)] &= P_T R^N_T(\mu_0) +  \frac{1}{2n}\int_{\R^{2d}} \delta^2_m P_T R^N_T(\mu_0,x,x)-\delta_m^2 P_T R^N_T(\mu_0,x,y) \,d\mu(x)\,d\mu(y)+ \frac{A}{n^2},
    \end{align*}
    for a quantity $A$ to be defined shortly. The key point here is that the  zero and first order terms vanish. Indeed, $P_T R^N_T(\mu_0)=R^N_T(\mu_T)=0$, and the formula $\delta_m^2 P_{T} f(\mu_0,x,y)=f(x)+g(y)$ implies that the integral vanishes. Hence,
    \begin{equation}
        \E[P_T R^N_T(m_0^n)] = A/n^2, \label{eq:PTRNT1}
    \end{equation}
    To finally define $A$, let us recall that $Y^1_0,\ldots,Y^n_0$ are iid with law $\mu_0$, and we let $\widetilde{Y}_0^1,\dots,\widetilde{Y}_0^n$ denote additional iid copies from $\mu_0$. For $q \in [0,1]^2$ and $u \in [0,1]^3$  define
    \begin{align*}
        \eta_{q}&= \mu_0+q_1(m_0^n-\mu_0)+\frac{q_1 q_2}{n}(\delta_{\widetilde{Y}_0^1}-\delta_{Y^1_0}), \\
        \theta_{u}&=\eta_{u_1,u_2}+\frac{u_3}{n}\Big(\mu_0+u_2(m_0^n-\mu_0)+\frac{u_1}{n}(\delta_{\widetilde{Y}_0^2}-\delta_{Y_0^2})-\eta_{u_1,u_2}\Big).
    \end{align*}
Then, $A$ is defined by 
\begin{align*}
        A &:= \frac12 A_1 + \frac12 A_2 + \frac{n-1}{2n}A_3, \\
        A_1 &:= \E\Big[\int_{[0,1]^2}\int_{\R^{2d}} (1-q_1^2)q_2 \delta^3_m P_T R^N_T(\eta_{q},\widetilde{Y}_0^1,x,y)d(\delta_{\widetilde{Y}_0^1}-\delta_{Y_0^1})^{\otimes 2}(x,y)\,dq\Big],  \\
        A_2 &:= \E\Big[\int_{[0,1]^2}\int_{\R^{2d}} (1-q_1^2) \delta^3_m P_T R^N_T(\eta_{q},\widetilde{Y}_0^1,x,y)d(\delta_{\widetilde{Y}_0^1}-\delta_{Y_0^1})(x)\,d(\delta_{Y_0^1}-\delta_{\widetilde{Y}_0^2})(y)\,dq\Big], \\
        A_3 &:=  \E\Big[\int_{[0,1]^3}\int_{\R^{2d}} (1-u_1^2)u_1 \delta^4_m P_T R^N_T(\theta_{u},\widetilde{Y}_0^1,x,\widetilde{Y}_0^2,y) \, d(\delta_{\widetilde{Y}_0^1}-\delta_{Y_0^1})(x)\,d(\delta_{\widetilde{Y}_0^2}-\delta_{\widetilde{Y}_0^3})(y)\,du\Big]. 
    \end{align*}
    We will show that $|A|$ is bounded, uniformly in $N$ and $n$.
    To this end, notice that both $\eta_q$ and $\theta_u$ are mixtures of the subgaussian measure $\mu_0$ with the random measures $m_0^n,\delta_{Y_0^1},\delta_{\widetilde{Y}_0^i}$ which are subgaussian in expectation. Thus, it follows that for some constant that does not depend on time or $n$,
    \begin{align}
    \label{eq: bounds on eta and theta}
        \sup_{u,q}\E\Big[\int_{\R^d} |x|^p\,d(\eta_q+\theta_u)(x)\Big]\leq C<\infty.
    \end{align}
    We first show how to estimate  $|A_1|$, with the other two terms being similar. Use Jensen's inequality, Lemma \ref{lm: first derivatives are subpolynomial} and symmetry of the Wasserstein derivatives to find  
    \begin{align*}
        |A_1| &\leq 2\E\Big[\int_{[0,1]^4} \big|\wgrad^3 P_T R^N_T(\eta_{q},\widetilde{Y}_0^1,r_1\widetilde{Y}_0^1,r_2Y_{0})\big||\widetilde{Y}_0^1|^2|Y^1_{0}|   \,dq\,dr\Big]\\
        &\quad +\E\Big[\int_{[0,1]^4} \big|\wgrad^3 P_T R^N_T(\eta_{q},\widetilde{Y}_0^1,r_1\widetilde{Y}_0^1,r_2\widetilde{Y}_0^1)\big||\widetilde{Y}_0^1|^3 \,dq\,dr\Big]\\
        &\quad +\E\Big[\int_{[0,1]^4} \big|\wgrad^3 P_T R^N_T(\eta_{q},\widetilde{Y}_0^1,r_1Y_{0},r_2Y_{0})\big||\widetilde{Y}_0^1||Y_{0}|^2  \,dq\,dr\Big].
    \end{align*}
    By Lemma \ref{lm: bounds onf wgrad semigroup}, the fact that $Y^i_0$ and $\widetilde{Y}^i_0$ are iid $\sim\mu_0$, H\"older's inequality, Lemma \ref{lm: estimates on remainder}(4) and Jensen's inequality we get for some $p\in \N$,
    \begin{align} 
    \label{eq: bound on A1}    
    |A_1|
    &\leq C\Theta(T)
     \E\bigg[\int_{[0,1]^4}\int_{\R^d}1+ |x|^{2p} \,d(\mu_T^{r_1Y_0^1,0,\eta_q}+\mu_T^{Y_0^1,0,\eta_q}+\mu_T^{0,\eta_q}+\mu_T)(x)\,dr\,dq\bigg] .
    \end{align}
     Here $C$ can depend on $\E[|Y_0^1|^6]$.
    The reasoning for $A_2$ is identical, and for $A_3$ the only difference is that we are working with the $4$-th flat derivative, and we thus obtain a bound in terms of $\wgrad^4 P_T R^N_T$ and the measure $\theta_{u}$. Ultimately,  recalling  \eqref{eq:PTRNT1}, we obtain
    \begin{align}
        \label{bound on P_T R}
        \begin{split}
            \E[P_T R^N_T(m_0^n)] &\leq  \frac{C }{n^2} \Theta(T)\E\bigg[1 + \int_{[0,1]^5} \int_{\R^d}|x|^{p}\,d(\mu_T+\mu_T^{Y_0^1,0,\eta_q}+\mu_T^{Y_0^1,0,\theta_u})(x)\,dq\,du\\
            &\quad +  \E\int_{[0,1]^6} \int_{\R^d}|x|^p\,d(\mu_T^{rY_0^1,0,\eta_q}+\mu_T^{rY_0^1,0,\theta_u}+\mu_T^{0,\eta_q}+\rho_T^{0,\theta_u})(x)\,dr\,dq\,du \bigg]. 
    \end{split}
    \end{align}
    The constants $C$ and $p$ do not depend on $N$, $n$, or $T$.
    
We next study the two $g_t$ terms from Proposition \ref{pr: PDE on Wasserstein Space}. We will again Taylor expand $P_tg_t(m^n_0)$ around $\mu_0$, but now we need only to do so at first order. Since $\wgrad^2 P_{T-t} R^N_T (\mu_t,\cdot,\cdot)=0$ by \ref{eq: derivative of semigroup of R} we obtain $g_t(\mu_t)=0$, and thus $P_tg_t(\mu_0)=g_t(\mu_t)=0$ by the flow property \eqref{eq:MVflow}. This is the zero order term, and so the Taylor expansion  from  \cite[Equation (2.25)]{Chassagneux2022-jt} yields
\begin{align}
    \E[P_t g_t(m_0^n)]&= \frac 1n \int_{[0,1]^2} q_2 \E\big[\delta_m^2 P_t g_t(\eta_{q},Y_0^1,Y^1_0)-\delta_m^2 P_t g_0(\eta_{q},Y^1_0,\widetilde{Y}_0^1)\big]\,dq  \nonumber  \\
    &\leq \frac 1n \int_{[0,1]^2}\E\big[|\delta_m^2 P_t g_t(\eta_{q},Y_0^1,Y_0^1)|+|\delta_m^2P_t g_t(\eta_{q},Y_0^1,\widetilde{Y}_0^1)|\big]\,dq \nonumber \\
    &\leq \frac{1}{n} \E \Big[\int_{[0,1]^4}|\wgrad^2 P_t g_t(\eta_{q},r_1Y_0^1,r_2 Y_0^1)||Y_0^1|^2\,dr\,dq\Big]  \nonumber \\
    &\quad +\frac{1}{n} \E \Big[\int_{[0,1]^4}|\wgrad^2 P_t g_t(\eta_{q},r_1Y_0^1,r_2 \widetilde{Y}_0^1)||Y_0^1||| \widetilde{Y}_0^1| \,dr\,dq\Big]. \label{eq: bound on g}
\end{align}
The last step used Lemma \ref{lm: first derivatives are subpolynomial}.
Use Cauchy-Schwarz, Lemma \ref{lm: bounds onf wgrad semigroup}, and Corollary \ref{cr: estimates on g} to find 
\begin{align*}
    &\E[P_tg_t(m_0^n)]\\
    &\leq 
    \frac{C}{n}\Theta(t)\Theta(T-t)  \E\Big[\int_{[0,1]^4}\int_{\R^d}(1+|x|^p)\,d(\mu_T^{0,\eta_q}+\mu_T^{r_1 Y_0^1,\eta_q}+\mu_T^{r_2Y_0^1,0,\eta_q}+\mu_T)(x)drdq\Big]
\end{align*}
for constants $C$ and $p$ which do not depend on $N$, $n$, or $T$.
To estimate the final $g_t$ term from Proposition \ref{pr: PDE on Wasserstein Space}, there is no Taylor expansion needed, and we simply apply Corollary \ref{cr: estimates on g} to obtain
\begin{align}
\begin{split}
    \label{eq: bound on wgrad g}
    \frac{1}{n} &\sum_{i=1}^n \E[\tr\wgrad^2 P_{t-s}g_t (m_s^n,Y_s^i,Y_s^i)]= \E\Big[\int_{\R^d}\tr\wgrad^2P_{t-s}g_t(m_s^n,y,y)\,dm_s^n(y)\Big]
    \\&\leq \frac {C}{n}\Theta(t-s)\Theta(T-t) \E\Big[\int_{\R^d} \int_{\R^d} (1+|x|^p)\,d(\mu_T + \mu_{T-s}^{0,m_s^n}+\mu_{T-s}^{y,0,m_s^n})(x)  \,dm_s^n(y)\Big].
\end{split}
\end{align}
Plugging the estimates \eqref{bound on P_T R}, \eqref{eq: bound on g}, and \eqref{eq: bound on wgrad g} into \eqref{eq: P_{T-t} f equation} yields
\begin{align*}
    \E&[R^N_T(m_T^n)]\leq \frac{C }{n^2}\Theta(T) \E\bigg[1+\int_{[0,1]^5}\int_{\R^d}|x|^p \,d(\mu_T^{rY_0^1,0,\eta_q}+\mu_T^{rY_0^1,0,\theta_u}+\mu_T)(x)\,dq\,du\bigg] \\
    &+ \frac{C }{n^2}\Theta(T)\E\Big[\int_{[0,1]^6}\int_{\R^d}|x|^p\,d(\mu_T^{0,\eta_q}+\mu_T^{0,\theta_u}+\mu_T^{Y_0^1,0,\eta_q}+\mu_T^{Y_0^1,0,\theta_u})(x)\,dq\,du\Big]  \\
    &+\frac{C }{n^2}\int_0^T \Theta(t)\Theta(T-t) \E\Big[\int_{[0,1]^4}\int_{\R^d} (1+|x|^p)\,d(\mu_T^{0,\eta_q}+\mu_T^{rY_0^1,0,\eta_q})(x)\,dr\,dq\Big]\,dt\\
    &+\frac{C}{n^2}\int_0^T \int_0^t \Theta(t-s)\Theta(T-t)\E\Big[\int_{\R^d}\int_{\R^d}(1+|x|^p)\,d(\mu_T+\mu_{T-s}^{0,m_s^n}+\mu_{T-s}^{y,0,m_s^n})(x) \,dm_s^n(y)\Big]\,ds\,dt.
\end{align*}

Notice that under Assumption \ref{assumption on Phi} we have
\begin{equation}
\sup_{0 \le s \le T}\E\Big[\int_{\R^d} |x|^p\,dm^n_s(x)\Big] = \sup_{0 \le s \le T}\E\big[|Y^1_s|^p\big] < \infty, \label{pf:momentbound5}
\end{equation}
by Lemma \ref{lm: T1 short time},
and under Assumption \ref{Assumption: UiT}, we have
\begin{equation}
\sup_{s \ge 0}\E\Big[\int_{\R^d} |x|^p\,dm^n_s(x)\Big] = \sup_{s \ge 0}\E\big[|Y^1_s|^p\big] < \infty, \label{pf:momentbound5uit}
\end{equation}
by Proposition \ref{pr: LSI}.
Now, to complete the proof of Propositions \ref{pr: estimating remainder}, under merely Assumption \ref{assumption on Phi}, we use the moment bounds from Lemma \ref{lem:flow moment bounds} and \eqref{pf:momentbound5}, and the estimates on $\|X_t^s\|_{\mathcal C^4}$ from \ref{lm: estimates on flows}; in particular, recalling the definition of $\Theta$ \eqref{eq: definition of Theta}, under assumption \ref{assumption on Phi} we know that for $t\leq T, \Theta(t)\leq C$ for some constant $C$ that depends on $T$. Along with  \eqref{eq: bounds on eta and theta}, and recalling $Y^1_0$ is subgaussian, this yields the desired $\E[R^N_T(m_T^n)] \le C/n^2$ for a constant $C$ that can depend on $T$ but not on $N$ or $n$.
If we also assume \ref{Assumption: UiT}, then we complete the proof of Proposition \ref{pr: remainder UIT} similarly. The moment bounds are uniform in time by Lemma \ref{lem:flow moment bounds UIT} and \eqref{pf:momentbound5uit}, and we have the exponential decay  $\Theta(t) \le Ce^{-bt}$ from Proposition \ref{pr: decay of wgrad X}. This allows us to obtain $\E[R^N_T(m_T^n)] \le C/n^2$ as above, but with the constant $C$ no longer depending on $T$. \hfill\qedsymbol
\begin{remark}
\label{rk: long time weak chaos}
By inspecting what properties of the function $R_T$ were used in the proofs of Propositions \ref{pr: remainder UIT}, we see that a more general statement holds.
Let $\Phi \in \mathcal C_{\text{bd}}^4(\R^d)$ be such that $|\wgrad \Phi(\mu_T,\cdot)|=|\wgrad^2 \Phi(\mu_T,\cdot,\cdot)|=0$. Then, under assumption \ref{Assumption: UiT}, we have
    \begin{align*}
        \sup_{t\geq 0}|\E[\Phi(m_t^n)]-\Phi(\mu_t)|\leq \frac C{n^2}.
    \end{align*}
    This has implications for stationary measures of Wasserstein gradient flows. For instance, let $V(\nu,x)=-\wgrad \Psi(\nu,x)$ for some $\Psi$ that satisfies the assumptions of Corollary \ref{co: Mean Field Langevin Dynamics}. By \cite[Proposition 4.4]{ChaosMFLD}, the fixed point problem for $\rho \in \mathcal P_2(\R^d)$
    \[\rho(dx)\propto \exp\{-\delta_m \Psi(\rho,x)\}\,dx\]
    has a unique solution that is also the stationary measure for \eqref{eq: MFLD}. On the other hand, by \cite[Proposition 4.3]{ChaosMFLD}, the law of the particle system \eqref{eq: MFLD particles} converges to a stationary measure $\pi_\infty^n\in \mathcal P_2((\R^d)^n)$ given by $\pi^n_\infty(dx) \propto \exp(-n\Psi(m^n_x))dx$ where $m^n_x = (1/n)\sum_i\delta_{x_i}$; see \cite[Equation (2.16)]{ChaosMFLD}. Starting the particle system at $\rho^{\otimes n}$ yields 
    \[\Big|\int_{(\R^d)^n} \Phi\Big(\frac 1n \sum_{i=1}^n \delta_{y^i}\Big)\,d\pi^n_\infty(y) -\Phi(\rho) \Big|\leq \frac{C}{n^2}\]
    under the assumption that the first and second derivatives of $\Phi$ vanish at $\rho$.
\end{remark}
\begin{remark}
\label{rk: beyond iid}
In Theorems \ref{th: short time} and \ref{th: UiT}, the assumption that the initial condition of the particle system be independent could be relaxed in the following way. Let $\pi_0=\text{Law}(Y_0)$, denote by $\pi^k_0=\text{Law}(Y_0^1,\dots,Y_0^k)$ and assume the following:
    \begin{enumerate}
        \item[(1)]  the law $\pi_0$ is exchangeable and subgaussian;
        \item[(2)]  for every $k\leq n$, $H(\pi_0^k\,\|\,\mu_0^{\otimes k})\leq C k^2/n^2$ for some constant $C\geq 0$ that does not depend on $k$ or $n$.
    \end{enumerate}
    Assume moreover that $\pi_0$ satisfies the following weak propagation of chaos property: for every $\Phi \in \mathcal C_{\text{bd}}^4(\R^d)$ such that
    \begin{enumerate}
        \item[(3)] $\wgrad \Phi(\mu_0,\cdot)=0\,\,;\,\,\wgrad^2 \Phi(\mu_0,\cdot,\cdot)=0$;
        \item[(4)] There exists $C_{\delta}$ such that $|\delta_m^i \Phi(\nu,\boldsymbol{y})|\leq C_{\delta}(1+|\boldsymbol{y}|^p+\smallint |z|^{p}\,d\nu(z))$ for some $p\geq 1$ for $i=1,\dots,4$.
    \end{enumerate}
    Then, we have
    \begin{equation*}
        \Big|\int_{(\R^d)^n} \Phi\Big(\frac1n \sum_{i=1}^n\delta_{y_i}\Big)\,d\pi_0^n(y)-\Phi(\mu_0)\Big|\leq \frac{C}{n^2},
    \end{equation*}
    where $C$ depends on $C_\delta$ and $p$ only. 
    The proof of Proposition \ref{pr: estimating remainder} shows that the assumptions (3) and (4) are redundant when $\pi_0=\mu_0^{\otimes n}$ is a product measure. Remark \ref{rk: long time weak chaos} above provides another class of examples where (1--4) hold: measures of the form $\pi_0(dx) \propto \exp(-n\Psi(m^n_x))dx$ where $\Psi$ satisfies the assumptions of Corollary \ref{co: Mean Field Langevin Dynamics}.
\end{remark}

\appendix

\section{Proof of Lemma \ref{lm: T1 short time} } \label{ap:T1shorttime}

The moment estimates are standard, as $V$ is Lipschitz and the initial distribution $\mu_0$ has finite moments of all orders.
Let $W[t] \in \mathcal P(C([0,t];\R^d))$ denote the law of $(X_0+\sqrt{2}\sigma B_s)_{s \in [0,t]}$, where $B$ is a Brownian motion and $X_0 \sim\mu_0$ is independent of $B$. Then $\mu[t]$ is equivalent to $W[t]$.
Fix any $\nu[t] \ll \mu[t]$. Since $\nu[t] \ll W[t]$, by Girsanov's theorem, we have $\nu[t]=\Law(Z[t])$ for some process $Z$ satisfying
\[
dZ_s = \alpha(s,Z) dt + \sqrt{2} \sigma dB_s,
\]
with $Z_0 \sim \nu_0$ being independent of $B$, and where $\alpha$ is a progressively measurable function satisfying $\int_0^t |\alpha(s,Z)|^2\,ds < \infty$ a.s. Enlarge the probability space if necessary to support a random variable $\xi \sim \mu_0$ such that $(Z_0,\xi)$ is independent of $B$ and is optimally coupled in the sense that 
\[
\W_1(\mu_0,\nu_0) = \E[|Z_0-\xi|].
\]
Since $V$ is Lipschitz, we may construct a strong solution to the SDE
\[
dX_s = V(\mu_s,X_s)ds + \sqrt{2} \sigma dB_s, \quad X_0=\xi,
\]
and we have $X[t] \sim \mu[t]$. We then estimate the difference as
\begin{align*}
|X_s-Z_s| &\le |Z_0-\xi| + \int_0^t|V(\mu_s,X_s)-\alpha(s,Z)|\,ds \\
    &\le |Z_0-\xi| + L\int_0^t|X_s-Z_s|\,ds + \int_0^t|V(\mu_s,Z_s)-\alpha(s,Z)|\,ds ,
\end{align*}
where we used that $V$ is $L$-Lipschitz. Using Gronwall's inequality and taking expectations,
\[
\E\big[\sup_{s \in [0,t]}|X_s-Z_s|\big] \le e^{Lt}\W_1(\mu_0,\nu_0) + e^{Lt}\int_0^t\E[|V(\mu_s,Z_s)-\alpha(s,Z)|]\,ds
\]
Now, we have the well known entropy identity (see, e.g., \cite[Lemma 4.4]{HierarchiesPaper})
\[
\ent(\nu[t]\,\|\,\mu[t]) = \ent(\nu_0\,\|\,\mu_0) + \frac{1}{4\sigma^2}\int_0^t\E\big[|V(\mu_s,Z_s)-\alpha(s,Z)|^2\big]\,ds.
\]
Recall from Assumption \ref{assumption on Phi}(2) that $\W_1^2(\mu_0,\nu_0)\leq 2C_{\textup{T}_1}(0) \,\ent (\nu_0\,\|\,\mu_0)$. Hence,
\begin{align*}
\W_1(\nu[t],\mu[t]) &\le e^{Lt}\sqrt{2C_{\textup{T}_1}(0)\ent (\nu_0\,\|\,\mu_0)}  + e^{Lt}\int_0^t\E[|V(\mu_s,Z_s)-\alpha(s,Z)|]\,ds.
\end{align*}
Take squares and use Cauchy-Schwarz to get
\begin{align*}
\W_1^2(\nu[t],\mu[t]) &\le 4C_{\textup{T}_1}(0)e^{2Lt}\ent (\nu_0\,\|\,\mu_0)  + 2te^{2Lt}\int_0^t\E\big[|V(\mu_s,Z_s)-\alpha(s,Z)|^2\big]\,ds \\
    &\le e^{2Lt}\max(4C_{\textup{T}_1}(0),8\sigma^2 t)\ent(\nu[t]\,\|\,\mu[t]).
\end{align*}

\section{Convexity implies monotonicity}
\label{appendix: convexity}
\begin{lemma}
\label{lm: convexity implies monotonicity}
    Assume that $\Psi$ is in $\mathcal C_{bd}^6(\mathcal P_2(\R^d)\times \R^d)$ and assume it is $\lambda$-strongly displacement convex. Then for every random variable $X$ with law $\mu$ and every random variable $Y$ with law $\nu$,
    \begin{equation}
    \label{eq: monotonicity}
       \E[(Y-X)\cdot (\wgrad \Psi(\mu,X)-\wgrad \Psi(Y,\nu))]\geq \lambda \E[|Y-X|^2] 
    \end{equation}
\end{lemma}
\begin{proof}
    Let $(\Omega,\mathcal A,\P)$ be a probability space rich enough such that for every $\mu \in \mathcal P_2(\R^d)$, there exists a random variable on $\Omega$ with law $\mu$. Let $\psi: L^2(\Omega)\to \R$ be the Lions lift of $\Psi$:
    \begin{align*}
        \psi(X)=\Psi(\Law(X)).
    \end{align*}
    The differentiability assumed of $\Psi$ implies that $\psi$ is Frech\`et differentiable, and the Frech\`et derivative $D\psi(X)$ can be represented as $\wgrad \Psi( \Law(X),X)$; see the appendix of \cite{MasterEquation}, or \cite{GangboTudorascu}. Hence for every two random variables $X,Z$ in $L^2(\Omega)$,
    \begin{equation}
    \label{eq: gradient characterization}
        \psi(Z)-\psi(X)\geq \E[\wgrad \Psi( \Law(X),X) \cdot (Z-X)]+\frac \lambda 2 \E\big[|Z-X|^2\big].
    \end{equation}
    Swap the position of $Z,X$ in \eqref{eq: gradient characterization}, and add the two inequalities to obtain
    \begin{equation*}
        \E\big[(\wgrad \Psi(\Law(Z),Z)-\wgrad \Psi(\Law(X),X))\cdot (Z-X)\big] \geq \lambda \E\big[|Z-X|^2\big].
    \end{equation*}
    This is precisely the monotonicity condition of Assumption \ref{Assumption: UiT}(2) for $V=-\wgrad \Psi$.
\end{proof}

\section{Estimates for Mean Field Control}
\label{appendix: control}
In this section we prove Theorem \ref{th: chaos for control}; to do so, we need a preliminary estimate for an intermediate problem similar to \cite[Theorem 2.13]{MasterEquation}. This estimate is folklore, 
but we document it for completeness.
\begin{lemma}
\label{lm: Master equation estimate control}
Let $v:[0,T]\times \R^{dn} \to \R$ be a classical solution of the HJB equation 
\begin{align*}
    -\partial_t v_t(x)&- \Delta v_t(x)+\frac{1}{n}\sum_{i=1}^n H(x^i,n\nabla_{x^i} v_t(x))=F(m_x^n)\\
    v_T(x)&=G(m_x^n)\quad;\quad m_x^n=\frac1n\sum_{i=1}^n \delta_{x^i}
\end{align*}
and let $U_t$ be a classical solution of the Master Equation that satisfies the assumptions of Theorem \ref{th: chaos for control}:
\begin{align*}
    -\partial_t U_t(\nu)&-\int \nabla_x \cdot \wgrad U_t(\nu,x)-H(x,\wgrad U_t(\nu,x))\,d\nu(x)=F(\nu)\\
    U_T(\nu)&=G(\nu)
\end{align*}
    Consider the SDE systems
\begin{align*}
    dZ^i_t&=-\nabla_y H(Z^i_t, n\nabla_{x_i} v_t(Z_t))\,dt+\sqrt 2 dB_t^i\,\,;\,\, i=1,\dots,n;\\
    dY_t^i&=-\nabla_y H(Y_t^i,\wgrad U_t(m_t^n,Y_t^i))\,dt+\sqrt 2 \,dB_t^i\,\,;\,\,i=1,\dots,n.
\end{align*}
where $m_t^n=\frac 1n \sum_{i}\delta_{Y_t^i}$. Then,
\[\E[|Y_t^1-Z_t^1|^2]\leq \frac{C}{n^2}.\]
\end{lemma}

Once this lemma is established, the proof of Theorem \ref{th: chaos for control} is exactly the same as that of Theorem \ref{th: chaos for MFG}, just using Lemma \ref{lm: Master equation estimate control} in place of the estimates from \cite{MasterEquation}, so we do not repeat it.

\begin{proof}[Proof of Lemma \ref{appendix: control}]
Let $u_t$ be the finite dimensional projection of $U_t$, namely $u_t(x)=U_t(m_x^n)$; recalling that 
\begin{align*}
\nabla_{x_i} u_t(x)&=\frac{1}{n}\wgrad U_t(m_x^n,x^i)\\
\nabla^2_{x_i} u_t(x)&= \frac{1}{n^2}\wgrad^2 U_t(m_x^n,x^i,x^i)+\frac 1n \nabla_{x^i}\wgrad U_t(m_x^n,x^i)
\end{align*}
we get that $u_t$ solves
\begin{align*}
    -\partial_t u_t(x)&-\Delta u_t(x)+\frac{1}{n^2}\sum_{j=1}n\wgrad^2 U_t(m_x^n,x^j,x^j)-\frac{1}{n}\sum_{j=1}^n H(x^j,n \nabla_{x^j} v_t(x))=F(m_x^n)\\
    u_T(x)&=G(m_x^n)
\end{align*}
if we let $r_t= n^{-2}\sum_i \wgrad^2 U_t(m_x^n,x^i,x^i)$ we get that $|r_t|\leq C/n$ by boundedness of the derivatives Master equation given by our assumptions. Let $\mathcal U^i=\nabla_{x_i} u$ and $\mathcal V^i=\nabla_{x_i} v$; they solve the PDEs
\begin{align*}
    -\partial_t \mathcal V^i_t(x)&-\Delta \mathcal V_t^i(x)+\frac 1n \nabla_x H(x^i,n\mathcal V_t^i(x))+\sum_{j=1}^n \nabla_y H(x^j,n\mathcal V^j(x))\partial_j \mathcal V^i_t(x)=\frac 1n \wgrad F(m_x^n,x^i)\\
    \mathcal V_T^i(X)&=\frac1n \wgrad G(m_x^n,x^i)
\end{align*}
and 
\begin{align*}
    -\partial_t \mathcal U^i_t(x)&-\Delta \mathcal U_t^i+\frac 1n \nabla_x H(x^i,n\mathcal U_t^i)+\sum_{j} \nabla_y H(x^j,n\mathcal U^j_t(x))\partial_j \mathcal U^i_t(x)=\frac 1n \wgrad F(m_x^n,x^i)-q_t^i(x)\\
    \mathcal U_T^i(X)&=\frac1n \wgrad G(m_x^n,x^i)
\end{align*}
where 
\begin{align*}
    q_t^i(x)=\frac{1}{n^2} \nabla_x\wgrad^2 U_t(m_x^n,x^i,x^i)+\frac{1}{n^2} \nabla_y\wgrad^2 U_t(m_x^n,x^i,x^i)+\frac {1}{n^3}\sum_{j=1}^n\wgrad ^3 U_t(m_x^n,x^i,x^i,x^i)
\end{align*}
in particular, $|q_t^i(x)|_\infty \leq c/n^2$ given our assumptions. By Ito's formula,
\begin{align*}
    d\mathcal V_t^i(Z_t)&= -\frac 1n \Big(\wgrad F(L(Z_t),Z_t)-\nabla_x H(Z_t^i,n \mathcal V_t^i(Z_t))\Big)\,dt+\nabla \mathcal V_t^i(Z_t)\cdot dB_t
\end{align*}
and
\begin{align*}
    d\mathcal U_t^i(Z_t)&= -\frac 1n \Big(\wgrad F(L(Z_t),Z_t^i)-\nabla_x H(Z_t^i,n \mathcal U_t^i(Z_t))\Big)\,dt+q_t^i(Z_t)\,dt\\
    &+\sum_{j=1}^n \nabla_{x_j} \mathcal U_t^i(Z_t)\Big[\nabla_yH (Z_t^j,n\mathcal U_t^j(Z_t))-\nabla_y H(Z_t^j,n\mathcal V_t^j(Z_t^j))\Big]\,dt+\nabla \mathcal U_t^i(Z_t)\cdot dB_t
\end{align*}
as a result,
\begin{align*}
    0&=\E[|\mathcal V_T^i(Z_T)-\mathcal U_T^i(Z_T)|^2]=\E[|\mathcal V_s^i(Z_s)-\mathcal U_s^i(Z_s)|^2]+2\sum_{j=1}^n\int_s^T\E\big[|\nabla_{x_j} \mathcal U_t^i(Z_t)-\nabla_{x_j} \mathcal V_t^i(Z_t)|^2\big]\,dt\\
    &+\frac 1n\int_s^T\E\big[(\nabla_x H(Z_t^i,n\mathcal U_t^i(Z_t)-\nabla_x H(Z_t^i,n\mathcal V_t^i(Z_t)))\cdot (\mathcal U_t^i(Z_t)-\mathcal V_t^i(Z_t))\big]\,dt
    \\
    &+\int_s^T\E\big[q_t^i(Z_t)\cdot (\mathcal U_t^i(Z_t)-\mathcal V_t^i(Z_t))\big]\,dt\\
    &\sum_{j}\int_s^T \E\Big[(\mathcal U_t^i(Z_t)-\mathcal V_t^i(Z_t))^\top \nabla_{x_j} \mathcal U_t^i(Z_t)(\nabla_y H (Z_t^j,n\mathcal U_t^j(Z_t))-\nabla_y H(Z_t^j,n\mathcal V_t^j(Z_t^j)))\Big]\,dt
\end{align*}
Using Young's inequality, the lipschitz property of $H$ and its gradients, the fact that $|\partial_j \mathcal U_t^i| \leq C/n^2 $ for $j\neq i$, $|\partial_i \mathcal U_t^i| \leq C/n$ and exchangeability of $Z$ we get
\begin{align*}
    \E[|\mathcal V_s^i(Z_s)-\mathcal U_s^i(Z_s)|^2]&+2\sum_{j=1}^n\int_s^T\E\big[|\nabla_{x_j} \mathcal U_t^i(Z_t)-\nabla_{x_j} \mathcal V_t^i(Z_t)|^2\big]\,dt\\
    &\leq C \int_s^T \E[|\mathcal U_t^i(Z_t)-\mathcal V_t^i(Z_t)|^2]+\E[|q_t^i(Z_t)|^2]\,dt\\
\end{align*}
and thus Gronwall yields
\begin{align*}
    \E[|\nabla_{x_i} v_s(Z_s)-\nabla_{x_i} u_s(Z_s)|^2]=\E[|\mathcal V_s^i(Z_s)-\mathcal U_s^i(Z_s)|^2]\leq C \int_s^T \E[|q_t^i(Z_t)|^2]\,dt\leq \frac{C}{n^4}.
\end{align*}
and thus $\E[|n\nabla_{x^i} v_s(Z_s)-n\nabla_{x^i} u_s(Z_s)|^2]\leq C/n^2$.

Now synchronous coupling and exchangeability yield
\begin{align*}
    \E[|Y_T^i-Z_T^i|^2]&= -2\int_0^T \E[(Y_t^i-Z_t^i)\cdot (\nabla_y H(Y_t^i,n\nabla_{x_i} u_t(Y_t))-\nabla_y H(Z_t^i,n\nabla_{x_i} v_t(Z_t))]\,dt\\
    &\leq \int_0^T \E[|Y_t^i-Z_t^i|^2]+\E[|\nabla_y H(Y_t^i,n\nabla_{x_i} u_t(Y_t))-\nabla_y H(Z_t^i,n\nabla_{x_i} u_t(Z_t)|^2]\,dt\\
    &+\int_0^T\E[|\nabla_y H(Z_t^i,n\nabla_{x_i} u_t(Z_t))-\nabla_y H(Z_t^i,n\nabla_{x_i} v_t(Z_t)|^2]\,dt\\
    &\leq C\int_0^T \E[|Y_t^i-Z_t^i|^2]+L\E[|\wgrad U_t(m_t^n)-\wgrad U_t(m_{Z_t}^n)|^2]\,dt+ \frac{C}{n^2}\\
    &\leq C\int_0^T \E[|Y_t^i-Z_t^i|^2]\,dt+\frac{C}{n^2}.
\end{align*}
Gronwall yields the desired result.    
\end{proof}
\bibliographystyle{abbrv}
\bibliography{Biblio}

\end{document}